\documentclass{amsart}
\usepackage{amsmath}
\usepackage{amssymb}
\usepackage{mathrsfs}
\usepackage{stmaryrd}
\usepackage[all]{xy}
\usepackage{amsfonts}
\usepackage{fixltx2e}

\usepackage{comment}

\RequirePackage{color}
\definecolor{bwgreen}{rgb}{0.183,1,0.5}
\definecolor{bwmagenta}{rgb}{0.7,0.0,0.1}
\definecolor{bwblue}{rgb}{0.317,0.161,1}

\makeatletter \renewcommand{\everyentry@}{\vphantom{A_{[]}�{[]}}}
\makeatother

\newtheorem{theorem}{Theorem}[subsection]

\newtheorem{lemma}[theorem]{Lemma}
\newtheorem{prop}[theorem]{Proposition}
\newtheorem{co}[theorem]{Corollary}

\theoremstyle{definition}
\newtheorem{definition}[theorem]{Definition}
\newtheorem{example}[theorem]{Example}

\theoremstyle{remark}
\newtheorem{remark}[theorem]{Remark}

\numberwithin{equation}{subsection}

\def \m {\mathfrak m}

 \def \E{\mathcal E}

\def \Z {\mathbb Z}
\def \inj {\hookrightarrow }
\def \to {\rightarrow}
\def \spec \text{spec}
 \def \M{\mathfrak M}

\def \e { {\underline \epsilon}}

\def \p {\underline \pi}

\def \GL {\t{GL}}
\def \L {\mathfrak L}

\def \Hom {\textnormal{Hom}}

\def  \cris {\textnormal{cris}}
\DeclareMathOperator{\gal}{Gal}
\DeclareMathOperator{\dR}{dR}
\DeclareMathOperator{\Fil}{Fil}
\DeclareMathOperator{\MF}{MF}

\DeclareMathOperator{\ord}{ord}
\DeclareMathOperator{\rk}{rk}

\DeclareMathOperator{\diag}{diag}

\def \Q {\mathbb Q}
\def \C {\mathbb C}
\def \t {\textnormal}
\def \Z {\mathbb Z}

\def \O {\mathcal O}

\def \gs {\mathfrak S}
\def \gu {\mathfrak u}
\def \ur {\t{ur}}
\def \D {\mathcal D}

\def \gf {{\mathfrak f}}
\def \N {\mathfrak N}

\def \dR {{\textnormal{dR}}}
\def \m {\mathfrak m}

\def \v {\vee}

\def \f {\mathfrak f}

\DeclareMathOperator{\Fr}{Fr}
\def \Fil {\t{Fil}}
\def \g {\mathfrak g}

\def \gt {\mathfrak t}
\def \CM {\mathcal M}

\def \CN {\mathcal N}

\def \dr {\textnormal{dR}}

\def \R {\mathcal R}
\def \Rep { \t{Rep}_F^{F\t{-\cris}, r}(G)}
\def \sfi {\t{Mod}_{\gs_F}^{\varphi, r}}

\def \rep {\textnormal{Rep}}
\def \upi {\underline \pi }

\def \< {\left <}
\def \> {\right >}

\def \upi {{\underline{\pi}}}

\def \rig {{ {\t{B}^+_\t{rig}}}}

\def \Md {{\rm M}_d}

\def \fe {{\mathfrak e}}

\def \tmax{ {{\rm {max}}}}
\def \rig {{\rm{rig}}}

\def \gO {\mathfrak O}

\def \wgs {\widehat \gs}

\def \OEhat {\O_{\widehat \E^\ur}}

\DeclareMathOperator{\Aut}{Aut}
\DeclareMathOperator{\sep}{sep}
\DeclareMathOperator{\Gal}{Gal}

\DeclareMathOperator{\Mod}{Mod}

\newcommand*{\wt}[1]{\widetilde{#1}}
\newcommand*{\wh}[1]{\widehat{#1}}
\newcommand*{\ul}[1]{\underline{#1}}
\newcommand*{\ol}[1]{\overline{#1}}

\begin{document}

\title{On $F$-crystalline representations}

\author{Bryden Cais}
\address{Department of Mathematics, University of Arizona, Tucson, Arizona, 85721, USA.}
\email{cais@math.arizona.edu
}

\author{Tong Liu}
\address{Department of Mathematics, Purdue University, West Lafayette, IN 47907}
\email{tongliu@math.purdue.edu}
\thanks{The second author is partially supported by a Sloan fellowship and NSF grant DMS-1406926.}





\begin{abstract} We extend the theory of Kisin modules and crystalline representations to allow more general
coefficient fields and lifts of Frobenius.
In particular, for a finite and totally ramified extension $F/\Q_p$,
and an arbitrary finite extension $K/F$,
we construct a general class of infinite and totally wildly ramified
extensions $K_\infty/K$ so that the functor $V\mapsto V|_{G_{K_\infty}}$ is fully-faithfull
on the category of $F$-crystalline representations $V$.
We also establish a new classification of $F$-Barsotti-Tate groups via Kisin modules of height 1 which allows
more general lifts of Frobenius.
\end{abstract}

\maketitle




\section{Introduction}

Let $k$ be a perfect field of characteristic $p$ with ring of Witt vectors $W:=W(k)$,
write $K_0:=W[1/p]$ and let $K/K_0$ be a finite and totally ramified extension.
We fix an algebraic closure $\ol{K}$ of $K$ and set $G_K:=\Gal(\ol{K}/K)$.
The theory of {\em Kisin modules} and its variants, pioneered in \cite{kisin2},
provides a powerful set of tools for understanding Galois-stable $\Z_p$-lattices in $\Q_p$-valued
semistable $G_K$-representations, and has been a key ingredient in many recent advances
({\em e.g.} \cite{kisin4}, \cite{kisin5}, \cite{kisin}).
Throughout Kisin's theory, the non-Galois ``Kummer'' extension $K_{\infty}/K$---obtained by
adjoining to $K$ a compatible system of choices $\{\pi_n\}_{n\ge 1}$
of $p^n$-th roots of a uniformizer $\pi_0$ in $K$---plays central role.
Kisin's theory closely parallels Berger's work \cite{Ber0}, in which the cyclotomic extension
of $K$ replaces $K_{\infty}$, and can be thought of as a ``$K_{\infty}$-analogue''
of the theory of Wach modules developed by Wach \cite{Wach}, Colmez \cite{Colmez2} and Berger \cite{Ber}.
Along these lines, Kisin and Ren \cite{Kisin-Ren} generalized the theory of Wach modules
to allow the cyclotomic extension of $K$ to be replaced by an arbitrary Lubin--Tate extension.

This paper grew out of a desire to better understand the role of $K_{\infty}$
in Kisin's theory and related work, and is an attempt to realize Kisin modules
and the modules of Wach and Kisin--Ren as ``specializations" of a more general
theory.  To describe our main results, we first fix some notation.

Let $F\subseteq K$ be a subfield which is finite over $\Q_p$ with residue field $k_F$
of cardinality $q=p^s$.  Choose a power series
$$f(u) := a_1u + a_2u^2 + \cdots \in \O_F[\![u]\!]$$
with $f(u)\equiv u^q\bmod \mathfrak{m}_F$ and a uniformizer $\pi$ of $K$ with minimal polynomial
$E(u)$ over $F_0:=K_0\cdot F$.  We set $\pi_0:=\pi$ and we choose $\underline{\pi}:=\{\pi_n\}_{n\ge 1}$
with $\pi_n\in \overline{K}$ satisfying $f(\pi_n)=\pi_{n-1}$ for $n\ge 1$.
The resulting extension $K_{{\underline{\pi}}}:=\bigcup_{n\ge 0} K(\pi_n)$
(called a {\em Frobenius iterate extension} in \cite{CaisDavis}) is an infinite
and totally wildly ramified extension of $K$ which in general need not be Galois.
We set $G_{\ul{\pi}}:=\Gal(\ol{K}/K_{\ul{\pi}})$.

Define $\gs:=W[\![u]\!]$ and put $\gs_F   = \O_F \otimes _{W(k_F)} \gs = \O_{F_0}[\![u]\!]$.
We equip $\gs_F$ with the (unique continuous)
Frobenius endomorphism $\varphi$ which acts on $W(k)$ by the canonical $q$-power Witt-vector Frobenius,
acts as the identity on $\O_F$,
and sends $u$ to $f(u)$.
%
A \emph{Kisin module of $E$-height $r$} is a finite free $\gs_F$-module $\M$ endowed with $\varphi$-semilinear endomorphism $\varphi_\M : \M \to \M$ whose linearization
$1 \otimes \varphi : \varphi^* \M \to \M$
has cokernel killed by $E(u)^r$.

When $F= \Q_p$ and $f(u) = u^p$ (which we refer  to as the \emph{classical situation} in the following),
Kisin's theory \cite{kisin2} attaches to any $G_{K_{\infty}}$-stable $\Z_p$-lattice $T$
in a semistable $G_K$-representation $V$ with Hodge--Tate weights in $\{0, \dots, r\}$
a unique Kisin module $\M$ of height $r$ satisfying $T \simeq T_{\gs} (\M)$ (see \S \ref{subsec-3.2} for the definition of $T_{\gs}$).  Using this association, Kisin proves that
the restriction functor $V \to V|_{G_{K_{\infty}}}$ is fully faithful when restricted to the category of  crystalline representations, and that the category of Barsotti--Tate groups over $\O_K$
is {\em anti-equivalent} to the category of Kisin modules of height 1.

In this paper, we extend much of the framework of \cite{kisin2} to allow general $F$ and $f(u)$,
though for simplicity we will restrict ourselves to the case that $q=p$, or equivalently that $F/\Q_p$ is totally ramified.
When we extend our coefficients from $\Q_p$ to $F$, we must further restrict ourselves to studying
\emph{$F$-crystalline representations}, which are defined following  (\cite{Kisin-Ren}):  Let $V$ be a finite dimensional $F$-vector
 space with continuous $F$-linear action of $G_K$. If $V$ is crystalline (when viewed as a $\Q_p$-representation) then $D_{\dr}(V)$ is naturally an $F \otimes_{\Q_p}K$-module and one has a decomposition
 $D_\dr(V) = \prod_{\m} D_\dr(V)_\m$, with $\m$ running over the maximal ideals of $F \otimes _{\Q_p}K$.
We say that $V$ is \emph{$F$-crystalline} if $\Fil ^0 D_\dR (V) = D_\dr(V) $ and
 $\Fil ^1 D_\dr(V)_\m= 0 $ unless $\m$ corresponds to the canonical inclusion $F \subset K$.

\begin{theorem}\label{thm-intro-1}Let $V$ be an $F$-crystalline representation with Hodge-Tate weights in $\{0, \dots , r\}$ and $T\subset V$ a $G_{\ul{\pi}}$-stable $\O_F$-lattice. Then there exists a Kisin module $\M $ of $E(u)$-height $r$ satisfying $T_{\gs} (\M)\simeq T$.
\end{theorem}

Writing $v_F$ for the normalized valuation of $\ol{K}$ with $v_F(F)=\Z$,
apart from the classical situation $f(u)=u^p$ of Kisin, the above theorem is also known when
$v_F(a_1)= 1$, which corresponds to the Lubin--Tate cases covered by the work of \cite{Kisin-Ren}.
An important point of our formalism is that $\M$ may in general {\em not} be unique for a fixed lattice $T$:
our general construction produces as special cases the $\varphi$-modules over $\gs_F$
which occur in the theory of Wach modules and its generalizations \cite{Kisin-Ren},
so without the additional action of a Lubin--Tate group $\Gamma$, one indeed {\em does not expect}
these Kisin modules to be uniquely determined; ({\em cf.}~Example \ref{ex-cyclotomic2}).
This is of course quite different from the classical situation.
Nonetheless, we prove the following version of Kisin's ``full-faithfulness" result.
Writing $\Rep$ for the category of $F$-crystalline representations with Hodge-Tate weights in $\{0, \dots , r\}$ and $\rep_F(G_{\ul{\pi}})$ for the category of $F$-linear representations of $G_{\ul{\pi}}$, we prove:

\begin{theorem}\label{thm-intro-2}
Assume that $\varphi ^n (f(u)/u)$ is not a power of $E(u)$ for all $n \geq 0$ and that $v_F(a_1)>r$,
with $a_1=f'(0)$ the linear coefficient of $f(u)$.
Then the restriction functor $\Rep \rightsquigarrow \t{Rep}_F (G_{\ul{\pi}})$ induced by $V \mapsto V|_{G_{\ul{\pi}}}$ is fully faithfull.
\end{theorem}

Although Beilinson and Tavares Ribeiro \cite{Beilinson} have given an almost elementary proof of Theorem
\ref{thm-intro-2}
in the classical situation $F=\Q_p$ and $f(u)=u^p$, their argument relies crucially
on an explicit description of the Galois closure of $K_{\infty}/K$.
For more general $F$ and $f$, we have no idea what the Galois closure of $K_{\ul{\pi}}/K$
looks like, and describing it in any explicit way seems to be rather difficult in general.

It is natural to ask when two different choices $f$ and $f'$ of $p$-power Frobenius lifts
and corresponding compatible sequences $\ul{\pi}=\{\pi_n\}_n$ and $\ul{\pi}'=\{\pi_n'\}$
in $\ol{K}$ yield the {\em same} subfield $K_{\ul{\pi}}=K_{\ul{\pi}'}$ of $\ol{K}$.
We prove that this is {\em rare} in the following precise sense: if $K_{\ul{\pi}}=K_{\ul{\pi}'}$,
then the lowest degree terms of $f$ and $f'$ coincide, up to multiplication by a unit in $\O_F$;
see Proposition \ref{fdetbyK}.  It follows that there are {\em infinitely many} distinct
$K_{\ul{\pi}}$ for which Theorem \ref{thm-intro-2} applies.
We also remark that any Frobenius--iterate extension $K_{\ul{\pi}}$
as above is an infinite and totally wildly ramified {\em strictly APF} extension
in the sense of Wintenberger \cite{win}.
We therefore think of Theorem \ref{thm-intro-2} as confirmation of the philosophy
that ``crystalline $p$-adic representations are the $p$-adic analogue of unramified $\ell$-adic representations,''
since Theorem \ref{thm-intro-2} is obvious if ``crystalline" is replaced with ``unramified" throughout
(or equivalently in the special case $r=0$).
More generally, given $F$ and $r\ge 0$, it is natural to ask for a characterization of
{\em all} infinite and totally wildly ramified strictly APF extensions $L/K$ for which
restriction of $F$-crystalline representations of $G_K$ with Hodge--Tate weights in $\{0,\dots,r\}$
to $G_L$ is fully--faithful.  We believe that there should be a deep and rather general phenomenon
which deserves further study.


While the condition that $v_F(a_1)>r$ is really essential
in Theorem \ref{thm-intro-2} (see Example \ref{ex-counter}), we suspect the conclusion is still valid if we remove the assumption that $\varphi ^n (f(u)/u)$ is not a power of $E(u)$ for all $n \geq 0$.
However, we have only successfully removed this assumption when $r=1$, thus generalizing Kisin's classification
of Barsotti--Tate groups:

\begin{theorem}\label{thm-intro-3}Assume $v_F(a_1)>1$. Then the category of Kisin modules of height $1$
is equivalent to the category of $F$-Barsotti-Tate groups over $\O_K$.
\end{theorem}
Here, an \emph{$F$-Barsotti-Tate group} is a Bartotti--Tate grouop $H$ over $\O_K$ with the property
that the $p$-adic Tate module $V_p (H) = \Q_p \otimes_{\Z_p} T_p (H)$ is an $F$-crystalline representation.
We note that when $F=\Q_p$, Theorem \ref{thm-intro-3} is proved (by different methods)
in \cite{CaisLau}.

Besides providing a natural generalization of Kisin's work and its variants
as well as a deeper understanding of some of the finer properties of crystalline $p$-adic Galois
representations, we expect that our theory will have applications to the study of
potentially Barsotti--Tate representations.  More precisely,
suppose that $T$ is a finite free $\O_F$-linear representation of $G_K$ with the property that
$T|_{G_{K'}}$ is Barsotti-Tate for some finite extension $K'/K$. If $K'/K$ is not tamely ramified then it is well-known that it is in general difficult to construct ``descent data'' for the Kisin module $\M$ associated to $T|_{G_{K'}}$
in order to study $T$ (see the involved computations in \cite{bcdt}).
However, suppose that  we can select $f(u)$ and $\pi_0$ such that
$K' \subseteq  K(\pi_n)$ for some $n$. Then, as in the theory of Kisin--Ren \cite{Kisin-Ren} (see also 
\cite{Ber3}), we expect
the appropriate descent data on $\M$ to be {\em much easier} to construct in this ``adapted" situation,
and we hope this idea can be used to study the reduction of $T$.

Now let us sketch the ideas involved in proving the above theorems and outline the organization of this paper.
For any $\Z_p$-algebra $A$, we set $A_F:=A \otimes_{\Z_p} \O_F$.   In order to connect $\gs_F$ to Galois representations, we must first embed $\gs _F$ as a Frobenius-stable subring of $W(R)_F $, which
we do \S2.1 following \cite{CaisDavis}. In the following subsection, we collect some useful properties of
this embedding and study some ``big rings" inside $B^+_{\cris, F}$. Contrary to the classical situation, the Galois closure of $K_{\ul{\pi}}$ appears in general to be rather mysterious.
Nonetheless, in \S\ref{GactsOnu} we are able to establish some basic results on the $G_K$-conjugates
of $u\in \gs_F \subseteq W(R)_F$ which are just barely sufficient for the development of our theory.
Following Fontaine \cite{fo4} and making use of the main result of \cite{CaisDavis},
in \S\ref{sec-3} we establish a classification of $G_{\ul{\pi}}$-representations via
\'etale $\varphi$-modules and Kisin modules. 
In the end of \S\ref{sec-3}, we apply these considerations to prove that the functor $T_{\gs}$ is fully faithful under the assumption that $\varphi ^n ({f(u)}/{u})$ is not a power of $E(u)$ for any $n$.

The technical heart of this paper is \S\ref{sec-4}.
In \S\ref{FcrysReps}, we define $F$-crystalline representations and attach to each $F$-crystalline representation $V$ a filtered $\varphi$-module $D_{\cris , F}(V)$ (we warn the reader that the filtration of $D_{\cris, F} (V)$ is slightly different from that of $D_\cris (V)$). Following \cite{kisin2}, in \S \ref{subsec-3.1} we
then associate to $D= D_\cris(V)$ a $\varphi$-module $\CM(D)$ over $\gO$ (here we use $\gO$ for
the analogue of $\O$---the ring of rigid-analytic functions on the open unit disk---in Kisin's work).
A shortcoming in our situation is that we do not in general know how to define a reasonable differential
operator
$N_\nabla $, even at the level of the ring $\gO$. Consequently, our $\CM(D)$ only has a Frobenius structure,
in contrast to the classical (and Lubin--Tate) situation in which $\CM (D)$ is also equipped with a natural
$N_\nabla$-structure. Without such an $N_\nabla$-structure, there is no way to follow Kisin's (or Berger's) original strategy to prove that the scalar extension of $\CM(D)$ to the Robba ring
 is pure of slope zero, which is key to showing that there exists a Kisin module $\M$ such that $\gO \otimes_{\gs_F} \M \simeq \CM(D)$. We bypass this difficulty
by appealing to the fact that $\CM(D)$ is known to be pure of slope zero in the classical situation of Kisin
as follows: letting a superscript of ``$c$" denote the data in the classical situation and using the fact that both $\CM (D)$ and $\CM^c (D)$ come from the same $D$, we prove that $\wt B_\alpha \otimes_\gO \CM (D) \simeq \wt B_\alpha  \otimes_{\gO^c} \CM ^c(D)$ as $\varphi$-modules for a certain period ring $\wt B_{\alpha}$
that contains the ring $\wt B ^+_{\rig,F}$.
 It turns out that this isomorphism can be descended to $\wt B ^+_{\rig,F}$. Since Kedlaya's theory of the slope filtration is unaffected by base change from the Robba ring to
$\wt B^+_{\rig,F}$, it follows that $\CM(D)$ is of pure slope $0$ as this is the case for
 $\CM^c (D)$ thanks to \cite{kisin2}. With this crucial fact in established, we are then able to prove Theorem \ref{thm-intro-1} along the same lines as \cite{kisin2}.
If our modules came equipped with a natural $N_{\nabla}$-structure,
the full faithfulness of the functor $V \mapsto V|_{G_{\ul{\pi}}}$ would follow easily from the full faithfulness of $T_{\gs}$. But without such a structure, we must instead rely
heavily on the existence of  a \emph{unique} $\varphi$-equivariant section $\xi: D(\M) \to \gO_\alpha \otimes\varphi ^* \M$ to the projection $\varphi ^* \M  \twoheadrightarrow \varphi ^*\M/ u \varphi ^* \M $, where $D(\M) = (\varphi ^*\M/ u \varphi ^* \M) [1/p]$. The hypothesis $v_F(a_1)>r$ of Theorem \ref{thm-intro-2} guarantees the existence and uniqueness of such a section $\xi$. With these preparations, we finally prove Theorem \ref{thm-intro-2} in  \S \ref{subsec-mainresults}.

 In \S 5, we establish Theorem \ref{thm-intro-3}: the equivalence between the category of Kisin modules of height 1 and the category of $F$-Barsotti-Tate groups over $\O_K$. Here we adapt the ideas of \cite{liu-BT} to prove that the functor $\M \mapsto T_{\gs}(\M)$ is an equivalence between the category of Kisin module of height 1 and the category of $G_K$-stable $\O_F$-lattices in $F$-crystalline representations with Hodge-Tate weight in $\{0, 1\}$. The key difficulty is to extend the $G_{\ul{\pi}}$-action on $T_{\gs}(\M)$ to a $G _K$-action which
gives $T_{\gs}(\M)[1/p]$ the structure of an $F$-crystalline representation. In the classical situation, this is done using the (unique) monodromy operator $N$ on $S \otimes _\gs \varphi ^* \M$ (see \S2.2 in \cite{liu-BT}). Here again, we are able to sidestep the existence of a monodromy operator to construct a (unique) $G_K$-action on $W(R)_F \otimes_{\gs_F} \M$ which is compatible with the additional structures (see Lemma \ref{lem-extend Ginfty action}), and this is enough for us to extend the given $G_{\ul{\pi}}$-action to a $G_K$-action on $T_{\gs}(\M)$. As this paper establishes analogues of many of the results of \cite{kisin2} in our more general context, it is natural ask
to what extent the entire theory of \cite{kisin2} can be developed in this setting.  To that end,
we list several interesting (some quite promising) questions for this program in the last section.




\subsection*{Acknowledgements:} It is pleasure to thank Laurent Berger, Kiran Kedlaya and Ruochuan Liu
for very helpful conversations and correspondence.

\subsection*{Notation}Throughout this paper, we reserve $\varphi$ for the Frobenius operator,
adding appropriate subscripts as needed for clarity:
for example, $\varphi_\M$ denotes the Frobenius map on $\M$. We will always drop these subscripts when
there is no danger of confusion.
Let $S$ be a ring endowed with a Frobenius lift $\varphi_S$ and $M$ an $S$-module.
We always write $\varphi ^*M := S \otimes_{ \varphi_S, S } M$. Note that if 
$\varphi_M: M \to M$ is a $\varphi _S$-semilinear endomorphism, 
then $1 \otimes \varphi_M : \varphi ^* M \to M$ is an $S$-linear map.
We reserve $f(u) = u ^p + a_{p -1} u + \cdots+ a_1 u \equiv u ^p \mod \mathfrak{m}_F$ for the polynomial over $\O_F$
giving our Frobenius lift $\varphi(u):=f(u)$ as in the introduction
For any discretely valued subfield $E\subseteq \ol{K}$,
we write $v_E$ for the normalized $p$-adic valuation of $\ol{K}$
with $v_E(E)=\Z$, and for convenience will simply write $v:=v_{\Q_p}$.
If $A$ is a $\Z_p$-module, we set $A_F:= A \otimes_{\Z_p} \O_F$ and
$A[1/p] := A \otimes_{\Z_p}\Q_p$. For simplicity, we put $G = G_K := \gal (\overline K / K)$
and $G_{\ul{\pi}}:=\Gal(\ol{K}/K_{\ul{\pi}})$.
Finally, we will write $\Md (S)$ for the ring of $d \times d$-matrices with entries in $S$ and $I_d$ for
the $d \times d $-identity matrix.

\section{Period rings}\label{sec-2}

In this section, we introduce and study the various ``period rings" which will play a central role
in the development of our theory.

As in the introduction, we fix a perfect field $k$ of characteristic $p$ with ring of Witt vectors $W:=W(k)$,
as well as a finite and totally ramified extension $K$ of $K_0:=W[1/p]$.
Let  $F$ be a subfield of $K$, which is finite and totally ramified over $\Q_p$, and
put $F_0 := K_0 F \subset K.$
Choose uniformizers $\pi$ of $\O_K$ and $\varpi$ of $\O_F$, and let $E(u)\in \O_{F_0}[u]$ be the minimal polynomial
of $\pi$ over $F_0$.  We set $e:= [K: K_0]$, and put $e_0 := [K:F_0]$ and $e_F := [F : \Q_p]$.
Fix a polynomial $f(u) = u ^p + a_{p-1} u ^{p-1} + \dots + a_1 u \in \O_F [u]$ satisfying $f(u) \equiv u ^p \mod \varpi$, and recursively choose $\pi_n\in \ol{K}$ with $f(\pi_n) = \pi_{n-1}$ for $n\ge 1$ where $\pi_0 := \pi$.
Set $K_{\ul{\pi}} := \bigcup_{n \geq 0 } K(\pi_n)$ and $G_{\ul{\pi}} := \gal (\overline K / K_{\ul{\pi}})$, and
recall that for convenience we write $G = G_K := \gal (\overline K /K)$.

Recall that  $\gs  = W [\![u]\!]$, and that we equip the scalar extension $\gs_F$
with the semilinear Frobenius endomorphism $\varphi : \gs_F \to \gs_F$ which acts on $W$
as the unique lift of the $p$-power Frobenius map on $k$, acts trivially on $\O_F$,
and sends $u$ to $f(u)$.  The first step in our classification of $F$-crystalline
$G_K$-representations by Kisin modules over $\gs_F$ is to realize this ring as a Frobenius stable subring of
$W(R)_F$, which we do in the following subsection.

 \subsection{$\gs_F$ as a subring of $W(R)_F$}
As usual, we put $R:=\varprojlim \limits_{x\to x^p} \O_{\overline K}/ (p)$, equipped with its natural coordinate-wise action of $G$.
It is well-known that the natural reduction map
$$\varprojlim\limits_{x\to x^p} \O_{\overline K}/ (p)\rightarrow \varprojlim\limits_{x\to x^p} \O_{\overline K}/(\varpi)$$
is an isomorphism, so $\{\pi_n\}_{n \geq 0}$ defines an element $\upi \in R$.
Furthermore, writing $\C_K$ for the completion of $\ol{K}$, reduction modulo $p$ yields a multiplicative bijection $\varprojlim_{x\to x^p} \O_{\C_K}\simeq R$, and for any $x \in R$ we write $(x^{(n)})_{n \geq 0}$ for the $p$-power compatible sequence in
$\varprojlim_{x\to x^p} \O_{\C_K}$ corresponding to $x$ under this identification.
We write $[x] \in W(R)$ for the Techm\"uller lift of $x\in R$, and denote by
$\theta: W(R) \to \O_{\C_K}$ the unique lift of the projection $R\twoheadrightarrow \O_{\C_K}/(p)$
which sends $\sum_n p^n [x_n]$ to $\sum_n p^n x^{(0)}$. By definition, $B^+_\dr$ is the $\t{Ker}(\theta)$-completion of $W(R) [1 /p]$, so $\theta$ naturally extends to $B^+_\dR$. For any subring $B \subset B^+_\dR$, we define $\Fil ^i B:  = (\t{Ker}\theta)^i \cap B.$

There is a canonical section $\ol{K}\hookrightarrow B_{\dR}^+$, so we may view $F$ as a subring of $B_{\dR}^+$,
and in this way we obtain embeddings $W(R)_F \inj B^+_{\cris, F} \inj B^+_\dR$. Define $\theta_F := \theta|_{W(R)_F}$. 	One checks that $W(R)_F$ is $\varpi$-adically complete and that every element of $W(R)_F$ has the form
	$\sum_{n\ge 0} [a_n]\varpi^n$ with $a_n\in R$.  The map $\theta_F$ then carries
	$\sum_{n \geq 0} [a_n]\varpi^n$ to $\sum_{n\ge 0} a_n^{(0)}\varpi^n\in \O_{\C_K}$ (see pg. 11 and Prop. 3.1 of \cite{CaisDavis}).



\begin{lemma}\label{lem-basic}
There is a unique set-theoretic section $\{\cdot\}_f: R\rightarrow W(R)_F$ to the reduction modulo $\varpi$ map
which satisfies $\varphi(\{x\}_f) = f(\{x\}_f)$ for all $x\in R$.
\end{lemma}

\begin{proof}
	This is\footnote{In the version of Colmez's article available from his website,
	it is Lemme 8.3.} \cite[Lemme 9.3]{Colmez}.  Explicitly,
	using the fact that $f(u)\equiv u^p\bmod \varpi,$
	one checks that the endomorphism
	$f\circ \varphi^{-1}$ of $W(R)_F$ is a $\varpi$-adic contraction, so that for any lift $\wt{x}\in W(R)$
	of $x\in R$, the limit
	$$\{x\}_f:=\lim_{n\rightarrow\infty} (f\circ \varphi^{-1})^{(n)}(\wt{x})$$
	exists in $W(R)_F$ and is the unique fixed point of $f\circ\varphi^{-1}$,
	which uniquely characterizes it independent of our choice of $\wt{x}$.
\end{proof}

It follows immediately from Lemma \ref{lem-basic} that there is a unique
continuous $\O_F$-algebra embedding $\iota: \gs_F\hookrightarrow W(R)_F$
with $\iota(u):=\{\upi\}_f$.  We henceforth identify $\gs_F$
with a $\varphi$-stable $\O_F$-subalgebra of $W(R)_F$ via $\iota$
on which we have $\varphi(u)=f(u)$.

\begin{example}[Cyclotomic case]\label{ex-cyclotomic}Let $\{\zeta_{p ^n}\}_{n\ge 0}$ be a compatible system of
primitive $p^n$-th roots of unity. Let $K = \Q_p (\zeta_p)$, and put $\pi= \zeta_p -1$ and $f(u) = (u +1)^p -1 \in \Q_p [u]$. Choosing $\pi_n = \zeta_{p ^{n+1}} -1$, we obtain $K_{\ul{\pi}}:= \bigcup_{n \geq 1} \Q_p (\zeta_{p ^n})$. It is obvious that $\e_1 := (\zeta_{p ^n})_{n \geq 1 } \in R$. In this case, $\iota(u) = [\e_1]-1 \in W(R)$.

\end{example}

Recall that $R$ has the structure of a valuation ring via $v_R (x):= v (x^{(0)})$, where
$v$ is the normalized $p$-adic valuation of $\C_K$ with $v(\Z_p)=\Z$.
\begin{lemma}\label{lem-generater}
	We have $\theta_F(u)=\pi$ and $E(u)$ is a generator of $\t{Ker}(\theta_F)= \Fil ^1 W(R)_F$.
\end{lemma}


\begin{proof}
	The first assertion is \cite[Lemme 9.3]{Colmez}.
	To compute $\theta_F(\{\upi\}_f)$, we first choose $[\upi]$ as our lift
	of $\upi$ to $W(R)$, and compute
	\begin{equation*}
		\theta_F(\{\upi\}_f)=\theta_F\left(\lim_{n\rightarrow\infty} f^{(n)}\varphi^{(-n)}([\upi])\right)
		=\lim_{n\rightarrow\infty} f^{(n)}\theta_F([\upi^{p^{-n}}])
		=\lim_{n\rightarrow\infty} f^{(n)}(\upi^{(n)})
	\end{equation*}
	But $\upi^{(n)}\equiv \pi_n\bmod \varpi$, so
	\begin{equation*}
		f^{(n)}(\upi^{(n)})\equiv f^{(n)}(\pi_n)\equiv \pi \bmod \varpi^{n+1},
	\end{equation*}
	which gives the claim.  Now certainly $\theta_F(E(u))=E(\pi)=0$, so $E(u) \in \Fil ^1 W(R)_F$. Since $E(u)\equiv\upi ^{e_0}\bmod \varpi$, we conclude that
$$v_R(E(u) \bmod \varpi )= e_0 v_R (\upi)= e_0 v (\pi)= v(\varpi),$$
and it follows from \cite[Prop.~7.3]{Colmez} that $E(u)$ is a generator of $\t{Ker}(\theta_F) = \Fil ^1W(R) _F$.
\end{proof}



Now let us recall the construction of $B^+_\tmax$ and $\wt B^+_\rig$ from Berger's paper \cite{Ber0}. Let $\xi$ be a generator of $\Fil ^1 W(R)$. By definition,
$$B^+_\tmax:= \left  \{\sum _{n \geq 0} a_n \frac{\xi ^n }{p^n}  \in B_\dR^+\ | \ a_n \in W(R)[1/p], \  a_n \to 0 { \rm {\ when\  }} n \to +\infty \right \}.$$
and $\wt B^+_\rig := \bigcap_{n \geq 1} \varphi ^n (B^+_\tmax) .$


Write $\gu: = [\upi]$. The discussion before Proposition 7.14 in \cite{Colmez} shows:

\begin{lemma}\label{lem-Bmaxf}
\begin{eqnarray*}
B^+_{\tmax, F}&=&  \left  \{\sum _{n \geq 0} a_n \frac{E(u) ^n }{\varpi^n}  \in B_\dR^+\  | \  a_n \in W(R)_F[1/ p], \  a_n \to 0 { \rm {\ when\  }} n \to +\infty  \right \}\\  & = & \left  \{\sum _{n \geq 0} a_n \frac{\gu  ^{e_0 n} }{\varpi^n}  \in B_\dR^+ \ | \ a_n \in W(R)_F[1/ p ], \  a_n \to 0 { \rm {\ when\  }} n \to +\infty  \right \}.
\end{eqnarray*}
\end{lemma}

We can now prove the following result, which will be important in \S \ref{subsec-comparions}:
\begin{lemma}\label{lem-inBrig} Let $x \in B^+_{\tmax,F}$, and suppose that $x E(u)^r = \varphi ^m (y)$
for some $ y \in B^+_{\tmax,F}$. Then $x = \varphi ^m (y')$ with $y'  \in B^+_{\tmax,F}$.
\end{lemma}

\begin{proof} By Lemma \ref{lem-Bmaxf}, we may write $y = \sum_n  b_n\frac{\gu ^{e_0 n}}{\varpi ^n}$ with $b_n \in W(R)_F[1/p]$ converging to $0$. Write $E(u) = \gu ^{e_0} + \varpi z$ with $z \in W(R)_F$. We then have
$$\varphi ^m (y) =\sum _{n= 0}^\infty \varphi ^m (b_n) \frac{\gu ^{e_0p^m n}}{\varpi ^n} = \sum _{n=0} ^\infty \varphi ^m ( b_n)  \frac{(E(u)- \varpi z)^{ p^m n }}{\varpi^n} = \sum _{n=0} ^\infty c_n \frac{(E(u))^{ p^m n }}{\varpi^n} $$ with $c_n \in W(R)_F[1/p]$  converging to $0$. By Lemma \ref{lem-generater}, $E(u)$ is a generator of $\Fil ^1 W(R)_F$,
so definining $s:=1+\max\{ n | p^mn < r \}$, it follows that
$\sum \limits_{n = 0}^{s-1} c_n \frac{(E(u))^{ p^m n }}{\varpi^n}$ is divisible by $E(u)^r$ in $W(R)_F[1/p]$.
Writing $\sum \limits_{n = 0}^{s-1} c_n \frac{(E(u))^{ p^m n }}{\varpi^n} = E(u)^r x_0$ with $x_0 \in W(R) [1/ p]$
and, without loss of generality, replacing $x$ by $x-x_0$ we have
$$x = \sum_{n = s}^\infty c_n  \frac{(E(u))^{ p^m n-r }}{\varpi^n}= \sum_{n = s}^\infty d_{n+ s }  \frac{(E(u))^{ p^m {(n-s)} }}{\varpi^{n-s}}=
\sum_{n = 0}^\infty d_n  \frac{(E(u))^{ p^m n }}{\varpi^n} $$ with $d_{n+s} =c_n \frac{E(u)^{p^m s -r}}{\varpi ^s}. $ Using again the equality $E(u) = \gu ^{e_0} + \varpi z $, we then obtain $x = \sum \limits _{n = 0}^\infty e_n \frac{\gu ^{e_0 p^m n}}{\varpi ^n}$ with $e_n \in W(R)_F[1/p]$ converging to $0$. We now have
$x = \varphi ^m (y ')$ for $y' :=\sum \limits _{n = 0}^\infty f_n \frac{\gu ^{e_0 n}}{\varpi ^n}$ with $f_n = \varphi ^{-m} (e_n)$. As $f_n \in W(R)_F [ 1/ p ]$ converges to $0$, we conclude that $y '  \in B^+_{\tmax, F}$, as desired.
\end{proof}

\subsection{Some subrings of $B^+_{\cris, F}$} For a subinterval
$I\subset [0,1)$, we write $\gO_I$ for the subring of
$F_0 (\!(u)\!)$ consisting of those laurent series which converge for all $x \in \C_K$ with
$ |x| \in I$, and we simply write $\gO= \gO_{[0, 1)}$. Let $\wt B_\alpha : = W(R)_F [\![\frac{E(u)^p}{\varpi}]\!] [1/p] \subset B^+_{\cris, F}$. It is easy to see that $\Fil^n \wt B_\alpha = E(u)^n \wt B_\alpha$. 

\begin{lemma}\label{lem-ringresults}There are canonical inclusions of rings $\gO \subset \wt B^+_{\rig, F} \subset \wt B_\alpha$.
\end{lemma}
\begin{proof} We first show that $\gO \subset \wt B^+_{\rig, F}$.
For any $h(u) = \sum\limits_{n = 0} ^\infty a_n u ^n \in \gO$, we have to show that $h_m (u) = \sum\limits_{n = 0} ^\infty \varphi ^{-m}(a_nu ^n)$ is in $B^+_{\tmax, F}$ for all $m\ge 0$. Writing $u = \gu + \varpi z$ with $\gu = [\upi]$ and $z \in W(R)_F$, we have $\varphi ^{-m}(u)= \gu^{p^{-m}} + \varpi z^{(m)}$ with $z^{(m)} = \varphi ^{-m} (z)\in W(R)_F$. Setting $a^{(m)}_n := \varphi ^{-m} (a_n) \in F_0$, we then have
$$ h_m (u ) = h (\gu^{\frac{1}{p ^m}} + \varpi z ^{(m)}) = \sum_{k = 0} ^\infty \frac {h ^{(k)} (\gu ^{\frac {1}{p ^m }} )}{k !} (\varpi z^{(m)})^k , $$
for $h ^{(k)}$ the $k$-th derivative of $\wt h (X) := \sum\limits_{n =0 } ^\infty a^{(m)}_n X^n $. Therefore,
$$h_m (u ) = \sum_{n = 0} ^\infty \left ( \sum_{k = 0} ^\infty {k+n \choose k} a^{(m)}_{n+k} (\varpi z^{(m)})^k  \right ) \gu ^{\frac {n}{p^m}}. $$
Since $h(u) \in \gO_{[0, 1 )}$, we have $\lim \limits_{n \to \infty } |a^{(m)}_n| r ^n = 0$ for any $r < 1 $. It follows  that the inner sum $\sum\limits_{k = 0} ^\infty {k+n \choose k} a^{(m)}_{n+k} (\varpi z^{(m)})^k$ converges to $b_n \in W(R)_F [1/ p]$, with $ b_n \varpi ^{\frac{n}{e_0 p ^m}}$ converging to $0$ in $W(R)_F$. We may therefore write
$$h_m (u) = \sum_{n = 0} ^\infty b_n \gu ^{\frac{n} {p^m}} =  \sum_{n = 0} ^\infty b_n \varpi ^{\frac{n}{e_0 p^m}} \frac{(\gu ^{e_0}) ^{\frac{n}{e_0 p^m}}}{\varpi^{\frac{n}{e_0 p^m}}},$$
and Lemma \ref{lem-Bmaxf} implies that $h_m(u) \in B^+_{\tmax, F}$, which completes the proof that $\gO \subset \wt B^+_{\rig , F}$.

To show that $\wt B^+ _{\rig,F} \subset  \wt B_\alpha$, we first observe that
\begin{equation}
	\wt B_\alpha = W(R)_F [\![\frac{u ^{e_0p}}{\varpi }]\!][1/ p ]= W(R)_F [\![\frac{\gu ^{e_0p}}{\varpi }]\!][1/ p ]  .\label{BaRewrite}
\end{equation}
 For any $x \in \wt B^+_{\rig, F} $, we may write $x = \varphi (y)$ with $y = \sum \limits_{n= 0} ^\infty a_n \frac{{\gu} ^{e_0 n} }{\varpi ^n} \in B^+_{\tmax , F}$, and we see that $x = \sum \limits_{n= 0} ^\infty \varphi (a_n) \frac {\gu ^{e_0 pn}}{\varpi ^n}$ lies in $\wt B_\alpha$ by (\ref{BaRewrite}), as desired.
\end{proof}

Finally let us record the following technical lemma: recall that our Frobenius lift on $\gs_F$
is determined by $\varphi(u):=f(u)$, with $f(u) = u ^p + a_{p-1}u ^{p-1} + \cdots + a_1 u.$
We define $\gO_\alpha:= \gs_F [\![\frac{u ^{e_0p}}{\varpi}]\!] [1/ p] \subset \wt B_\alpha$.

\begin{lemma}\label{lem-f(n)} Suppose that $\varpi ^{r+1} | a_1$ in $\O_F$. Then there exists $h_i ^{(n)}(u) \in \O_{F}[u]$ such that
 $$f^{(n)}(u)= \sum \limits _{i = 0} ^n h ^{(n)}_{n-i} (u) u ^{2^{n-i}} \varpi^{(r+1)i}. $$
In particular, $\varphi^n(u)/\varpi ^{rn}$ converges to $0$ in $\gO_\alpha$.

\end{lemma}
\begin{proof}We proceed by induction on $m = n$. When $m = 1$, we may write
\begin{equation}
	f(u) = u ^p + a_{p-1} u ^{p _1} + \cdots + a_1 u= u ^2 h(u ) + b_0 \varpi ^{r+1} u\quad\text{with}\quad
	b_0 \in \O_{F}.\label{breakup}
\end{equation}
Supposing that the assertion holds for $m=n$ and using (\ref{breakup}) we compute
\begin{eqnarray*}
f^{(n+1)}(u)&= & \sum  _{i = 0} ^n h ^{(n)}_{n-i} ( f(u) ) f(u) ^{2^{n-i}} \varpi^{(r+1)i}\\
& = &  \sum_{i = 0} ^n h ^{(n)}_{n-i} ( f(u) ) ( u ^2 h(u) + b_0 \varpi ^{r+1} u ) ^{2^{n-i}} \varpi ^{(r+1)i}\\
 &= &  \sum_{i = 0} ^ n h ^{(n)}_{n-i} ( f(u) )  \left ( \sum_{k = 0} ^{2 ^{n-i}} { 2^{n-i} \choose k }  (u ^2 h(u)) ^{2^{n-i }-k } (b_0 \varpi ^{r+1} u)^k   \right ) \varpi ^{(r+1)i}\\
& = & \sum_{i = 0} ^n  \sum _{k = 0} ^{2^{n-i}}\left (h ^{(n)}_{n-i} ( f(u) )  { 2^{n-i} \choose k } h (u)^{2 ^{n-i}-k} b_0 ^{k }\right )  u ^{2^{n+1 -i}-k} \varpi ^{(r+1)(i+k )}
\end{eqnarray*}
To complete the inductive step, it then suffices to show that whenever $i + k \leq n +1$,
we have $2 ^{n+1 -i} -k \geq 2 ^{n+1-i -k}$. Equivalently, putting $j:=n+1-i-k$,
we must show that $2^{j+k}-k \geq 2^j$ for all $j\ge 0$, which holds
as $2 ^k \geq k+1$ for all $k\ge 0$.
\end{proof}

\subsection{The action of $G$ on $u$}\label{GactsOnu} In this subsection, we study the action of $G$
on the element $u\in W(R)_F$ corresponding to our choice of $f$-compatible
sequence $\{\pi_n\}_n$ in $\ol{K}$ and our Frobenius lift determined by $f$.
From the very construction of the embedding $\gs_F \inj W(R)_F$ in Lemma \ref{lem-basic},
the action of $G_{\ul{\pi}}$ on $u$ is trivial. However, for arbitrary $g \in G \setminus G_{\ul{\pi}}$,
in contrast to the classical case we  know almost nothing about the shape of $g(u)$; {\em cf.} the discussion in \S \ref{subsec-3.0}. Fortunately, we are nonetheless able to prove the following facts,
which are sufficient for our applications.


Define $$I_F ^{[1]}: = \{ x\in W(R)_F | \varphi ^n (x) \in\Fil ^1 W(R)_F, \forall n \geq 0\}.$$
Recall that $e_F := [F: \Q_p]$, and for $x \in W(R)_F$ write $\bar x := x \mod \varpi \in R$.
Thanks to Example \ref{ex-dim1}, there exists $\gt_F \in W(R)_F$ satisfying $\varphi (\gt_F) = E(u) \gt_F$.
As $E(u) \in \Fil ^1 W(R)_F$, it is easy to see that $\varphi (\gt_F) \in I ^{[1]}W(R)_F$, and since
${\bar{\gt}}_F^p = u ^{e_0} \bar \gt_F$, we have  $v_R (\overline{\varphi(\gt_F)}) = \frac{p}{e_F (p-1)}$.

\begin{lemma} \label{lem-I1} The ideal $I_F ^{[1]}$ is principal. Moreover, $x \in  I ^{[1]}_F$ is a
generator of $I^{[1]}_F$ if and only if $v_R (\bar x) = \frac{p}{e_F(p-1)}$.
\end{lemma}

\begin{proof} When $F= \Q_p$, this follows immediately from \cite [Proposition 5.1.3]{{fo3}} with $r= 1$.
The general case follows from a slight modification of the argument in {\em loc.~cit.}, as follows:
For $y  \in I^{[1]}_F$, we first claim that $v_R(\bar y) \geq \frac{p}{e_F(p -1)}$. To see this, we write $y = \sum \limits_{n = 0} ^\infty \varpi^n [y_i]$ with  $y _i  \in R$ given by the $p$-power compatible sequence
$y _ i = (\alpha ^{(n)}_i)_{n \geq 0}$ for $\alpha ^{(n)}_i \in \O _{\C_K}$.  Then
$$ 0 = \theta _F (\varphi ^m (y)) = \sum_{n = 0}^\infty \varpi ^n (\alpha^{(0)}_i) ^{p ^m}.  $$
By induction on $n$ and $m$, it is not difficult to show that
$$v(\alpha^{(0)}_i) \geq \frac{1}{e_F} p^{-i}(1 + p^{-1} \cdots + p ^{-j})$$ for all $j\ge 0$.
In particular, $v_R(\bar y) = v (\alpha_0^{(0)}) \geq \frac{p}{e_F(p -1)}$.

Now pick a $x \in I^{[1]}_F$ with $v_R(\bar x) = \frac{p}{e_F (p -1)}$ (for example, we may take $x = \varphi (\gt_F)$). Since $v_R(y) \geq v_R(x)$, we may write $ y = a x + \varpi z $ with $a, z \in W(R)_F$.
One checks that $z \in I ^{[1]}_F $ and hence that $z \in (\varpi, x)$. An easy induction argument then shows that
$y = \sum\limits _{n= 0}^\infty \varpi^n a_n x$, and it follows that $I ^{[1]}_F$ is generated by $x$.
\end{proof}

It follows at once from Lemma \ref{lem-I1} that $\varphi (\gt_F)$ is a generator of $I ^{[1]}_F$.
Write $I ^+$ for the kernel of the canonical projection $\rho:W(R)_F \to W(\bar k)_F$ induced by the projection $R \to \bar k$. Using the very construction of $u$, one checks that $u \in I^+$: Indeed, writing $\gu= [\upi]$ as before,
we obviously have $\gu \in I ^+$, and it follows from the proof of Lemma \ref{lem-basic} that
$u = \varinjlim_{n\rightarrow \infty}f^{(n)} \circ \varphi^{-n} (\gu)$ lies in $I^+$ as well.

\begin{lemma} \label{lem-z in I plus.} Let $g \in G$ be arbitrary. Then $g(u)- u $ lies in $I ^{[1]}W(R)_F$. Moreover,  if $\varpi ^2 |a_1$ in $\O_F$ then $\frac{g(u)- u}{\varphi(\gt_F)} $ lies in $I^+. $
\end{lemma}
\begin{proof}
As before, writing $f ^{(n)} = f \circ \cdots \circ f $ for the $n$-fold composition of $f$ with itself,
we have $\theta_F (\varphi^n(u))  =  f^{(n)}(\pi) \in K$, from which it follows that $g(u)- u$ is in $I ^{[1]}_F$.
By Lemma \ref{lem-I1}, we conclude that $z : = \frac{g(u)-u}{\varphi (\gt_F)}$ lies in $W(R)_F$. It remains to show that $z \in I ^+$ when $\varpi^2|a_1$.
We first observe that
$$ \varphi (z) = \frac{ f((g(u))- f(u)}{\varphi^2 (\gt_F)}= \frac{\sum\limits _{i = 1} ^p a_i \left ((g(u)) ^i - u ^i  \right )}{\varphi(E(u)) \varphi(\gt_F)}.$$
For each $i$, we may write $(g(u)) ^ i - u^i= (g(u)-u) h_i(g(u), u)= \varphi (\gt _F) z h_i (g(u),u)$
for some bivariate polynomials $h_i$ with coefficients in $W(R)_F$.
We therefore have
\begin{equation}
	{\varphi(E(u)) }\varphi (z) = \sum\limits _{i = 1} ^p a_i \left ( z h _i (g(u), u) \right ). \label{redI+}
\end{equation}
Reducing modulo $I^+$ and noting that both $u$ and $g(u)$ lie in $I ^+$, we
conclude from (\ref{redI+}) that $\varpi \varphi(\rho(z)) = a_1 \rho(z)$, where $\rho : W(R)_F \to W(\bar k)_F$ is the natural projection as above.  Using the fact that $v(\varphi(\rho(z)))= v(\rho(z))$, our assumption that $v(a_1) > v (\varpi)$ then implies that $\rho(z)= 0$. That is, $z \in I ^{+}$ as desired.
\end{proof}

\begin{example}
The following example shows that the condition $\varpi^2 | a_1$ in $\O_F$
is genuinely necessary for the conclusion of Lemma \ref{lem-z in I plus.} to hold.
Recall the situation of Example \ref{ex-cyclotomic}, with $K = \Q_p (\zeta_p)$,  $\pi = \zeta_p -1$,  $f(u) = (u +1)^p -1$ and $u = [\underline \epsilon_1 ]-1$, where $\underline \epsilon_1 = (\zeta_{p^n})_{n \geq 1} \in R$.
We may choose $g \in G$ with $g (\underline \epsilon_1) = \underline \epsilon_1^{1+p}$. We
then have $g(u )-u =  [\underline \epsilon_1] ([\underline \epsilon_1] ^p -1)$.
Now it is well-known that $[\underline \epsilon_1] ^p -1$ is a generator of $I ^{[1]}_{\Q_p}$
(or one can appeal to Lemma \ref{lem-I1}). Then $z =(g(u)-u)/\varphi(\gt_F) $ is a unit in $W(R)$
and does not lie in $I ^+$.
\end{example}

We conclude this discussion with the following lemma, which will be useful in \S \ref{subsec-6.1}:
\begin{lemma}\label{lem-stable} The ideal $\gt_F I^+ \subset W(R)_F$ is stable under the
canonical action of $G$: that is, $g(\gt_FI^+) \subset \gt_F I^+$ for all $g \in G$.
\end{lemma}
\begin{proof} It is clear that $I^+$ is $G$-stable, so it suffices to show that $g(\gt_F) =x \gt_F $
for some $x \in W(R)_F$. Since $\varphi (\gt_F)$ is a generator of $I^{[1]}$, which is obviously $G$-stable
from the definition, we see that $g(\varphi (\gt_F))= y \varphi (\gt_F)$ with $y \in W(R)_F$.
Hence $g(\gt_F)= \varphi ^{-1}(y) \gt_F$.
\end{proof}

\section{\'Etale $\varphi$-modules and Kisin modules}\label{sec-3}

In this section, following Fontaine, we establish a classification of $G_{\ul{\pi}}$-representations by
\'etale $\varphi$-modules and Kisin modules.
To do this, we must first show that $K_{\ul{\pi}}/K$ is {\em strictly Arithmetically Profinite, or APF,}
in the sense of Fontaine--Wintenberger \cite{win}, so that the theory of norm fields applies.

\subsection{Arithmetic of $f$-iterate extensions}\label{subsec-3.0}

We keep the notation and conventions of \S \ref{sec-2}.
Recall that our choice of an $f$-compatible sequence $\{\pi_n\}_n$
(in the sense that $f(\pi_{n})=\pi_{n-1}$ with $\pi_0=\pi$
a uniformizer of $K$) determines an element
$\underline{\pi}:=\{\pi_n\bmod \varpi\}$ of $R$.
It also determines
an infinite, totally wildly ramified
extension $K_{\ul{\pi}}:=\cup_{n\ge 1} K(\pi_n)$ of $K$,
and we write $G_{\ul{\pi}}=\gal(\overline{K}/K_{\ul{\pi}})$.

\begin{lemma}\label{strapflem}
	The extension $K_{\ul{\pi}}/K$ is strictly APF in the sense of \cite{win};
	in particular, the associated norm field $\mathbf{E}_{K_{\ul{\pi}}/K}$ is canonically identified
	with the subfield $k(\!(\ul{\pi})\!)$ of $\Fr(R)$.
\end{lemma}

\begin{proof}
	That $K_{\ul{\pi}}/K$ is strictly APF follows immediately from \cite{CaisDavis},
	which handles a more general situation.
	In the present setting with $f(u)\equiv u^p\bmod \varpi$, we can give a short proof as follows.
	As before, let us write
	\begin{equation*}
		f(u) = a_1u + a_2u^2 + \cdots + a_{p-1}u^{p-1} + a_pu^p,
	\end{equation*}
	with $a_i\in \varpi\O_F$ for $1\le i\le p-1$ and $a_p:=1$.
	For each $n\ge 1$, set $f_n:=f-\pi_{n-1}$ and put $K_{n}:=K(\pi_{n-1})$.
	We compute the ``ramification polynomial"
	$$g_n:= \frac{f_n(\pi_nu + \pi_n)}{u}=\sum_{i=0}^{p-1} b_i u^i,$$
	with coefficients $b_i$ given by
	\begin{equation*}
		b_i = \sum_{j=i+1}^p a_j \pi_n^j \binom{j}{i+1}\qquad \text{for}\qquad 0\le i\le p-1
	\end{equation*}
	For ease of notation, put $v_n:=v_{K_{n+1}}$, and denote by $e_n:=v_n(\varpi)$
	the ramification index of $K_{n+1}/F$ and by $e:=v_{F}(p)$ the absolute ramification index of $F$.
	  Since $K_{n+1}/K_{n}$ is totally ramified of degree $p$,
	we have $e_n = p^{n} e_0$; in particular, $v_n(a_j\binom{j}{i+1}\pi_n^j) \equiv j\bmod p^{n}$.
	It follows that $v_n(b_{p-1})= p$, and for $0\le i\le p-2$ we have
	\begin{equation*}
		v_n(b_i) = \min\{e_ne+p,\ e_nv_F(a_j) + j\ :\ i+1\le j \le p-1\}
	\end{equation*}
	It is easy to see that for $n\ge 1$ the lower convex hull of these points
	is the straight line with endpoints $(0,v_n(b_0))$ and $(p-1,p)$.
	In other words, defining
	\begin{equation}
		i_{\min}:=\min_i\{ i\ :\ \ord_{\varpi}(a_i) \le e,\ 1\le i\le p\}.\label{icalc}
	\end{equation}
	the Newton polygon of $g_n$ is a single line segment
	with slope the negative of
	\begin{equation}
		i_{n}:=\frac{e_n\left(v_F(a_{i_{\min}}) + \lfloor i_{\min}/p \rfloor e\right)+i_{\min}-p}{p-1},
		\label{ineqn}
	\end{equation}
	In particular, for $n\ge 1$ the extension $K_{n+1}/K_{n}$ is
	{\em elementary of level $i_{n}$} in the sense of \cite[1.3.1]{win};
	concretely, this condition means that
	\begin{equation}
		v_n(\pi_n - \sigma \pi_n) = i_n+1
	\end{equation}
	for {\em every} $K_{n}$-embedding $\sigma : K_{n+1}\hookrightarrow \overline{K}$.
	It follows from this and \cite[1.4.2]{win} that $K_{\ul{\pi}}/K$
	is APF.
	Now let $c(K_{\ul{\pi}}/K)$ be the constant defined in \cite[1.2.1]{win}.
	Then by \cite[\S1.4]{win}
	\begin{equation*}
		c(K_{\ul{\pi}}/K) = \inf_{n >0} \frac{i_n}{[K_{n+1}:K]},
	\end{equation*}
	so from (\ref{icalc}) we deduce
	\begin{align*}
		c(K_{\ul{\pi}}/K) &= \inf_{n>0} \frac{e_n\left(v_F(a_{i_{\min}}) + \lfloor i_{\min}/p \rfloor e\right)+i_{\min}-p}{p^n(p-1)}\\ 	
		&= \frac{e_0}{p-1}\left(v_F(a_{i_{\min}}) + \lfloor i_{\min}/p \rfloor e\right) -  \frac{p-i_{\min}}{p(p-1)}	
	\end{align*}
	since $p-i_{\min}\ge 0$, so the above infimum occurs when $n=1$.
	As $i_{\min}\ge 1$, the above constant is visibly positive,
	so by the very definition \cite[1.2.1]{win}, $K_{\ul{\pi}}/K$ is {\em strictly} APF.
	
	The canonical embedding
	of the norm field of $K_{\ul{\pi}}/K$ into $\Fr(R)$ is described in \cite[\S4.2]{win};
	that the image of this embedding coincides with $k(\!(\underline{\pi})\!)$
	is a consequence of \cite[2.2.4, 2.3.1]{win}.
\end{proof}

\begin{remark}
	Observe that if the coefficient $a_1$ of the linear term of $f(u)$
	has $v(a_1)\le 1$, then we have $i_{\min}=1$ and
	$$c(K_{\ul{\pi}}/K)=\frac{e_0}{p-1} v_F(a_1)-\frac{1}{p}.$$
	In this situation, $v_{F}(a_1)$---which plays an important role
	in our theory---is encoded in the ramification structure of $K_{\ul{\pi}}/K$.
\end{remark}

It is natural to ask when two given polynomials $f$ and $f'$ with corresponding
compatible choices $\ul{\pi}$ and $\ul{\pi}'$ give rise to the same iterate extension.
Let us write $f(x)=x^p + a_{p-1}x^{p-1}+\cdots + a_1x$
and $f'(x)= x^p + a'_{p-1}x^{p-1}+\cdots + a'_1 x$, with $a_i, a'_i\in \O_F$
and $a_i\equiv a'_i\equiv 0\bmod \varpi $ for $1\le i < p$.
Let $\{\pi_n\}$ (respectively $\{\pi_n'\}$) be an $f$ (resp. $f'$) compatible
sequence of elements in $\overline{K}$. 
Set $K_{n}:=K(\pi_{n-1})$ (resp. $K'_{n}=K(\pi'_{n-1})$) and let $a_s u ^s$ and $a'_{s'} u ^{s'}$ be the lowest degree terms of $f(u)$ and $f'(u)$ respectively.

\begin{prop}\label{fdetbyK}
	If $K_{\ul{\pi}} = K_{\ul{\pi}'}$ as subfields of $\ol{K}$,
	 then $K_{n}=K_{n}'$ for all $n\ge 1$ and
	there exists an invertible power series $\xi(x)\in \O_F[\![x]\!]$
	with $\xi(x)=\mu_0 x + \cdots$ and $\mu_0\in \O_F^{\times}$ such that
	$$f(\xi(x)) = \xi(f'(x)).$$
	In particular, $s=s'$ and $v(a_s)=v(a_s')$ are numerical invariants
	of $K_{\ul{\pi}}=K_{\ul{\pi}'}$.
	
	Conversely, given $f$ and $f'$  with $s=s'$ and $v(a_s)=v(a_s')$,
	we have $a_s=\mu_0^{1-s} a_s'$ for a unique $\mu_0\in \O_F^{\times}$ and
	there is a unique power series $\xi(x)\in F[\![x]\!]$
	with $\xi(x) \equiv \mu_0 x \bmod x^2$ satisfying $f(\xi(x))=\xi(f'(x))$ as formal power series
	in $F[\![x]\!]$.
	If $\xi(x)$ lies in $\O_F[\![x]\!]$, then for any choice $\{\pi'_n\}_n$
	of $f'$-compatible sequence with $\pi'_0$ a uniformizer of $K$, the sequence defined by
	$\pi_n:=\xi(\p_n')$ is $f$-compatible with $\pi_0 = \xi(\pi'_0)$
	a uniformizer of $K$ and $K_{\ul{\pi}}=K_{\ul{\pi}'}$. Furthermore, if $v(a_s) = v(a'_s) = v(\varpi)$,
	then $\xi(x)$ always lies in $\O_F[\![x]\!]$.
\end{prop}

\begin{proof}
	Suppose first that $K_{\ul{\pi}}=K_{\ul{\pi}'}$, and write simply $K_{\infty}$
	for this common, strictly APF extension of $K$ in $\ol{K}$.
	It follows from the proof of Lemma \ref{strapflem} that
	$K_{n+1}$ and $K_{n+1}'$ are both the $n$-th elementary subextension of $K_{\infty}$;
	{\em i.e.} the fixed field of $G_K^{b_n}G_{K_{\infty}}$,
	where $b_n$ is the $n$-th break in the ramification
	filtration $G_K^u G_{K_{\infty}}$; see \cite[1.4]{win}.
	In particular, $K_{n+1}=K_{n+1}'$ for all $n\ge 0$.
	Now let $W_{\varpi}(\bullet)$ be the functor of $\varpi$-Witt vectors;
	it is the unique functor from $\O_F$-algebras to $\O_F$-algebras
	satisfying
	\begin{enumerate}
		\item For any $\O_F$-algebra $A$, we have $W_{\varpi}(\bullet)=\prod_{n\ge 0} \bullet =:\bullet^{\mathbf{N}}$
		as functors from $\O_F$-algebras to sets.
		\item The ghost map $W_{\varpi}(\bullet)\rightarrow \bullet^{\mathbf{N}}$ given by
		$$(a_0,a_1,a_2,\ldots)\mapsto (a_0, a_0^p + \varpi a_1, a_0^{p^2} + \varpi a_1^p +\varpi^2 a_2, \ldots)$$
		is a natural transformation of functors from $\O_F$-algebras to $\O_F$-algebras.
	\end{enumerate}
	We remark that $W_{\varpi}(\bullet)$ exists and depends only on $\varpi$,
	and is equipped with a unique natural transformation $\varphi:W_{\varpi}(\bullet)\rightarrow W_{\varpi}(\bullet)$
	which on ghost components has the effect $(a_0,a_1,\ldots)\mapsto (a_1,a_2,\ldots)$;
	see \cite[\S2]{CaisDavis}.
	
	Define the ring
	$$\mathbf{A}_{K_{\infty}/K}^+:= \{(\underline{x}_i)_i\in \varprojlim_{\varphi}
	W_{\varpi}(\O_{\widehat{K}_{\infty}})\ :\ \underline{x}_n\in W_{\varpi}(\O_{K_{n+1}})\ \text{for all}\ n\},$$
	which depends only on $F, \varpi$, and $K_{\infty}/K$.  The main theorem of \cite{CaisDavis}, implies that
	$\mathbf{A}_{K_{\infty}/K}^+$ is a $\varpi$-adically complete and separated $\O_F$-algebra
	equipped with a Frobenius endomorphism $\varphi$, which is canonically a Frobenius-stable subring
	of $W(R)_F$ that is closed under the weak topology on $W(R)_F$.
	Giving $\mathbf{A}_{K_{\infty}/K}^+$ the subspace topology,
	the proof of \cite[Prop 7.14]{CaisDavis} then shows that
	the $f$ (respectively $f'$)-compatible sequence $\ul{\pi}$
	(respectively $\ul{\pi}'$) determine isomorphisms of topological $\O_F$-algebras
	\begin{equation*}
		\xymatrix{
			{\eta, \eta':\O_F[\![x]\!]} \ar[r] & {\mathbf{A}_{K_{\infty}/K}^+}
			}
	\end{equation*}
	characterized by the requirement that the ghost components of $(\eta)_n$
	(resp. $(\eta')_n$)
	are $(\pi_n, f(\pi_n), f^{(2)}(\pi_n),\ldots)$ (resp.
	$(\pi'_n, f'(\pi'_n), f'^{(2)}(\pi'_n),\ldots)$; here we give $\O_F[\![x]\!]$
	the $(\varpi,x)$-adic topology.
	These isomorphisms moreover satisfy
	\begin{equation*}
		\eta(f(x)) = \varphi(\eta(x)) \qquad\text{and}\qquad \eta'(f'(x)) = \varphi(\eta'(x)).
	\end{equation*}
We therefore obtain a continuous automorphism $\xi: \O_F[\![x]\!]\rightarrow \O_F[\![x]\!]$
	satisfying
	\begin{equation}
		f(\xi(x)) = \xi(f'(x)).
		\label{intertwine}
	\end{equation}
Since $\xi$ is a continuous automorphism of $\O_F [\![x]\!]$, we have that $\xi$ preserves the maximal ideal
$(\varpi, x)$. This implies  that $\xi (x) \equiv \mu_0 x \mod x^2$ with $\mu_0 \in \O_F^\times$. Then \eqref{intertwine} forces that $a_s\mu_0^s x ^s = a'_{s'}\mu_0 x ^{s'}$ which implies
$s= s'$ and $v(a_s) = v(a'_s)$.

Conversely, suppose given $f$ and $f'$ with $s=s'$ and $v(a_s)=v(a'_s)$ and let $\mu_0 \in \O_F^{\times}$
be the unique unit with $a_s=\mu_0^{1-s}a_s'$; note that this exists as $s-1< p$.
	We inductively construct degree $i$ polynomials $\xi_i(x) = \sum_{j = 1} ^i \mu _j x ^j  $
	so that $f(\xi_i(x)) \equiv \xi_i(f'(x))\bmod x^{i+s}$.
	As $\mu_0^sa_s=\mu_0a'_s$, we may clearly take $\xi_1(x) = \mu_0 x$.  Supposing that $\xi_i(x)$
	has been constructed, we write $\xi_{i+1}(x) = \xi_i(x) + \mu_{i +1} x^{i+1}$ and
	$f(\xi_i(x))-\xi_i(f'(x)) \equiv \lambda x^{i+s}\bmod x^{i+s+1}$ and
	seek to solve
	\begin{equation}
		f(\xi_{i+1}(x)) \equiv \xi_{i+1}(f'(x))\bmod x^{i+s +1}.
		\label{seek2solve}
	\end{equation}
As $f(\xi_{i +1}(x)) = f (\xi _i (x)) + \frac{df}{dx}(\xi_i (x)) (\mu_{i +1} x ^{i +1}) + \cdots$,
we see that (\ref{seek2solve}) is equivalent to
	\begin{equation}
		\lambda = \mu_{i +1} (a_1-{a'_1}^{i+1}) \text{ if  } s=1,\quad\text{and}\quad\lambda = \mu_{i +1} s a_s \mu_0^{s-1} \text{ if  } s>1
		\label{musys}
	\end{equation}
	which admits a unique solution $\mu_{i+1}\in F$. 	
	We set $\xi(x)=\lim_i \xi_i(x) \in F[\![x]\!]$, which by construction
	satisfies the desired intertwining relation \eqref{intertwine}.
	If $\xi\in \O_F[\![x]\!]$, it is clear that any $f'$-compatible sequence $\pi_n'$
	with $\pi_0'$ a uniformizer of $K$ yields an $f$-compatible sequence $\pi_n:=\xi(\pi_n')$
	with $\pi_0$ a uniformizer of $K$ and $K_n:=K(\pi_{n-1})=K(\pi_{n-1}')=K_n'$ for all $n\ge 1$.
	Finally, since $f(x)= f'(x) \equiv x ^p \bmod \varpi$, we have $f (\xi_i (x)) - \xi_i (f' (x)) \equiv 0 \mod \varpi$, {\em i.e.}~$\lambda \equiv 0 \mod \varpi$ in the above construction.  In particular, when $v(a_s) = v(a'_s) = v(\varpi)$,
	it follows from
	(\ref{musys}) that $\mu_{i +1 } \in \O_F$, and $\xi(x) \in \O_F [\![x]\!]$
	as claimed.
	%
\end{proof}

It follows at once from Proposition (\ref{fdetbyK}) that there are infinitely many distinct
$f$-iterate extensions $K_{\ul{\pi}}$ inside of $\ol{K}$.

\subsection{\'Etale $\varphi$-modules}

Let $\O _\E$ be the $p$-adic completion of $\gs_F[1/u]$, equipped with the unique continuous
extension of $\varphi$.
Our fixed embedding $\gs_F\hookrightarrow W(R)$ determined by $f$ and $\ul{\pi}$
uniquely extends to a $\varphi$-equivariant embedding $\iota:\O_{\E}\hookrightarrow W(\t{Fr} R)_F$, and we
identify $\O_{\E}$ with its image in $W(\t{Fr} R)_F$.
We note that $\O_\E$ is a complete discrete valuation ring with uniformizer $\varpi$ and residue field
$k (\!(\ul{\pi})\!)$,
which, as a subfield of $\Fr R$, coincides with the norm field of $K_{\ul{\pi}}/K$ thanks to Lemma
\ref{strapflem}.  As $\Fr R$ is algebraically closed,
the separable closure $k(\!(\ul{\pi})\!)^{\sep}$ of $k(\!(\ul{\pi})\!)$
in $\Fr R$ is unique, and the maximal unramified extension ({\em i.e.}~strict Henselization)
$\O _{\E^\ur}$ of $\O_{\E}$ with residue field $k(\!(\ul{\pi})\!)^{\sep}$
is uniquely determined up to unique isomorphism.  The universal property of strict
Henselization guarantees that $\iota$ uniquely extends to an embedding $\O_{\E^{\ur}}\hookrightarrow W(\Fr R)_F$,
which moreover realizes $\O_{\E^{\ur}}$ as a $\varphi$-stable subring.
We write $\O _{\widehat \E^\ur}$ for the $p$-adic completion of $\O_{\E^\ur}$, which is
again a $\varphi$-stable subring of $W(\Fr R)_F$.
Again using the universal property of strict Henselization,
one sees that each of $\O_\E, \O_{\E^\ur}$ and $\O_{\widehat \E^\ur}$ are $G_{\ul{\pi}}$-stable subrings of
 $W(\t{Fr} R)_F $, with $G_{\ul{\pi}}$ acting trivially on $\O_{\E}$.
As suggested by the notation, we write $\E$, $\E^{\ur}$, and $\wh{\E}^{\ur}$
for the fraction fields of $\O_\E, \O_{\E^\ur}$ and $\O_{\widehat \E^\ur}$,
respectively.
Finally, we define $\gs ^\ur_F : = W(R) _F \cap \O _{\widehat \E^\ur}$.


\begin{lemma}\label{APF}
With notation as above:
\begin{enumerate}
	\item 
	The natural action
	of $G_{\ul{\pi}}$ on $\OEhat$ induces an isomorphism of profinite groups
	\begin{equation*}
		G_{\ul{\pi}}:=\Gal(\overline{K}/K_{\ul{\pi}})\simeq \Aut(\OEhat/\O_{\E}) = \Gal(\widehat{\E}^{\ur}/\E).
	\end{equation*}\label{APF1}
\item The inclusions $\O_F\hookrightarrow (\OEhat)^{\varphi=1}$ and
	$\O_{\E}\hookrightarrow (\OEhat)^{G_{\ul{\pi}}}$ are isomorphisms.\label{APF2}

\end{enumerate}
\end{lemma}

\begin{proof}
	By the very construction
	of $\OEhat$ and the fact that  the residue field of $\O_{\E}$ is identified with the
	norm field $\mathbf{E}_{K_{\ul{\pi}}/K}$ by Lemma \ref{strapflem},
	we have an isomorphism of topological groups
	$\Gal(\mathbf{E}_{K_{\ul{\pi}}/K}^{\sep}/\mathbf{E}_{K_{\ul{\pi}}/K})\simeq \Aut(\OEhat/\O_{\E})$
	by the theory of unramified extensions of local fields.
	On the other hand, the theory of norm fields \cite[3.2.2]{win} provides
	a natural isomorphism of topological groups
	$G_{\ul{\pi}}\simeq \Gal(\mathbf{E}_{K_{\ul{\pi}}/K}^{\sep}/\mathbf{E}_{K_{\ul{\pi}}/K})$,
	giving (\ref{APF1}).

	To prove (\ref{APF2}), note that the maps in question are local maps of $\varpi$-adically separated and complete
	local rings, so by a standard successive approximation argument it suffices
	to prove that these maps are surjective modulo $\varpi$.
	Now left-exactness of $\varphi$-invariants
	(respectively $G_{\ul{\pi}}$-invariants)
	gives an $\mathbf{F}_p$-linear (respectively $\mathbf{E}_{K_{\ul{\pi}/K}}$-linear) injection
	$$(\OEhat)^{\varphi=1}/(\varpi)\hookrightarrow (\mathbf{E}_{K_{\ul{\pi}/K}}^{\sep})^{\varphi=1}=
	\mathbf{F}_p=\O_F/(\varpi),$$
	respectively
	$$
	(\OEhat)^{G_{\ul{\pi}}}/(\varpi)\hookrightarrow (\mathbf{E}_{K_{\ul{\pi}/K}}^{\sep})^{G_{\ul{\pi}}}
	=\mathbf{E}_{K_{\ul{\pi}/K}} = \O_{\E}/(\varpi)
	$$	
	which must be an isomorphism of vector spaces over $\mathbf{F}_p$
	(respectively $\mathbf{E}_{K_{\ul{\pi}/K}}$) as the source is nonzero and the target is
	1-dimensional.  We conclude that $\O_F\hookrightarrow (\OEhat)^{\varphi=1}$
	(respectively $\O_{\E}\hookrightarrow (\OEhat)^{G_{\ul{\pi}}}$)
	is surjective modulo $\varpi$, and therefore an isomorphism as desired.
\end{proof}



Let $\t{Mod}_{\O_\E}^\varphi$ (resp. $\t{Mod} _{\O_\E} ^{\varphi, \t{tor}}$) denote the category of finite free $\O_\E$-modules $M $ (resp. finite $\O_\E$-modules $M$ killed by a power of $\varpi$), equipped with a $\varphi_{\O_\E}$-semi-linear endomorphism $\varphi _M : M\to M$ whose linearization $1 \otimes \varphi : \varphi ^*M \to M $ is an isomorphism. In each case, morphisms are $\varphi$-equivarant $\O_{\E}$-module homomorphisms.
Let $\t{Rep}_{\O_F}(G_{\ul{\pi}}) $ (resp. $\t{Rep}^\t{tor}_{\O_F}(G_{\ul{\pi}})$) be the category of finite,
free $\O_F$ modules (resp. finite $\O_F$-modules killed by a power of $\varpi$)
that are equipped with a continuous and $\O_F$-linear action of $G_{\ul{\pi}}$.

For $M$ in $\t{Mod}_{\O_\E}^\varphi$ or in  $\t{Mod} _{\O_\E} ^{\varphi, \t{tor}}$, we define
$$ \ul{V}(M):= ( \O _{\widehat \E ^\ur} \otimes_{\O_\E} M) ^{\varphi =1},  $$
which is an $\O_F$-module with a continuous action of $G_{\ul{\pi}}$.
For $V $ in $\t{Rep}_{\O_F}(G_{\ul{\pi}}) $ or in  $\t{Rep}^\t{tor}_{\O_F}(G_{\ul{\pi}})$, we define
 $$ \ul{M} (V) = (\O _{\widehat \E ^\ur} \otimes_{\O_F} V )^{G_{\ul{\pi}}}, $$
 which is an $\O_{\E}$-module with a $\varphi$-semilinear endomorphism
 $\varphi_{\ul{M}}:=\varphi_{\OEhat} \otimes 1.$

\begin{theorem} \label{thm-equi of etal stuffs}
The functors $\ul{V}$ and $\ul{M}$ are
quasi-inverse equivalences between the exact tensor categories
$\t{Mod}_{\O_\E} ^\varphi$ $($resp. $\t{Mod}_{\O_\E} ^{\varphi, \t{tor}}$$)$ and $\t{Rep}_{\O_F} (G_{\ul{\pi}})$
 $($resp. $\t{Rep}^\t{tor}_{\O_F}(G_{\ul{\pi}})$$)$.
\end{theorem}

\begin{proof}
	As in the proof of \cite[Theorem 1.6]{Kisin-Ren}, the original arguments of Fontaine
	\cite[A1.2.6]{fo4} carry over to the present situation.
	Indeed, by standard arguments with inverse limits, it is enough to prove the Theorem
	for $\varpi$-power torsion objects.  To do so, one first
	proves that
	$\ul{M}$ is exact, which by (faithful) flatness of the inclusion
	$\O_{\E}\hookrightarrow \O_{\E^{\ur}}$ amounts to the vanishing of $H^1(G_{\ul{\pi}},\cdot)$
	on the category of finite length $\O_{\E^{\ur}}$-modules with a continuous semilinear
	$G_{\ul{\pi}}$-action.  By a standard d\'evissage,
	such vanishing is reduced to the case of modules killed by $\varpi$, where it follows
	from Hilbert's Theorem 90 and Lemma \ref{APF}.
	One then checks that for any torsion $V$, the natural comparison map
	${\ul{M}(V)\otimes_{\O_{\E}}\O_{\E^{\ur}}}\rightarrow {V\otimes_{\O_F} \O_{\E^{\ur}}}$
	induced by multiplication in $\O_{\E^{\ur}}$ is an $\O_{\E^{\ur}}$-linear,
	$\varphi$, and $G_{\ul{\pi}}$-compatible isomorphism by d\'evissage (using the settled
	exactness of $\ul{M}$) to the case
	that $V$ is $\varpi$-torsion, where it again follows from Hilbert Theorem 90.
	Passing to submodules on which $\varphi$ acts as the identity
	and using Lemma \ref{APF}(2) then gives a natural isomorphism $\ul{V}\circ \ul{M} \simeq \mathrm{id}$.
	
	In a similar fashion,
	the exactness of $\ul{V}$ and the fact that the natural comparison map
	\begin{equation}
			\xymatrix{
			{\ul{V}(M)\otimes_{\O_{F}}\O_{\E^{\ur}}}\ar[r] & {M\otimes_{\O_{\E}} \O_{\E^{\ur}}}
			}\label{Vcomparemap}
	\end{equation}
	induced by multiplication is an isomorphism for general $\varpi$-power torsion modules $M$
	follows by d\'evissage from the the truth of these claims in the case of $M$ killed
	by $\varpi$.  In this situation, the comparison map (\ref{Vcomparemap})
	is shown to be injective by checking that any $\mathbf{F}_p$-linearly independent
	set of vectors in $\ul{V}(M)$ remains $\mathbf{E}^{\sep}_{K_{\infty/K}}$-linearly
	independent in $\mathbf{E}^{\sep}_{K_{\infty/K}}\otimes_{\mathbf{F}_p} \ul{V}(M)$,
	which is accomplished by a standard argument using
	the Frobenius endomorphism and Lemma \ref{APF}(2).
	 To check surjectivity is then a matter of showing that both sides of (\ref{Vcomparemap})
	 have the same $\mathbf{E}^{\sep}_{K_{\infty/K}}$-dimension, {\em i.e.}~ that
	 the $\mathbf{F}_p$-vector space $\ul{V}(M)$ has dimension
	$d:=\dim_{\mathbf{E}_{K_{\infty/K}}} M$.
	Equivalently, we must prove that $\ul{V}(M)$ has $p^d$ elements.
	Identifying $M$ with $\mathbf{E}_{K_{\infty/K}}^d$ by a choice of $\mathbf{E}_{K_{\infty/K}}$-basis
	and writing $(c_{ij})$ for the resulting matrix of $\varphi$,
	one (noncanonically) realizes $\ul{V}(M)$ as the set of $\mathbf{E}^{\sep}_{K_{\infty/K}}$-solutions
	to the system of $d$-equations $x_i^p = \sum a_{ij}x_j$ in $d$-unknowns,
	which has exactly $p^d$ solutions as $\varphi$ is \'etale, so the matrix $(c_{ij})$
	is {\em invertible}.
\end{proof}

In what follows, we will need a contravariant version of Theorem \ref{thm-equi of etal stuffs},
which follows from it by a standard duality argument ({\em e.g.}~ \cite[\S1.2.7]{fo4}).
For any $M \in \t{Mod}_{\O_\E}^\varphi$
(respectively $M \in \t{Mod}_{\O_\E}^{\varphi,\t{tor}}$), we define
$$\ul{T}(M):= \Hom_{\O_\E, \varphi} (M , \OEhat),
\quad\text{respectively}\quad
\ul{T}(M):= \Hom_{\O_\E, \varphi} (M , \wh{\E}^{\ur}/\OEhat),
$$
which is naturally an $\O_F$-module with a continuous action of $G_{\ul{\pi}}$.

\begin{co}\label{co-TM} The contravariant functor $\ul{T}$ induces an anti-equivalence between  $\t{Mod}_{\O_\E} ^\varphi$ $($resp. $\t{Mod}_{\O_\E} ^{\varphi, \t{tor}}$$)$ and $\t{Rep}_{\O_F} (G_{\ul{\pi}})$ $($
resp. $\t{Rep}_{\O_F}^{\t{tor}} (G_{\ul{\pi}})$$)$.
\end{co}

\subsection{Kisin modules and Representations of finite $E$-height}\label{subsec-3.2}

For a nonnegative integer $r$,
we write $'\sfi$  for the category of finite-type $\gs_F$-modules $\M$ equipped with
a $\varphi_{\gs_F}$-semilinear endomorphism $\varphi_\M : \M \to \M$ satisfying
\begin{itemize}
	\item the cokernel of the linearization $1\otimes \varphi: \varphi ^*\M \to \M$ is killed by $E(u)^r$;
	\item the natural map $\M \to \O_\E \otimes_{\gs_F} \M$ is injective.
\end{itemize}
One checks that together these
conditions guarantee that the scalar extension $\O_\E \otimes_{\gs_F} \M$ is an object of
$\t{Mod}_{\O_\E} ^\varphi$ when $\M$ is torsion free,  and an object of
$\t{Mod}^{\varphi, \t{tor}}_{\O_\E}$ if $\M$ is killed by a power of $\varpi$.
Morphisms in $'\sfi$ are $\varphi$-compatible $\gs_F$-module homomorphisms.
By definition, the category of {\em Kisin modules of $E(u)$-height $r$},
denoted $\sfi$, is the full subcategory of $'\sfi$ consisting
of those objects which are finite and free over $\gs_F$.
For any such Kisin module $\M \in \t{Mod}_{\gs_F} ^{\varphi, r}$, we define
$$T_{\gs} (\M): = \Hom_{\gs_F , \varphi} (\M , \gs ^\ur_F),$$
with $\gs ^\ur_F:= W(R) _F \cap \O _{\widehat \E^\ur}$
as above Lemma \ref{APF}; this is naturally an $\O_F$-module with a linear
action of $G_{\ul{\pi}}$.

\begin{prop}\label{lem-Tgsbasic}
Let $\M\in \sfi$ and write $M = \O_\E \otimes _{\gs_F }\M$
for the corresponding object of  $\t{Mod}^\varphi_{\O_\E}$.
\begin{enumerate}
\item There is a canonical isomorphism of $\O_F[G_{\ul{\pi}}]$-modules
$T_{\gs} (\M) \simeq \ul{T}(M)$. In particular, $T_{\gs} (\M) \in \t{Rep}_{\O_F}(G_{\ul{\pi}})$ and
$\t{rank} _{\O_F} (T_{\gs}(\M)) = \t{rank} _{\gs _F}(\M)$.
\item The inclusion $\gs ^\ur_F\hookrightarrow W(R)_F$ induces a natural isomorphism
of $\O_F[G_{\ul{\pi}}]$-modules
$T_{\gs} (\M )\simeq \Hom_{\gs_F , \varphi}(\M , W(R)_F)$.
\end{enumerate}
\end{prop}

\begin{proof}
	As in the proofs of \cite[2.1.2, 2.1.4]{kisin2} and \cite[3.2.1]{Kisin-Ren},
	the Lemma follows from B1.4.2 
	and B1.8.3 of \cite{fo4} ({\em cf}. B1.8.6),
	using \cite[A1.2]{fo4}
	and noting that Fontaine's arguments---which are strictly speaking
	only for $F=\Q_p$---carry over {\em mutatis mutandis} to our more general situation.
\end{proof}

\begin{example}\label{ex-dim1} Let $\M$ be a Kisin module of rank 1 over $\gs_F$.
Choosing a basis $\fe$ of $\M$ and identifying $\M=\gs_F\cdot\fe$,
it follows from Weierstrass preparation that we must have
$\varphi (\fe)= \mu E(u)^m\fe $ for some $\mu \in \gs_F ^\times$.
Consider the particular case that $\varphi (\fe) = E(u) \fe$,
which is a rank-1 Kisin module of $E$-height 1.
Proposition (\ref{lem-Tgsbasic}) then shows that $T_{\gs} (\M)$ gives an
$\O_F$-valued character of $G_{\ul{\pi}}$ and that there exists $\gt \in  W(R)_F$
satisfying $\varphi (\gt) = E(u) \gt$ and $\gt \bmod \varpi \not = 0$ inside $R$. We will see in \S \ref{Sec-BT} that the character of $G_{\ul{\pi}}$
furnished by $T_{\gs}(\M)$ can be extended to a Lubin-Tate character of $G$ if we assume that $\varpi ^2 | a_1$ in $\O_F$,
where $a_1$ is the linear coefficient of $f(x)\in \O_F[x]$.
\end{example}

Let $\rep_{F}( G_{\ul{\pi}}) $ denote the category of continuous, $F$-linear representations of $G_{\ul{\pi}}$.
We say that an object  $V$ of $\rep_{F}( G_{\ul{\pi}})$ is of \emph{$E(u)$-height $r$} if there exists a
Kisin module $\M \in \sfi$ with $V \simeq T_{\gs_F} (\M)[1/ p]$, and
we say that $V$ is of {\em finite $E(u)$-height} if there exists an integer $r$ such that $V$ is of $E(u)$-height $r$.
As $E=E(u)$ is fixed throughout this paper, we will simply say that $V$ is of {\em $($finite$)$ height $r$}.

For $\M$ an arbitrary object of $\sfi$, we write $V_{\gs}(\M) : = T_{\gs}(\M)[1/p]$
for the associated height-$r$ representation of $G_{\ul{\pi}}$.
We will need the following generalization of
\cite[Lemma 2.1.15]{kisin2} (or \cite[Corollary 2.3.9]{liu2}):


\begin{prop} \label{prop-rep-height r}
Suppose that $V\in \rep_F(G_{\ul{\pi}})$ is of height $r$.
Then for any $G_{\ul{\pi}}$-stable $\O_F$-lattice $L\subset V$, there exists
$\N \in \sfi$ such that $T_{\gs}(\N ) \simeq L$ in $\rep_{\O_F}(G_{\ul{\pi}})$.
\end{prop}

The proof of Proposition \ref{prop-rep-height r} we make use of the following key lemma:

\begin{lemma}\label{lem-digfree}
Let $\M$ be an object of $'\sfi$ that is torsion-free.  Then the intersection
$\M ' : = \M [1/p] \cap  (\O_\E \otimes_{\gs_F} \M) $,
taken inside of $\E\otimes_{\gs_F} \M$, is an object in $\t{Mod} ^{\varphi, r} _{\gs_F}$ and
there are canonical inclusions $\M \subset \M' \subset \O_\E \otimes_{\gs_F} \M.$
\end{lemma}

\begin{proof}
The proof of Lemma 2.3.7 in \cite{liu2} carries over {\em mutatis mutandis}
to the present situation.
\end{proof}

\begin{proof}[Proof of Proposition $\ref{prop-rep-height r}$]  As the proof is a simple adaptation of
that of Corollary 2.3.9 in \cite{liu2}, we simply sketch the highlights.
Let $V\in \rep_F(G_{\ul{\pi}})$ be of height $r$, and select $\M\in \sfi$ with $V\simeq V_{\gs}(\M)$.
Put $T:=T_{\gs}(\M)$, which is a $G_{\ul{\pi}}$-stable $\O_F$-lattice in $V$, and let $L\subset V$
be an arbitrary $G_{\ul{\pi}}$-stable $\O_F$-lattice.
Put $M := \O_\E\otimes _{\gs_F} \M$ and let
$N \in \Mod_{\O_\E} ^\varphi$ be the object of $\Mod_{\O_{\E}}^{\varphi}$
corresponding to $L$ via Corollary \ref{co-TM}, so
$\ul{T}(N) \simeq L$ in $\rep_{\O_F}(G_{\ul{\pi}})$.
Without loss of generality, we may assume that $N \subset M$.
Writing $\f: M \to M/ N$ for the natural projection, it is easy to check that $\f(\M)$ is an object of $'\sfi$.
It then follows from Proposition \cite[ B 1.3.5]{fo4} that $\N':=\ker(\f\big|_{\M})\subset N$ is an object of $'\sfi$.
Writing $\N:=\N'[1/p] \cap N$, we have that $\N$ is an object of $\sfi$ thanks to Lemma \ref{lem-digfree},
and by construction we have $\O_{\E}\otimes_{\gs_F}\N \simeq N$, so that
 $T_{\gs}(\N)\simeq L$ as $\O_F[G_{\ul{\pi}}]$-modules thanks to Proposition \ref{lem-Tgsbasic}
 and the choice of $N$.
\end{proof}

\begin{prop}\label{prop-fullfaith} Assume that $\varphi^n(f(u)/u)$ is not power of $E(u)$ for any $n\geq 0$.
Then the functor $T_\gs : \t{Mod}_{\gs_F}^{\varphi, r}\rightsquigarrow \t{Rep}_{\O_F}(G_{\ul{\pi}})$ is fully faithful.
\end{prop}

\begin{proof} Here we use an idea of Caruso \cite[Proposition 3.1]{Car2}.
Let us fix $\M , \M' \in \t{Mod}^{\varphi, r}_{\gs_F}$.
Appealing to Corollary \ref{co-TM} and Lemma \ref{lem-Tgsbasic},
we immediately reduce the proof of Proposition \ref{prop-fullfaith}
to that of the following assertion: if $\f: \O_\E \otimes_{\gs_F} \M \to \O_\E \otimes_{\gs_F} \M' $ is a morphism in $\t{Mod} ^{\varphi}_{\O_\E}$ then $\f(\M ) \subset \M'$.
By applying Lemma \ref{lem-digfree} to $\f(\M) + \M'$, we may further reduce the proof to that of the following statement: if $\M \subset \M' \subset \O_\E \otimes_ {\gs_F}\M $ then $\M = \M'$.
Writing $d:=\rk_{\gs_F}(\M)=\rk_{\gs_F}(\M')$ and applying $\wedge^d$,  we may reduce to
the case $d=1$, and now calculate with bases.
Let $e$ (resp. $e' $) be an $\gs_F$-basis of $\M$ (resp. $\M'$), and let $a\in \gs_F$
be the unique element with $e = a e'$. Since $\O_\E \otimes_ {\gs_F}\M = \O_\E \otimes_ {\gs_F}\M' $, by Weierstrass preparation, we may modify our choices of $e$ and $e'$ to assume that $a = A(u) = u ^s + c_{s-1} u ^{s-1} + \cdots +c_1 u +  c_0$ with $c_i \in \varpi \O_{F_0}$.
As in Example \ref{ex-dim1}, we may
suppose that $\varphi (e') = \gamma'  E(u) ^n e'$ and $\varphi (e)= \gamma E(u)^{n'}e$ for some $\gamma, \gamma' \in \gs_F ^\times$ .
Then
$$\gamma E(u)^{n'}  A(u)e' = \gamma E(u)^{n'} e= \varphi (e)= \varphi (A(u)) \varphi (e')= \varphi (A(u)) \gamma' E(u)^n  e' $$
which necessitates $\gamma A(u)  E(u) ^{n'}= \gamma' \varphi (A(u)) E(u)^{n}$. Reducing
modulo $\varpi$ and comparing degrees $u$, we see easily that $n' \geq n$. We therefore have
\begin{equation}
	\gamma_0  A(u) E(u)^{n '-n} =  \varphi (A(u))\quad\text{for}\quad\gamma_0 =
\gamma (\gamma ') ^{-1} \in \gs ^\times_F.\label{Aurelation}
\end{equation}
As $\gamma_0$ is a unit, it follows from (\ref{Aurelation}) that
$A(u) E(u)^{n '-n}$ and $\varphi (A(u))$ must have the same roots.
Since  $A(u)$,  $\varphi (A(u))$ and $E(u)$ are monic polynomials with roots either 0 or with positive valuation,
we conclude that in fact $A(u) E(u)^{n '-n} =  \varphi (A(u))$.
Let us put $A(u) = u ^l A_0(u)$ with $A_0(0)\not = 0$ and  $m = n'-n$. Then (\ref{Aurelation})
simplifies to
\begin{equation} \label{eqn-disjoint}
A_0 (u) E(u)^m = (f(u)/u)^l \varphi (A_0 (u)).
\end{equation}

We first treat the case $l =0$ (so $A=A_0$); we will then reduce the general case to this one.
Write $A^\varphi (u) = u ^s + \varphi (c_{s-1}) u ^{s-1} + \cdots + \varphi (c_1) u + \varphi (c_0)$.
There is then a bijection between the roots of  $A^\varphi (u)$ and the roots of $A(u)$
which preserves valuation.
Let $x_0$ be a nonzero root of $A(u)$ which achieves the maximal valuation. Then $A (u) E(u)^m  =  \varphi (A(u))$ implies that $x_0$ is root of $\varphi (A (u))= A^\varphi (f(u))$. That is $f(x_0)$ is a root of $A ^\varphi (u)$.
If $f(x_0)\neq 0$, then
since $f(u)\equiv u^p\bmod\varpi$ and $x_0$ has positive valuation, we have $v(f(x_0)) > v(x_0)$,
so there exists a root of $A(u)$ with valuation strictly greater than $v(x_0)$, which
contradicts our choice of $x_0$.
We must therefore have that $f(x_0)= 0$ is root of $A^\varphi (u)$, which contradicts our assumption
that $A (0 ) \not = 0$ ($l=0$). We conclude that $A(u)=A_0(u)$ has degree zero, and hence that
$\M = \M'$ as desired.

Now suppose that $l > 0$
and let $r_1\ge 0$ be the unique integer with $(f(u)/u)^l = E(u)^{r_1} h_1(u)$
for some unique monic $h_1\in \gs_F$ with $E(u)\nmid h_1(u)$.
Comparing $u$-degrees in (\ref{eqn-disjoint}) shows that $r_1 \leq  m $,
so $h_1 (u)| A_0 (u)$ and we my write $A_0 (u)= h_1(u)A_1(u)$ for a unique monic polynomial $A_1$
and $m_1 := m-r_1 \ge 0$. Equation (\ref{eqn-disjoint}) then becomes
$$ A_1(u) E(u)^{m_1}= \varphi (h_1(u)) \varphi(A_1(u)).$$
Now let $r_2\ge 0$ be the unique integer
with $\varphi(h_1(u))= E(u)^{r_2} h_2(u)$ for $h_2$ a monic polynomial with $E(u)\nmid h_2(u)$,
and write $A_1(u)= h_2(u) A_2(u)$ with $A_2$ monic and $m_2 := m_1 -r_2 \ge 0$. We then have
$$A_2(u) E(u)^{m_2}= \varphi (h_2(u)) \varphi (A_2(u)).$$
We continue in this manner, constructing nonnegative integers $r_n, m_n$ with $m_{n+1}:=m_n -r_n$
and monic $A_n, h_n \in \gs_F$ with $E\nmid h_n$, $h_nE^{r_n}= \varphi (h_{n -1})$
and $A_{n-1}= h_{n}A_{n}$ satisfying the equation
\begin{equation}
	A_n (u) E(u)^{m_n}= \varphi (h_n (u)) \varphi (A_n(u)).\label{Anhneqn}
\end{equation}
So long as $h_n$ and $A_n$ are non-constant, we have $\deg A_n < \deg A_{n-1}$, which can not continue
indefinitely. We conclude that there is some $n\ge 1$ with either $h_n$ or $A_n$ constant, which forces
$h_n=1$ or $A_n=1$ by monicity.  In the latter case, (\ref{Anhneqn}) implies that $h_{n+1}=1$,
so in any case there is some $n>0$ with $h_n=1$.
By the construction of the $h_m$, we then have
\begin{equation}
	\varphi^{n-1} ((f(u)/ u )^l )= \prod_{m =1, r_m \not = 0}^{n} \varphi^{n -m} (E(u) ^{r_m}).\label{Eprod}
\end{equation}
We claim that in fact there is only one $m$ with $r_m \not = 0$. Indeed, if there exist $m_1 > m_2$ with
$r_{m _i} \not = 0$ for $i=1,2$, then writing $f_0 (u) = f(u)/ u$, we see that
$ f_0( f ^{(m _i)}(\pi))= 0$ for $i=1,2$.
Since $f(u)= f_0(u) u$, this implies that $f^{(m_2 + 1)} (\pi) = 0$.  Then
$$0 = f_0 (f ^{(m_1)} (\pi ))= f_0 (f ^{(m_1 - m_2 -1)}(f^{(m_2 +1)}(\pi)))=  f_0 (f ^{(m_1 - m_2 -1)}(0)) = f_0 (0),$$
which implies that $u | f_0 (u)$. But this contradicts \eqref{eqn-disjoint} because $u \nmid A_0 (u)$.
We conclude that there is a unique $m$ such that $r_m \not = 0$, and it follows from (\ref{Eprod})
that there exists $n \ge0$
such that $\varphi ^n (f(u)/ u)$ is a power of $E(u)$, contradicting our hypothesis.
We must therefore in fact have $l=0$, whence $\M=\M'$ as we showed above.
\end{proof}

\begin{remark}
The assumption that $\varphi^{(n)}(f(u)/u)$ is not a power of $E(u)$ for any $n\ge 0$
is satisfied in many cases of interest.  For example, it is always satisfied when $a_1=0$
(which includes the classical situation $f(u)=u^p$),
as then $f(u)/u$ has no constant term while any power of $E(u)$ has nonzero constant term.
\end{remark}

\begin{example}\label{ex-cyclotomic2}
	The hypothesis of Proposition \ref{prop-fullfaith} that $\varphi^{(n)}(f(u)/u)$ is not a power of $E(u)$ for any $n\ge 0$
	is genuinely necessary, as the following examples show:
	\begin{enumerate}
		\item For fixed $r$, let $0\le l\le r$ be an integer and suppose that $\varphi^{(n)}(f(u)/u)=E(u)^l$.
		Setting $A(u):=f(u)\cdot \varphi(f(u)/u)\cdots \varphi^{n-1}(f(u)/u)$ if $n>0$
		and $A(u)=u$ if $n=0$, we have $A E^l = \varphi(A)$.  In particular, definining
		$\M=A(u)\gs_F$ and $\M':=\gs_F$, we have $\M\subseteq \M'$
		and both $\M$ and $\M'$ are objects of $\t{Mod}^{\varphi, r}_{\gs_F}$
		with height $l \le r$.  However, $\M\neq \M'$ and it follows that the conclusion of
		Proposition \ref{prop-fullfaith} does not hold.\label{counter1}
		
		\item More concretely,  recall the situation in Example \ref{ex-cyclotomic}: we have
  $K = \Q_p (\zeta_p)$, $F = \Q_p$, $\pi = \zeta_p-1$ and $\varphi (u) = (u +1)^p -1$.
	In this case,  $E(u)= \varphi (u)/ u$, and the Kisin modules $\M':=\gs_F$ and $\M:=u \gs_F$
	are both of height 1 and are non-isomorphic, but $T_{\gs}(\M)\simeq T_{\gs}(\M')$.

	\item  As a less familiar variant, we take
	$\varphi (u) = (u-p)^{p-1} u$ and $E(u) = \varphi (u)  - p$.
	Then $\varphi(f(u)/u) = (E(u)) ^{p -1}$, and the construction of (\ref{counter1})
	provides a counterexample.
	\end{enumerate}
\end{example}

\begin{co}\label{co-more}Suppose that $\varphi^n(f(u)/u)$ is not a power of $E(u)$ for any $n\geq 0$ and  $\psi : V' \to V$ is a morphism of height-$r$ representations. Then there are exact sequences 
\begin{equation*}
	\xymatrix@1{
	0 \ar[r] & {\L} \ar[r] & {\M} \ar[r] & {\N} \ar[r] &  0,
	}\quad\text{and}\quad
	\xymatrix@1{
		0 \ar[r] & {\N} \ar[r] & {\M'} \ar[r] & {\N'} \ar[r] &  0
	}
\end{equation*}
in $\sfi$ which correspond via $V_{\gs}(\cdot)$ to the following exacts sequences in $\rep _F(G_{\ul{\pi}})$:
\begin{equation*}
\xymatrix@1{
	0 \ar[r] & {\psi(V ')}\ar[r] & {V} \ar[r] & {V/ \psi(V ')} \ar[r] & 0
	},\ \text{and}\ 
	\xymatrix@1{
	0 \ar[r] & {\ker(\psi)} \ar[r] & {V'} \ar[r]^-{\psi} & {\psi(V')} \ar[r] &  0.
	}
\end{equation*}
\end{co}
\begin{proof}
We may and do select $G_{\ul{\pi}}$-stable $\O_F$-lattices $T\subseteq V$ and $T'\subseteq V'$
with $\psi(T')\subseteq T$ and $T/\psi(T')$ torsion-free.
Thanks to Proposition \ref{prop-rep-height r}, there exist $\M$ and $\M'$ in $\sfi$
with $T=T_{\gs}(\M)$ and $T'=T_{\gs}(\M')$, and we set
$M := \O_\E \otimes _{\gs_F} \M$ and $M':= \O_\E \otimes _{\gs_F} \M' $, and write
$\f : M\to M '$ for the unique morphism in $\Mod_{\O_{\E}}^{\varphi}$
with $\ul{T}(\f)= \psi |_{T'}$. Let $N ': = M' / \f (M)$ and write $\g: M' \to N '$ for the natural projection.
Writing $N:= \f(M) = \ker \g$, we then have exact sequences in $\Mod^{\varphi}_{\O_\E}$
\begin{equation*}
	\xymatrix@1{
		0 \ar[r] &{\ker(\f )}\ar[r] &{M} \ar[r]^-{\f} & {N} \ar[r] & 0
		}\quad\text{and}\quad
		\xymatrix@1{
			0 \ar[r] &{N} \ar[r] &{M'} \ar[r]^-{\g} &{N'} \ar[r] & 0
			}
\end{equation*}
which correspond, via $\ul{T}(\cdot)$, to the exact sequences in $\rep_{\O_F}(G_{\ul{\pi}})$
\begin{equation*}
	\xymatrix@1{
		0 \ar[r] & {\psi (T')}\ar[r] & {T} \ar[r] & {T/ \psi(T ')} \ar[r] &  0
		},\ \text{and}\ 
	\xymatrix@1{
		0 \ar[r] & {\ker(\psi)} \ar[r] & {T'} \ar[r]^-{\psi} & {\psi(T')}\ar[r] &  0.
		}
\end{equation*}
Since $N'$ corresponds to $T/\psi(T')$, which is torsion-free, it follows that $N'$ is also torsion free
and hence finite and free as an $\O_{\E}$-module.
Define $ \N := \ker (\g)\cap {\M'}$, the intersection taken inside of $M'$.
We claim that $\N$ is an object in $\t{Mod}^{\varphi, r} _{\gs_F}$.
First note that by \cite[B 1.3.5]{fo4}, the fact that $\M'$ has height $r$ implies that both $\g(\M')$ and $\N$ have height $r$, and we need only show that $\N$ is free over $\gs_F$.
To do this, it suffices by Lemma \ref{lem-digfree} to prove that
$\N = \N [1/ p] \cap N $ inside $\E\otimes_{\O_{\E}}N$,
or equivalently that $\N[1/p]\cap N \subseteq \N$.
For any $x \in \N[1/p] \cap N$, we have by the very definition of $\N$ that
$x \in \M'[1/p]\cap M'= \M'$. As  $x \in N =\ker\g$, we then have $x\in \ker\g \cap \M'=\N'$
as desired.
A similar argument shows that
$\L : = \ker(\f)\cap \M $ is a Kisin module in $\sfi$ as well.

Again using Lemma \ref{lem-digfree}, both $\wt{\N} := \f(\M)[1/p]\cap N $ and $\N' : = \g (\M ')[1/p ] \cap N'$ are objects of $\sfi$. As $\O_\E \otimes_{\gs_F} \wt{\N} = N = \O_\E\otimes _{\gs_F} \N $,
it follows from Proposition \ref{prop-fullfaith} that $\wt{\N} = \N$. We therefore have exact sequences
\begin{equation*}
	\xymatrix@C=11pt{
		0 \ar[r] &{\L[1/p]}\ar[r] &{\M[1/p]}\ar[r]^-{\f} &{\N[1/p]}\ar[r] & 0,
		}\quad\text{and}\quad
	\xymatrix@C=11pt{
		 0\ar[r] & {\N[1/p]} \ar[r] &{\M'[1/p]}\ar[r]^-{\g} &{\N'[1/p]}\ar[r] & 0
		 }.
\end{equation*}
Unfortunately, it need not be true in general that $\f(\M) = \N$ or $\g(\M') = \N'$.
To remedy this defect, we modify $\M$ and $\M'$ as follows.
Using the inclusion $\N\subseteq \M'$ and the above exact sequences, we
may select a $\gs_F [1/p]$-basis
$\fe_1 , \dots,  \fe_s, \fe_{s+1}, \dots , \fe_d $ of $\M[1/p]$
with the property that $\fe_1 , \dots,  \fe_s$ is an $\gs_F$-basis of $\L$
and $\fe_{s+1}':=\f(\fe_{s+1}), \dots , \fe_d':=\f(\fe_d)$ is an $\gs_F$-basis of $\N$.
We may further complete $\fe_{s+1}',\dots,\fe_d'$ to a $\gs_F[1/p]$-basis
$\fe_{s+1}' , \dots,  \fe_{d}', \fe'_{d+1}, \dots , \fe'_{d'} $ of $\M'$
with the property that $\fe'_{d+1}, \dots , \fe'_{d'}$ projects via $\g$ to an
$\gs_F$-basis of $\N'$.  We then have matrix equations
$$\varphi (\fe_1, \dots , \fe_d)= (\fe_1, \dots , \fe_d) \begin{pmatrix} A &  C \\ 0 & A' \end{pmatrix}$$
and
$$\varphi (\fe_{s+1}', \dots , \fe_d', \fe'_{d+1}, \dots , \fe' _{d'})= (\fe_{s+1}', \dots , \fe_d', \fe'_{d+1}, \dots , \fe' _{d'})  \begin{pmatrix} B &  D \\ 0 & B' \end{pmatrix},$$
where the entries of  $A, A', B , B'$ are in $\gs _F$, while the entries of $C$ and $D$ are in $\gs_F [1/p]$.
Let $m\ge 0$ be such that $p ^m C$ and $p^m D$ have all entries in $\gs_F$.
Replacing $\M $ by the $\gs_F$-submodule of $\M[1/p]$ generated over $\gs_F$ by
$p^{-m}\fe_1 , \dots , p^{-m}\fe_s , \fe_{s+1}, \dots ,  \fe_{d}$, and $\M '$ by  the $\gs_F$-submodule of
$\M'[1/p]$ generated by  $(\fe_{s+1}', \dots , \fe_d', p^m \fe'_{d+1}, \dots , p^m \fe' _{d'})$
does the trick.
\end{proof}

\section{Constructing Kisin modules from $F$-crystalline representations}\label{sec-4}
In this section, we associate to any $F$-crystalline representation a Kisin module in the sense
of \S\ref{subsec-3.2} and employ our construction to prove Theorems
\ref{thm-intro-1} and \ref{thm-intro-2}.
Throughout, and especially in \S \ref{subsec-3.1}, we make free use of many of the ideas of \cite{kisin2} and \cite{Kisin-Ren}. To surmount the difficulty that we do not in general have a natural $N_{\nabla}$-structure
(see the introduction), we will compare our modules over the Robba ring to those of Kisin's classical setting in
in \S  \ref{subsec-comparions}, which will allow us to descend these modules to the desired Kisin modules.
The proofs of our main results (Theorems \ref{thm-intro-1} and \ref{thm-intro-2}) occupies \S \ref{subsec-mainresults}.

\subsection{Generalities on $F$-crystalline representations}\label{FcrysReps}
 Let $V$ be  an $F$-linear representation of $G=G_K$ or of $G_{\ul{\pi}}$.
 We write $V^\v$ for the $F$-linear dual of $V$ with its natural $G$ or $G_{\ul{\pi}}$-action.
We warn the reader at the outset that our notational conventions regarding Fontaine's functors
are dual to the standard ones; we have chosen to depart from tradition here
as it will be more convenient to deal with the integral theory.

Let $V$ be an object of $\rep_F(G)$.
Then $D_{\dR}(V):=(V^\v \otimes_{\Q_p} B_{\dR})^{G}$ is naturally a module over the semilocal
ring $K_F:=K\otimes_{\Q_p} F$, so we have a decomposition
\begin{equation*}
	D_{\dR}(V)= \prod_{\m} D_{\dR}(V)_{\m}
\end{equation*}
with the product running over all maximal ideals of $K_F$.
We give each $D_{\dR}(V)_{\m}$ the filtration induced from that of $D_{\dR}(V)$,
and we denote by $\m_0$ the kernel of the natural map $K\otimes_{\Q_p} F\rightarrow K$
coming from the given inclusion $F\hookrightarrow K$ and multiplication.
Following \cite{Kisin-Ren}, we define:

\begin{definition}
	We say that $V\in \rep_{F}(G)$ is
	{\em $F$-crystalline}
	if it is crystalline $($when viewed as a $\Q_p$-linear $G$-representation$)$
	and the filtration on $D_{\dR}(V)_{\m}$ is trivial
	$(\Fil^j D_{\dR}(V)_{\m} = 0$ if $j>0$ and $\Fil^0 D_{\dR}(V)_{\m}=D_{\dR}(V)_{\m}$$)$
	unless $\m=\m_0$.  We write $\rep_F^{F\t{-cris}}(G)$ for the category of $F$-crystalline
	$F$-representations of $G$.
\end{definition}

We now wish to describe the category of $F$-crystalline $G$-representations in terms
of filtered $\varphi$-modules.  To do this, we define:

\begin{definition}
	Let $\MF_{F_0,K}^{\varphi}$ be the category of triples $(D,\varphi,\Fil^j D_{F_0,K} )$
	where $D$ is a finite dimensional $F_0$-vector space, $\varphi: D\rightarrow D$
	is a semilinear (over the $F$-linear extension $\varphi$ of the $p$-power Frobenius map $K_0\rightarrow K_0$)
	endomorphism whose linearization is an $F_0$-linear isomorphism, and
	$\Fil^j D_{F_0,K}$ is a separated and exhaustive descending filtration by $K$-subspaces on
	$D_{F_0,K}:=D\otimes_{F_0} K$.  Morphisms in this category are $\varphi$-compatible $F_0$-linear maps
	$D\rightarrow D'$ which are moreover filtration-compatible after applying $\otimes_{F_0} K$.
\end{definition}



Let $V$ be an $F$-crystalline $G$-representation with $F$-dimension $d$.
Then $$D:=D_{\cris}(V):=(V^\v\otimes_{\Q_p} B_{\cris})^{G}$$
is naturally a module over $F\otimes_{\Q_p} K_0$, equipped with a semilinear (over $1\otimes \sigma$
for $\sigma$ the $p$-power Frobenius automorphism of $K_0$)
Frobenius endomorphism $\varphi:D\rightarrow D$ which linearizes to an isomorphism.
By our assumption that $K_0\cap F = \Q_p$, the natural multiplication map $F\otimes_{\Q_p}K_0\rightarrow FK_0=:F_0$
is an isomorphism, so $D$ is an $F_0$-vector space which, as $V$ is crystalline as a $\Q_p$-representation,
has $K_0$-dimension $d[F:\Q_p]$, and so must have $F_0$-dimension $d$.

The natural injective map
$$D\otimes_{K_0} K =D_{\cris}(V)\otimes_{K_0}K\hookrightarrow D_{\dR}(V):=(V^\v\otimes_{\Q_p}B_{\dR})^{G}$$
is necessarily an isomorphism of $F_K:=F\otimes_{\Q_p} K$-modules, so since
$V$ is $F$-crystalline we have a direct sum decomposition of filtered $K$-vector spaces
$D\otimes_{K_0} K=\bigoplus_{\m} D_{K,\m}$, with $D_{K,\m}$ having trivial filtration unless $\m=\m_0$.
Noting the canonical identification
$$D_{K,\m_0} = D\otimes_{F_0} K =: D_{F_0,K},$$
we therefore obtain a filtration on $D_{F_0,K}$.
In this way we obtain an object
$$D_{\cris, F }(V):=(D,\varphi,\Fil^j D_{F_0,K})$$
of $\MF_{F_0,K}^{\varphi}$.

Conversely, if $D$ is any object of  $\MF_{F_0,K}^{\varphi}$, we define
\begin{equation*}
	V_{\cris, F }(D):= \Hom_{F_0,\varphi}(D,B_{\cris,F}^+) \cap
	\Hom_{K,\Fil^{\bullet}}(D_{F_0,K},B^{+}_{\cris}\otimes_{K_0} K)
\end{equation*}
with the intersection taken inside of $\Hom_{K}(D_{F_0,K},B^{+}_{\cris}\otimes_{K_0} K)$,
via the embedding
\begin{equation*}
	\xymatrix{
		{\Hom_{F_0,\varphi}(D,B_{\cris,F}^+)} \ar@{^{(}->}[r] & {\Hom_{K}(D_{F_0,K},B^{+}_{\cris}\otimes_{K_0} K)}
		}
\end{equation*}
that sends an $F_0$-linear map
$f:D\rightarrow B_{\cris,F}^+$ to its linear extension along $F_0\rightarrow K$.


\begin{prop}\label{prop-Dcirs & Vcris} Let $V \in \t{Rep}_F^{F\t{-\cris}}(G) $. Then $ V_{\cris, F}( D_{\cris, F} (V)) \simeq V$.
\end{prop}
\begin{proof}
Set $D= D_\cris(V):=( V ^\v \otimes_{\Q_p} B^+_\cris) ^{G}$ and put $D_K := K \otimes_{K_0} D$.
As it is well-known (e.g. \cite[\S5.3.7]{fo7}) that
$V\simeq V_\cris (D)$ as $F[G]$-modules, for
$$V_\cris (D) = \Hom_{K_0, \varphi} (D, B^+_\cris ) \cap \Hom_{K, \Fil} (D_K , B^+_\cris \otimes _{K_0} K),$$
it is enough to prove that $V_\cris (D) \simeq V_{\cris, F} (D)$ as $F[G]$-modules.
We will first construct an  $F$-linear isomorphism
\begin{equation}
	\xymatrix{
	{\iota: \Hom_{K_0} (D, B^+_\cris)} \ar[r]^-{\simeq} & {\Hom_{F_0} (D, B^+_{\cris, F}).}
	}\label{iotamap}
\end{equation} Writing $D_{F_0} = D \otimes_{K_0} F_0$, which is an $F_0 \otimes_{K_0} F_0$-module, we note that
$F_0 \simeq F_0 \otimes _{F_0} F_0$ is a subfield of $F_0 \otimes_{K_0} F_0$, so we may and do
regard $D \simeq D \otimes_{F_0} F_0$ as an $F_0$-subspace of $D_{F_0}$.
Thus, restricting homomorphisms from $D_{F_0}$ to the subspace $D$ gives a natural map of $F$-vector spaces
$\iota' : \Hom_{F_0} (D_{F_0}, B^+_{\cris, F}) \to \Hom_{F_0} (D, B^+_{\cris, F})$.
As $\Hom_{K_0} (D, B^+_\cris)$ is easily checked to be an $F$-subspace of $\Hom_{F_0} (D_{F_0} , B^+_{\cris, F})$, restricting to $\iota '$ to $\Hom_{K_0} (D, B^+_\cris)$ then gives the desired map (\ref{iotamap}).

To check that (\ref{iotamap}) is an isomorphism, we explicitly compute with bases:
Let $e_1 , \dots , e_d$ be an $F_0$-basis of $D$ and $\beta_1 , \dots , \beta_{e_F}$ a $K_0$-basis of
$F_0$. Any $x \in D$ can then be uniquely expressed as a linear combination
$x= \sum_{ij} a_{ij} \beta_j e_i$ with $a_{ij} \in K_0$, while  any $y  \in D_{F_0}$ admits a unique
representation  $y  = \sum_{i, j, l} a_{ijl} \beta_j e_i \otimes \beta_l {\rm  \  with \ } a_{ijl} \in K_0$.
The natural $F$-linear inclusion $D \hookrightarrow D_{F_0}$ induced by $F_0 \otimes_{F_0} F_0 \subset F_0 \otimes_{K_0} F_0$ carries $x\in D$ above to
$$x= \sum_{ij} a_{ij} \beta_j e_i \otimes \beta_j \in D_{F_0}.$$
In particular, if $h \in \Hom_{K_0} (D, B^+_\cris)$, then $h$ is uniquely determined by
the matrix $\{c_{ij}\}$ with $c_{ij} := h (\beta_j e_i) \in B^+_\cris $,  and
it follows from definitions that
$\iota  (f) (x) = \sum _{ij} a_{ij }c_{ij} \otimes \beta_j$ as an element of
$B^+_\cris \otimes _{K_0} F_0$.  From this explicit description of $\iota$, one checks easily that $\iota $ is
indeed an isomorphism.

 From the very definition of (\ref{iotamap}), one checks that $\iota$ induces an isomorphism
$$\Hom_{K_0, \varphi} (D, B^+_\cris ) \simeq \Hom_{F_0, \varphi} (D, B^+_{\cris,F}),$$
so to complete the proof it now remains to show
that for any $h \in \Hom_{K_0} (D, B ^+_\cris)$,
the scalar extension $h\otimes 1: D\otimes_{K_0} K\rightarrow B^+_{\cris}\otimes_{K_0} K$
is compatible with filtrations if and only if this is true of
$\iota(h)\otimes 1: D\otimes_{F_0} K\rightarrow B^+_{\cris,F}\otimes_{F_0} K$.
Observe that
the construction of the map (\ref{iotamap}) gives the following commutative diagram,
\begin{equation}
\xymatrix @C=45pt{ D_K \ar[r] ^- {  h\otimes_{K_0} 1} &    B^+_{\cris} \otimes_{K_0} K \\ D_{F_0, K} \ar@{^(->}[u]\ar[r]^-{ \iota(h) \otimes_{F_0} 1} &   B^+_{\cris, F} \otimes_{F_0} K \ar@{=}[u]}\label{fildiag}
\end{equation}
where we make the identification  $B^+_\cris \otimes_{K_0} K = B^+_{\cris, F} \otimes_{F_0} K$.
As $V$ is $F$-crystalline, we have
$\Fil ^i D_K = \Fil ^i D_{F_0, K}$ for $i \geq 1$ by definition, and it follows from this
and (\ref{fildiag}) that
$$( h \otimes_{K_0} 1) (\Fil ^i D_K) \subset \Fil ^i B^+_\cris \otimes_{K_0}K\  \iff\
( \iota(h) \otimes_{F_0} 1) (\Fil ^i D_{F_0, K}) \subset \Fil ^i B^+_{\cris,F} \otimes_{F_0} K,$$
which completes that proof of $V_{\cris, F} (D) \simeq V_\cris (D) \simeq V$ as $F[G]$-modules.
\end{proof}

Let $V$ be an $F$-linear representation of $G$. For each field embedding $\tau : F \to \ol{K}$, we define the set \emph{$\tau$-Hodge-Tate weights} of $V$:
$$\t{HT}_\tau (V):= \{i \in \Z | (V \otimes _{F, \tau } \C_K (-i))^{G} \not = \{0\}\},$$
where $\C_K$ is the $p$-adic completion of $\ol{K}$. It is easy to see that $V$ is $F$-crystalline if and only if $V$ is crystalline and
$\t{HT}_\tau (V) = \{0\}$ unless  $\tau$ is the trivial embedding $\tau_0 : F \subset K \subset \ol{K}$.  For
the remainder of this paper, we will \emph{fix} a nonnegative integer $r$ with the property that $\t{HT}_{\tau _0} (V) \subset \{0, \dots , r\} $, or equivalently, $\Fil ^{r+1} D_{F_0, K} = \{0\}$. We denote by  $\t{Rep} ^{F\t{-cris}, r} _{F} (G) $ the category of $F$-crystalline representations $V$ of $G$ with $\t{HT}_{\tau_0} (V) \subset \{0, \dots , r\}$.



\subsection{$\varphi$-modules over $\gO$}\label{subsec-3.1}
Recall that we equip $ F_0(\!(u)\!)$ with the Frobenius
endomorphism  $\varphi: F_0(\!(u)\!) \to F_0(\!(u)\!)$  which  acts
as the canonical Frobenius on $K_0$, acts trivially on $F$, and sends $u$ to $f(u)$. For any sub-interval
$I\subset [0,1)$, we write  $\gO_I$ for the subring of
$F_0 (\!(u)\!)$ consisting of those elements which converge for all $x \in \overline K$  with
$ |x| \in I$. For ease of notation, we put  $\gO= \gO_{[0, 1)}$ and
as before we set $K_{n} = K (\pi _{n-1})$. We denote by $\widehat \gs_n $ the completion of $K_{n+1} \otimes_{F_0} \gs_F  $ at the maximal ideal $(u - \pi_n )$. The ring $\wgs_n $ is equipped with its $(u  - \pi _n)$-adic filtration,
which extends to a filtration on the quotient field $\wgs_n [1/(u - \pi _n)].$ Note that for any $n$ we have natural maps of $F_0$-algebras $\gs_F[1/p] \hookrightarrow \gO \hookrightarrow \wgs_n$  given by sending $u$ to $u$, where the first map has dense image. We will write $\varphi_W:  \gs_F \to \gs_F$ for the $\O_F[\![u]\!]$-linear map which acts on $W(k)$ via the canonical lift of Frobenius, and by $\varphi _{\gs/W }:  \gs _F \to \gs _F $
the $\O_{F_0}$-linear map which sends $u$ to $f(u)$. Let $c_0 = E(0) \in F_0 $ and set
$$\lambda  : = \prod \limits _{n = 0}^\infty \varphi ^n(E(u)/c_0 ) \in \gO .$$
A \emph{$\varphi$-module over $\gO$} is a finite free $\gO$-module $\CM$ equipped with a semi-linear endomorphism $\varphi_\CM : \CM \to \CM$. We say that $\CM$ is of \emph{finite $E(u)$-height} if the cokernel of the $\gO$-linear map $1 \otimes \varphi_\CM : \varphi ^* \CM \to \CM$ is killed by $E(u)^r$ for some $r$, and we
write $\t{Mod}^{\varphi,r}_\gO $ for the category of $\varphi$-modules over $\gO$ of $E(u)$-height $r$.
Note that scalar extension along the inclusion $\gs_F\hookrightarrow \gO$ gives a functor
$\sfi\rightarrow\t{Mod}^{\varphi,r}_\gO$ from Kisin modules of height $r$ to
$\varphi$-modules over $\gO$ of $E(u)$-height $r$.

Now let $V\in \Rep$ be any $F$-crystalline representation, and let $D= D_{\cris, F} (V)$
be the corresponding filtered $\varphi$-module.
We functorially associate to $D$ an $\gO$-module $\CM (D)$ as follows:
For each nonnegative integer $n$, let $\iota _n $  be the composite map:
\begin{equation}
\xymatrix@C=35pt{\gO \otimes _{F_0} D    \ar[r]^-{\varphi_W^{-n}\otimes \varphi _D^{-n}} &  \gO \otimes _{F_0} D  \ar[r] & \wgs_n \otimes _{F_0}D = \wgs_n  \otimes _{K} D_{F_0, K}, }
\end{equation}
where the second morphism is induced by the canonical inclusion $\gO \to \wgs_n$.
We again write $\iota_n$ for the canonical extension
$$ \iota _n : \ \gO[1/\lambda] \otimes _{F_0} D \longrightarrow \wgs _n [1/(u - \pi _n)] \otimes_{K} D_{F_0 , K},$$ and we define
$$\CM (D) : = \left \{x \in \gO[1/\lambda] \otimes_{F_0} D\  |  \ \iota_n (x) \in \Fil ^0 \left (\wgs _n [1/(u - \pi _n)] \otimes_{K} D_{F_0 , K} \right ), \  \forall n \geq 0 \right \}.$$

\begin{prop}\label{prop-Kisinconstruct} $\CM (D)$ is a $\varphi$-module over $\gO$ of $E(u)$-height $r$.
\end{prop}
\begin{proof} This is Lemma 1.2.2 in \cite{kisin2} (also see Lemma (2.2.1) in \cite{Kisin-Ren}) with the following minor modifications:  first note that we only discuss crystalline representation here, so we do not need the ``logarithm element" $\ell_n$ which occurs in Kisin's classical setting (strictly speaking, we do not know how to construct $\ell_n$
in our general setting).
Likewise, we may replace $\D_0 : = (\gO [\ell_n] \otimes _{K_0} D) ^{N=0}$ in the proof of \cite[1.2.2]{kisin2}
with $\D_0 = \gO \otimes_{F_0} D$ throughout. In the classical setting, Kisin showed that $\CM(D)$ also has an
$N_\nabla$-structure, which we entirely ignore here (once again, we do not know how to construct $N_\nabla$ in general).
This is of no harm, as the proof of Lemma 1.2.2 does not use the $N_\nabla$-structure of $\gO$ in any way.
Finally, we note that Lemma 1.1.4 of \cite{kisin2}, which plays an important role in the proof of
\cite[1.2.2]{kisin2}, is well-known for $\gO$-modules in our more general context.\footnote{Indeed,
Kisin's proof of \cite[1.1.4]{kisin2} relies on
\S 4 of Berger's paper \cite{Ber0} as well as results of
Lazard \cite[\S7--8]{Lazard} and Lemma 2.4.1 of \cite{Kedlaya},
while the required facts in \cite{Ber0} build on Lazard's work in a natural way.
But \cite{Lazard} already deals in the generality we need, as does Kedlaya \cite{Kedlaya}.
Thus, one checks that all the proofs of the results needed to establish
\cite[1.1.4]{kisin2} (as well as Kisin's argument itself) carry over {\em mutatis mutandis}
to our more general situation.
}
\end{proof}

As above, let us write $\D_0 := \gO \otimes _{F_0} D$. We record here the following useful facts,
which arise out of (our adaptation of) Kisin's proof of Proposition \ref{prop-Kisinconstruct}:
 \begin{enumerate}
 \item $ \D_0 \subset \CM \subset \lambda^{-r} \D_0$.
 \item $\iota_0$ induces an isomorphism of $\wgs_0$-modules $$ \wgs_0 \otimes_{\gO} \CM(D) \simeq \sum_{j \geq 0} (u - \pi) ^{-j} \wgs_0 \otimes_K \Fil ^j D_{F_0, K} = \sum_{j \geq 0} E(u) ^{-j}\wgs_0  \otimes_K \Fil ^j D_{F_0, K}.$$

 \end{enumerate}

 Consider now the obvious inclusions $D \inj \D_0 \subset \CM(D)$.
As Frobenius induces a linear isomorphism $\varphi^* D \simeq D$, we obtain a linear isomorphism
$\varphi ^*\D_0 \simeq \D_0$ and hence an injection $\xi : \D_0  \to \varphi ^*(\CM(D))$.
Defining $\gO_\alpha: = \gs_F [\![ \frac {E(u)^p }{\varpi}]\!][1/p]$, one checks that
$\gO_\alpha = \gs_F [\![\frac {u ^{e_0 p}}{\varpi}]\!][1/p]$ and  that $\gO\subset \gO_\alpha \subset \gO_{[0, |\pi|^{1/p})}$.

\begin{lemma}\label{lem-phisection} 
The map $\xi_\alpha : = \gO_\alpha   \otimes_{\gO} \xi : \gO_\alpha \otimes_{F_0} D \to \gO_\alpha \otimes_{\gO} \varphi ^*\CM(D)$ is an isomorphism.
\end{lemma}
\begin{proof} Since $ \D_0 \subset \CM (D) \subset \lambda^{-r} \D_0$ and $\D_0 \simeq \varphi ^* \D_0$, we have $\D_0 \subset \varphi ^*(
\CM (D)) \subset \varphi(\lambda)^{-r} \D_0$. It is easy to check that $\varphi (\lambda)$ is a unit in $\gO_\alpha$, and it follows that $\xi_\alpha$ is an isomorphism.
\end{proof}

For an object $\CM \in {\rm Mod} ^{\varphi, r}_{\gO}$, we define a decreasing filtration on $\varphi ^*\CM$ by:
\begin{equation}\label{eqn-definefil}
	\Fil ^i (\varphi ^* \CM ) := \{x \in \varphi ^*\CM | (1 \otimes \varphi)(x) \in E(u)^i \CM \}.
\end{equation}
On the other hand, using the evident inclusion $\gO_\alpha \subset \gO_{[0 , |\pi|^{1/p})} \subset \wgs_0 $
we obtain a canonical injection $\gO_\alpha  \otimes_{F_0} D  \inj \wgs_0  \otimes _K D_{F_0, K}$,
which allows us to equip $\gO_\alpha   \otimes _\gO \D_0$
with the natural subspace filtration, using the tensor product filtration on $\wgs_0  \otimes _K D_{F_0, K}$.


 \begin{lemma}\label{lem-fil}
 The inverse isomorphism
 $$
 \xymatrix{
 	{\xi'_ {\alpha}:  \gO_\alpha \otimes_{\gO}   \varphi ^*\CM (D)} \ar[r]^-{\simeq}_-{{(\xi_\alpha )}^{-1 }} & {\gO_\alpha \otimes_{F_0}  D}
	}$$
of Lemma $\ref{lem-phisection}$ is compatible with filtrations and Frobenius.

 \end{lemma}
\begin{proof}Clearly, $\xi'_\alpha$ is compatible with Frobenius.
To prove that $\xi'_\alpha$ is filtration compatible,
we use two facts recorded after Proposition \ref{prop-Kisinconstruct}.
As noted above,  $\varphi(\lambda)$ is a unit in $\wgs_0$, so
the first fact implies that the injection $\xi : \D_0 \simeq \varphi ^*\D_0 \hookrightarrow \varphi^* \CM(D) $ is an isomorphism after tensoring with $\wgs_0$. Put $\widehat \D_0 := \wgs_0 \otimes_\gO \D_0$ and define an auxillary filtration on $\widehat \D_0$ by
$$\wt{\Fil}^i \widehat  \D_0 : = \widehat \D_0  \cap E(u)^i (\wgs_0 \otimes_\gO \CM(D)). $$
From the very definition (\ref{eqn-definefil}), it is clear that
$1\otimes \xi: \widehat{\D}_0\simeq \wgs_0 \otimes_\gO  \varphi ^*\CM(D)$
carries $\wt{\Fil}^ i \widehat  \D_0$ isomorphically onto
$\Fil ^i (\wgs_0 \otimes_\gO  \varphi ^*\CM(D) )$.
On the other hand, the second fact recorded above
implies that an element $d \in \widehat \D_0$ lies in $E(u)^i( \wgs_0 \otimes_\gO \CM(D))$ if and only if $\iota_0(d) \in \Fil ^i (\wgs_0 \otimes _K D_{F_0, K})$, from which it follows that
$\wt \Fil ^i \widehat \D_0 = \Fil^i \widehat \D_0$.
Hence $\xi'_\alpha$ is indeed compatible with filtrations.
\end{proof}

For simplicity, let us put $\CM : = \CM(D)$.
It follows from Lemma \ref{lem-fil} that the isomorphism $\xi_{\alpha}'$
specializes to give a natural identification of $\varphi$-modules $D\simeq\varphi^*\CM/ u \varphi ^* \CM$
as well as a natural identification of filtered $K$-vector spaces
$D_{F_0, K}\simeq \varphi ^*\CM/ E(u) \varphi ^*\CM $.
Writing $\psi_{\pi}$ for the composite mapping
$$\psi_\pi : \varphi ^*\CM \twoheadrightarrow \varphi ^*\CM/ E(u) \varphi ^*\CM \simeq D_{F_0, K},$$
we therefore obtain:

\begin{co}\label{co-comp fil} The map  $\psi_\pi: \varphi^*\CM(D)  \to D_{F_0, K}$ is compatible with the
given filtrations.
\end{co}

\begin{remark}
In the classical situation where $F=\Q_p$ and $f(u)= u ^p$,
to {\em any} object $\CM$ of  $\t{Mod}_{\gO}^{\varphi, r}$, Kisin functorially associates a filtered $\varphi$-module
via $D(\CM):= \varphi ^*\CM/ u \varphi ^*\CM $ with  $\Fil ^ i(D(\CM)_K) := \psi_\pi (\Fil ^i \varphi^*\CM ) $.
That this is possible rests crucially on the existence of a unique
$\varphi$-equivariant isomorphism
$$\xi_\alpha : \gO_\alpha \otimes_{F_0} D(\CM)  \simeq  \gO_\alpha \otimes_{\gO} \varphi^* \CM$$
reducing modulo $u$ to the given identification $D(\CM)=\varphi^*\CM/u\varphi^*\CM$,
which is Lemma 1.2.6 of \cite{kisin2}.
For more general $F$ and $f(u)=u^p + \cdots + a_1u$, we are only able to
construct such a map $\xi_\alpha$ under the restriction $\varpi ^{r+1}|a_1$ in $\O_F$;  see Lemma \ref{lem-phisection 2}.

\end{remark}

To conclude this section, we record the following further consequence of Lemma \ref{lem-fil}:
Setting $\wt B_\alpha  : = W(R)[\![\frac{E(u)^p}{\varpi}]\!] [ 1/p] \subset B^+_{\cris, F }$,
one checks that the subspace filtration $\{\Fil^n \wt B_\alpha\}_n$ coincides with
the filtration $\{E(u)^n \wt B_\alpha\}_n$. As $\gs_F \subset W(R)_F$, we have
a canonical inclusion $\gO_\alpha \subset \wt B_\alpha $, and the map $\xi_{\alpha}$ of Lemma \ref{lem-phisection}
induces a natural isomorphism
\begin{equation}
	\xi'_{\wt B_\alpha} : \wt B_\alpha  \otimes_{\gO} \varphi ^*\CM  \simeq \wt B_\alpha\otimes_{F_0} D.
	\label{Balphaiso}
\end{equation}	
As the inclusion $ \wgs_0 \subset B^+_\dR$ is compatible with filtrations,  we deduce:

\begin{co}\label{co-comparison}
  	The isomorphism $\xi'_ {\wt B_\alpha}$  is compatible with Frobenius and filtrations.
\end{co}



 \subsection{Results of classical settings}
 For future reference, we now recall the main results in Kisin's classical situation, where $F=\Q_p$
 and $f(u)=u^p$. Throughout this section, we fix a choice $\ul{\pi}:=\{\pi_n\}_n$ of $p$-power compatible
 roots of a fixed uniformizer $\pi=\pi_0$ in $K$, and set $K_{\infty}:=K_{\ul{\pi}}$
 and $G_{\infty}:=G_{\ul{\pi}}$.
 The following summarizes the main results in this setting:
 \begin{theorem}[\cite{kisin2}]\label{thm-classical}Let $V$ be a $\Q_p$-valued
 crystalline representation of $G$ with Hodge-Tate weights in $\{0, \dots , r\}$ and
  $T \subset V$ a $G_\infty$-stable $\Z_p$-lattice. Then:
 \begin{enumerate}
 \item There exists a unique Kisin module $\M$  so that $T_\gs(\M) \simeq T$ as $\Z_p[G_{\infty}]$-modules.
 \item  If $D= D_\cris (V)$ is the associated filtered $\varphi$-module, then
	$\CM(D) \simeq \gO \otimes_\gs \M$ as $\varphi$-modules.
 \end{enumerate}
 \end{theorem}

 \begin{proof} These  are  the main results of \cite{kisin2} restricted to crystalline representations. \end{proof}

 Now let $F$ be an arbitrary extension of $\Q_p$ contained in $K$ and let $V$ be
an $F$-crystalline representation and $T$ a $G_{\infty}$-stable $\O_F$-lattice in $V$.
Viewing $V$ as a crystalline $\Q_p$-valued representation and $T$ as a $G_{\infty}$-stable
$\Z_p$-module, by Theorem \ref{thm-classical} there is a unique (classical) Kisin module $\M$ attached to $T$,
which is of finite $\wt{E}(u)$-height, for $\wt{E}(u)$ the minimal polynomial of
$\pi$ over $K_0$ (we write $\wt{E}(u)$ to distinguish this polynomial from $E(u)$,
which by definition is the minimal polynomial of $\pi$ over $F_0=FK_0$).
The additional $\O_F$-structure on $T$ is reflected on the classical Kisin module in the following way:

\begin{co}\label{cor-classical to F}
 	The classical Kisin module $\M$ is
 	naturally a finite and free $\gs_F$-module and as such has $E(u)$-height $r$.
\end{co}

 \begin{proof}   Proposition 3.4 of \cite{GLS} shows that $\M$ is naturally a finite and free $\gs_F$-module
 (see also the proof of \cite[Prop. 1.6.4]{kisin4}). Factor $\wt{E}(u)$
 in $\O_{F_0}[u]$ as $\wt{E}(u) = E_1(u) \cdots E_{e_F}(u)$ with $E_1(u) = E(u)$, and for each $i$ write
 $\wgs_{E_i}$ for the completion of the localization of $\gs_F$ at the ideal $(E_i)$.
 We must prove that the injective map $1\otimes \varphi: \varphi^* \M \to \M$ has cokernel killed by a power of $E=E_1$.
 To do this, it suffices to prove that the scalar extension
 \begin{equation}
 	\xymatrix@C=35pt{
 		{\varphi_i^{\sharp}: \wgs_{E_i} \otimes_{\gs_F} \varphi^*\M}
		\ar[r]^-{1\otimes (1\otimes\varphi)} & {\wgs_{E_i}\otimes_{\gs_F} \M}
		}\label{localmap}
 \end{equation}
 of  $1\otimes \varphi$ along $\gs_F\rightarrow \wgs_{E_i}$ is an isomorphism when $i>1$.
 Writing $\CM : = \CM (D)$, we recall that the map $\psi_\pi : \varphi^*\CM \to D_K$ is compatible with filtrations
 thanks to Corollary  \ref{co-comp fil}, from which it follows that the map
	$$\xymatrix{
		{\psi_\pi|_{\varphi^*\M} : \varphi^*\M} \ar[r] & {\varphi^*\M/ \wt E(u)\varphi^*\M} \ar@{^{(}->}[r] & {D_K}}$$
is also filtration-compatible.
As $V$ is $F$-crystalline, for any $i> 1$ we have
$\Fil ^j  D_{K, \m_i}  = 0$ for all $j \geq 1$, where $\m_i$ is the maximal ideal of
$F \otimes_{\Q_p} K $ corresponding to $E_i(u)$, and it follows that
$\Fil^1 \varphi^*\M \subset E_i \varphi ^*\M$ for all $i>1$.
We then claim that for $i>1$ the map $\ol{1\otimes \varphi}:\varphi ^*\M / (E_i) \to  \M/ (E_i)$
induced from $1\otimes \varphi$ by reduction modulo $E_i$ is {\em injective}.
To see this, observe that if $x\in \varphi^*\M$ has $(1\otimes \varphi)(x)=E_i m$ for $m\in \M$,
then writing $y:=\prod _{j \not = i} E_j x$, we have
$(1\otimes\varphi)(y)  = \wt{E} m $ so that $y\in \Fil^1 \varphi^*\M$ by the very definition of the filtration on
$\varphi^*\M$.  By what we have seen above, we then have $y\in E_i \varphi^*\M$,
so since $E_i$ is coprime to $\prod_{j\neq i} E_j$, we obtain $x\in E_i\varphi^*\M$ as claimed.
Now both $\varphi^*\M $ and $\M$ are $\gs_F$-free of the same rank, so as $\ol{1\otimes \varphi}$
is injective, we see that $\Q_p \otimes_{\Z_p} \ol{1\otimes\varphi}$
is an isomorphism for $i>1$. But this map coincides with the
map $\ol{\varphi}^{\sharp}_i$ obtained from (\ref{localmap}) by reduction modulo $E_i$,
so it follows that $\varphi_i^{\sharp}$ is an isomorphism as well, as desired.
\end{proof}

\subsection{Comparing constructions}\label{subsec-comparions}

Let us first recall some standard facts on the Robba ring as in \cite{kisin2}. For finer details of the Robba ring
$\R$ and its subring $\R^b $, we refer to \S 2 (in particular \S2.3) of \cite{Kedlaya}, noting that
several different notations are commonly used (in particular, we advise the
reader that $\R ^b = \E ^\dag = \Gamma^{k((t))}_{{\rm con}}[1/p] $).
The {\em Robba ring} is defined as
$$\R := \lim_{s\to 1 ^{-}} \gO_{(s, 1)}$$
and comes equipped with a Frobenius endomorphism, which is induced by the canonical maps
$\varphi: \gO_{(s,1)} \to \gO_{(s^{1/p}, 1) }$.
Writing $\gO^b _{(s,1 )} \subset \gO _{(s,1 )}$ for the subring of functions which are bounded,
we also define the {\em bounded Robba ring}:
$$\R^b := \lim _{s \to 1 ^-} \gO^b _{(s,1 )},$$
which is naturally a Frobenius-stable subring of $\R$.
Finally, we put
$$\O_{\R^b}:= \{f = \sum _{n \in \Z} a_n u ^n \in \R^b , a_n \in \O_{F_0},  \forall n \in \Z\};$$
this is a Henselian discrete valuation ring with uniformizer $\varpi$ and residue field $k (\!(u )\!)$.  One checks that the fraction field of
$\O_{\R^b}$ is $\R^b$, justifying our notation; in particular, $\R^b$ is a field.
Note that $\gO$ is canonically a Frobenius-stable subring of $\R$.

By definition, \emph{a $\varphi$-module over $\R$} is a finite dimensional $\R$-module
$\CM$ together with a  $\varphi$-semilinear endomorphism $\varphi_\CM : \CM \to \CM$
whose linearization $ 1\otimes \varphi : \varphi ^*\CM \to \CM $ is an isomorphism.
One checks that $E(u)\in \gO$ is a unit in $\R$, so that scalar extension along $\gO\hookrightarrow\R$
gives a functor from $\varphi$-modules over $\gO$ to $\varphi$-modules over $\R$.
A $\varphi$-module $\CM$ over $\R$ is \emph{\'etale} if $\CM$ admits a basis
with the property that the corresponding matrix of $\varphi_\CM$ lies in $\GL_d (\O_{\R^b})$;
by a slight abuse of terminology, we will say that a $\varphi$-module over $\gO$
is \'etale if its scalar extension to $\R$ is.
Our main result of this subsection is the following:

\begin{theorem}\label{thm-maintech} Let $V \in \Rep$ and $D := D_{\cris, F} (V)$ the corresponding filtered $\varphi$-module.  Writing $\CM(D)$ for the $\varphi$-module over $\gO$ attached to $D$ as in $\S \ref{subsec-3.1}$, we have:
\begin{enumerate}
\item $\CM (D)$ is \'etale;
\item There exists a Kisin module $\M\in \sfi$ such that $\gO \otimes_{\gs_F} \M\simeq \CM (D)$.

\end{enumerate}
\end{theorem}

First note that there is a canonical inclusion $\gs_F\hookrightarrow \O_{\R^b}$,
so that (2) implies (1). It follows that the above theorem is true in the classical setting of Kisin
by Theorem \ref{thm-classical}. In what follows, we will reduce the general case of Theorem \ref{thm-maintech}
to the known instance of it in the classical setting.  To ease notation,
we will adorn various objects with a superscript of ``$c$'' to signify that they are objects
in the classical setting.  We likewise abbreviate $\CM := \CM(D)$ and $\CM ^c := \CM ^c (D)$.
We note that $\gO^c_\alpha  \subset \wt{B}_\alpha$, as
$E(u ^c)$ is another generator of $\Fil ^1 W(R)$ so $E(u ^c) = \mu E(u)$ for some
$\mu \in W(R) ^\times$ thanks to Lemma \ref{lem-generater}.

By Corollary \ref{co-comparison}, the  $\wt{B}_\alpha$-linear isomorphism $\xi'_{\wt{B}_\alpha} : \wt B_\alpha  \otimes_{\gO} \varphi ^*\CM \simeq \wt B_\alpha \otimes _{F_0} D $ is compatible with Frobenius and filtration.
The key point is that the Frobenius and filtration on $\wt B_\alpha \otimes _{F_0}D$ are {\em canonical}
(recall that the filtration on $\wt B_\alpha\otimes _{F_0} D$ is induced by the inclusion $\wt B_\alpha\otimes_{F_0} D  \inj B^+_\dr \otimes_ K D_{F_0, K}$) and are independent of the choice of $\varphi (u)= f(u)$.
We therefore have a natural isomorphism
\begin{equation}\label{eqn-eta}
\wt \xi:\ \ \wt B_\alpha\otimes_\gO \varphi ^* \CM  \simeq \wt B_\alpha  \otimes_{\gO^c} \varphi ^* \CM ^c,
\end{equation}
that is compatible Frobenius and filtration.
\begin{lemma}\label{lem-iso-crystal}
There is a $\wt B_\alpha $-linear and Frobenius-compatible isomorphism
$$\eta : \wt B_\alpha \otimes _{\gO} \CM \simeq  \wt B_\alpha  \otimes_{\gO ^c} \CM^c.$$
\end{lemma}

\begin{proof}
Choose an $\gO$-basis $e_1 , \dots ,e _d$ of $\CM$, and let  $A \in \t{M}_d (\gO)$
be the corresponding matrix of Frobenius, so
$(\varphi(e_1), \dots, \varphi(e_d))= (e_1, \dots , e_d)A$.
We write $\fe _i :  = 1 \otimes e_i \in \varphi^* \CM$ for the induced basis of
$\varphi^*\CM$.
Using the definition of $\Fil ^i \varphi^* \CM$, it is not difficult to see that there
is a matrix $B\in  \t{M}_d(\gO)$ satisfying $AB = BA = E(u)^r I_d$
and with the property that $\Fil^r \varphi ^* \CM$ is generated by
$(\alpha _1,\dots,  \alpha _d):= (\fe_1 , \dots , \fe _d)B$.
As promised, we denote by $\fe^c_i$, $A^c$, {\em etc.} the objects in the classical setting
corresponding to a choice $e_1^c , \dots ,e^c_d$ of $\gO^c$-basis of $\CM^c$.
Let $X \in \GL_d (\wt B_\alpha)$ be the matrix determined by the requirement
$ \wt \xi (\fe_1, \dots , \fe_d) = (\fe_1^c , \dots, \fe_d ^c) X.$
As $\wt \xi$ is compatible with both Frobenius and filtrations, we find that
\begin{equation*}
	\wt \xi \circ  \varphi(\fe_1 , \dots, \fe_d) = (\fe_1^c,  \dots, \fe_d ^c ) X \varphi (A)= \varphi\circ \wt \xi(\fe_1, \dots, \fe_d)= (\fe_1 ^c \dots, \fe_d ^c) \varphi(A^c)   \varphi (X)
\end{equation*}
and there exists a matrix $Y \in \GL_d (\wt B_\alpha)$ with
\begin{equation*}
	\wt \xi(\alpha _1 , \dots, \alpha _d) = (\alpha _1 ^c , \dots, \alpha _d ^c) Y.
\end{equation*}
We conclude that  $X\varphi(A) = \varphi (A^c) \varphi(X)$ and $XB = B^c Y$. Since $\wt B_\alpha$ is an integral domain, the facts that $B = E(u)^r A^{-1}$ and $B^c = E(u^c )^r {(A^c)}^{-1}$ imply that $A^c X E(u)^r= E(u ^c) ^r  Y A$.
Due to Lemma \ref{lem-generater}, both $E(u)$ and $E(u ^c)$ are generators of $\Fil ^1 W(R)_F$,  so
$\mu: = E(u ^c)/ E(u)$ is a unit in $W(R)_F$. We therefore have the relation
$A^c X= \mu^r YA$. Combining this with the equality $X\varphi (A) = \varphi(A^c) \varphi (X)$
yields $X = \varphi (\mu^r Y)$, and we deduce $A^c \varphi (\mu^r Y) =  \mu^r Y A$.
Defining a $\wt B_\alpha $-linear map
$$\xymatrix@1{
	{\eta: \wt B_\alpha \otimes_\gO \CM}\ar[r] & {\wt B_\alpha \otimes _{\gO^c} \CM^c}
	}$$
by the requirement $(\eta(e_1), \dots , \eta(e_d))= (e^c_1, \dots, e^c _d)\mu^r Y$,
we then see that $\eta$ provided the claimed Frobenius-compatible isomorphism.
\end{proof}


Recall that Lemma \ref{lem-ringresults} gives inclusions $\gO \subset \wt B^+_{\rig, F} \subset \wt B_\alpha $.

\begin{co}\label{co-descent}
The isomorphism $\eta$
of Lemma $\ref{lem-iso-crystal}$ descends to a $\wt B^+_{\rig, F}$-linear and Frobenius-compatible isomorphism
$$\eta_\rig:  \wt B^{+}_{\rig, F } \otimes _{\gO} \CM \simeq  \wt B^{+}_{\rig, F }  \otimes_{\gO ^c}  \CM ^c .$$
\end{co}

\begin{proof}
We use the notation of the proof of Lemma \ref{lem-iso-crystal}.
Set $ Z := \mu ^rY \in \GL_d (\wt B_\alpha)$, so that $\eta (e_1, \dots , e_d) = (e_1, \dots, e_d) Z$,
and note that $A^c \varphi (Z) = ZA$ as $\eta$ is compatible with Frobenius.
To prove the corollary, it suffices to show that both $Z$ and $Z ^{-1}$ have entries
in $\wt B_{\rig, F}^+ $. We will show that $Z \in \t{M}_d (\wt B_{\rig , F}^+)$; the proof
of the corresponding fact for $Z^{-1}$ is similar and is left to the reader.
By the definition of $\wt B_{\rig , F}^+$, it
suffices to show that for any $m$, there exists $Z_m \in \t{M}_d (B^+_{\tmax , F})$ with $\varphi ^m (Z_m) = Z$,
which we prove by induction on $m$.  The base case $m = 0$ is obvious, as $\wt B_\alpha  \subset B^+_{\max, F}$. Now suppose that $Z_m$ exists, and note that  from the equality
$A^c \varphi (Z) = Z A$ we obtain $E(u)^r Z = A^c \varphi (Z) B$.
We may write $A^c = \varphi ^{m+1}(A_{m+1})$ and $B = \varphi^{m+1} (B_{m+1})$ thanks to Lemma \ref{lem-ringresults},
and we then have $E(u)^r Z = \varphi ^{m+1} (A_{m+1} Z _m B_{m+1})$.
Finally, Lemma \ref{lem-inBrig} implies the existence of $Z_{m+1}$.
\end{proof}

We can now prove Theorem \ref{thm-maintech}.

\begin{proof}[Proof of Theorem $\ref{thm-maintech}$]
We first prove that $\CM$ is \'etale, and to do so we will freely use the results
and notation of \cite{Kedlaya}.  By the main theorem of \cite{Kedlaya}, $\CM$ is \'etale if and only if $\CM$ is  \emph{pure of slope 0}. Hence $\CM^c$ is pure of slope $0$ thanks Theorem \ref{thm-classical}
and our remarks immediately following Theorem \ref{thm-maintech}.
Since the slope filtration of $\CM$ does not change after tensoring with the ring
$\Gamma ^{\rm alg}_{\rm an, con}$ constructed in \cite{Kedlaya}, it is enough to show that
\begin{equation*}
	\Gamma ^{\rm alg}_{\rm an, con} \otimes _{\gO} \CM \simeq \Gamma ^{\rm alg}_{\rm an, con} \otimes _{\gO^c} \CM^c.
\end{equation*}
as $\varphi$-modules over $\Gamma ^{\rm alg}_{\rm an, con}$, and to do this
it is enough thanks to Lemma \ref{lem-iso-crystal} to prove
that $\wt B^+_{\rig, F} \subset  \Gamma ^{\rm alg}_{\rm an, con}$.
But this follows from
Berger's construction \cite[\S 2]{Ber0} (strictly speaking, \cite[\S2]{Ber0} deals only with
the case $F= \Q_p$, but see \cite[\S3]{Ber-PGMLAV} for the general case), as he proves that
$\wt B^+_{\rig, F} \subset B^\dag_{\rig, F } = \Gamma ^{\rm alg}_{\rm an, con}$.
It follows that $\CM$ is \'etale.

Now the proof that $\CM:=\CM(D)$ admits a descent to a Kisin module $\M$ is exactly the same as the proof of
Lemma 1.3.13 in \cite{kisin2}, so we just sketch the highlights.
As $\CM$ is \'etale, there exists a finite free $\O_{\R^b}$-module $\CN$ with Frobenius endomorphism $\varphi_\CN$
satisfying
\begin{equation}
	\R \otimes_{\O_{\R^b}} \CN \simeq \R \otimes_\gO \CM=:\CM_\R\label{keyisoms}
\end{equation}
Proposition 6.5 in \cite{Kedlaya} shows that it is possible to select an $\R$-basis of $\CM_\R$
whose $\R^b$-span is exactly $\CN[1/p]$ and whose $\gO$-span is $\CM$ via the identifications (\ref{keyisoms}).
Define $\CM ^b\subseteq \CM$ to be the $\gs _F[1/p]$-span of this basis.
Since $\gs_F [1/ p] = \R^b \cap \gO$, we have the intrinsic description
$\CM ^b= \CM \cap \CN[1/p]$;  in particular, $\CM ^b$ is $\varphi$-stable and of $E(u)$-height $r$.
Let $\M':  = \CM ^b \cap \CN$ and $\M : = (\O_{\R^b } \otimes_{\gs_F} \M ') \cap \M'[1/p] \subset \CN [1/p]$. Then $\M$ is a finite and $\varphi$-stable $\gs_F$-submodule of $\CN [1/p]$. It follows from the structure theorem of finite
$\gs_F$-modules \cite[Proposition 1.2.4]{fo4} that $\M$ is in fact finite and {\em free} over $\gs_F$.
To see that $\M$ has $E(u)$-height $r$, it suffices to check that $\det (\varphi_\M)=E(u)^s w$ for some
$w \in \gs_F^\times$.
But  $\M [1/ p] = \CM^b$ and $\CM^b$ is of finite  $E(u)$-height, so $\det (\varphi_\M) = p ^m E(u)^s w $ for some $w \in \gs _F ^\times$; as $\CM$ is pure of slope $0$ (equivalently, $\det (\varphi _\M)\in \O_{\R^b}^{\times}$),
we must in fact have $m=0$.
\end{proof}

\subsection{Full-faithfulness of restriction}\label{subsec-mainresults}

Let $V$ be an $F$-crystalline representation of $G=G_K$ with Hodge--Tate weights in $\{0,\ldots,r\}$,
and let $\M$ be a Kisin module associated to $V$ by Theorem  \ref{thm-maintech} (2).

\begin{prop} \label{prop-Ginftycompatible}
	There exists a natural $\O_F$-linear injection
	$$\iota: T_{\gs_F}(\M) \inj V \simeq V_{\cris, F}(D)$$
	that is moreover $G_{\ul{\pi}}$-equivariant. In particular, $V_{\gs} (\M) \simeq V$
	as $F[G_{\ul{\pi}}]$-modules.
\end{prop}

\begin{proof} As usual, we put $\CM: = \CM (D)$.
As $\CM \simeq \gO \otimes_{\gs_F} \M$, we have a natural injection
$$\iota ' :  T _{\gs_F}(\M) = \Hom_{\gs_F, \varphi} (\M , \gs _F ^\ur ) \inj \Hom_{\gO, \varphi, \Fil} (\varphi ^* \CM , B ^+_{\cris, F }),$$
uniquely determined by the requirement that for any homomorphism $h: \M\rightarrow \gs_F^\ur$,
the value of $\iota'(h)$ on any simple tensor $s\otimes m\in  \gO \otimes_{\varphi,\gs_F} \M\simeq \varphi^*\CM $
is given by
$$\iota'(h) (s \otimes m) = s \varphi (h(m));$$
Using the fact that $E(u) \in \Fil ^1 W(R)_F$, one checks that this really does define a
filtration-compatible $\gO$-linear homomorphism $\iota'(h):\varphi^*\CM\rightarrow B_{\cris,F}^+$.

On the other hand, the isomorphism
$\xi_\alpha :  \gO_\alpha \otimes_{F_0}  D \xrightarrow{\sim}\gO_\alpha \otimes_{\gO}   \varphi ^*\CM$
of Lemma \ref{lem-phisection} induces  a  natural injection
\begin{equation*}
\xymatrix@C=30pt{
	{\Hom_{\gO, \varphi, \Fil} (\varphi ^* \CM , B ^+_{\cris, F })}  \ar@{^{(}->}[r]^-{h\mapsto 1\otimes h}
	& {\Hom_{\gO_\alpha, \varphi, \Fil}
	(\gO_\alpha 	\otimes_\gO\varphi ^* \CM , B ^+_{\cris, F })}\ar[d]_-{\simeq}^-{h\mapsto h\circ\xi_{\alpha}} \\
	 &  {\Hom_{\gO_\alpha, \varphi, \Fil} (\gO_\alpha \otimes _{F_0} D  , B ^+_{\cris, F })}
	}
\end{equation*}
thanks to Lemma \ref{lem-fil}, and we claim that $\Hom_{\gO_\alpha, \varphi, \Fil} (\gO_\alpha  \otimes _{F_0} D  , B ^+_{\cris, F }) = V_{\cris, F }(D)$.
By definition, we have
$$  V_{\cris, F }(D) =  \Hom_{F_0 , \varphi} (D, B^+_{\cris, F }) \cap \Hom_{K, \Fil} (D_{F_0, K}, B^+_\dR),$$
and it is clear that $\Hom_{\gO_\alpha, \varphi} (\gO _\alpha \otimes _{F_0} D  , B ^+_{\cris, F })= \Hom_{F_0, \varphi} (D  , B ^+_{\cris, F })$. Since
the injection $\gO_\alpha \otimes _\gO D \inj \wgs _0 \otimes D_{F_0, K}$ is compatible with filtrations by the very construction of the filtration on $\gO_\alpha \otimes _{F_0} D$, we conclude that
\begin{eqnarray*}
\Hom_{F_0} (D, B^+_{\cris, F }) \cap  \Hom_{\gO_\alpha, \Fil} (\gO_\alpha \otimes _{F_0} D  , B ^+_{\cris, F }) \\ \simeq \Hom_{F_0} (D, B^+_{\cris, F}) \cap \Hom_{K, \Fil} (D_{F_0, K}, B^+_\dR),
\end{eqnarray*}
which gives our claim.

Putting everything together, we obtain a  natural injection $\iota :  T_{\gs_F} (\M) \inj V_{\cris, F} (D)$ which is visibly compatible with the given $G_{\ul{\pi}}$-actions.
 \end{proof}

Combining Theorem \ref{thm-maintech} and Proposition \ref{prop-rep-height r}, we deduce
Theorem \ref{thm-intro-1}:

\begin{co}\label{co-Ginftylattice} Let $V$ be an object of $\Rep$ and $T\subset V$  a $G_{\ul{\pi}}$-stable $\O_F$-lattice. Then there exists a Kisin module $\M$ of $E(u)$-height $r$ with $T_{\gs_F} (\M) \simeq T$.
\end{co}

\begin{remark}\label{rem-nonunique}
It is an important point that in our general setup, the Kisin module
$\M$ may {\em not} be unique for a given $T$, contrary to the classical situation.
Indeed, the ``cyclotomic case" of Example \ref{ex-cyclotomic} is a prototypical instance of such non-uniqueness:
let $T$ be the trivial character, $\M= \gs_F$  the trivial rank-1 Kisin module, and $\M' = u \gs_F \subset \M $. Since $\varphi (u) = E(u) u$, one sees that  $\M'$ is also a Kisin module and $T_{\gs_F} (\M) = T_{\gs _F} (\M') =T$.

\end{remark}

We now prove Theorem \ref{thm-intro-2}:

\begin{theorem}\label{thm-main1}
Assume that $\varphi ^n (f(u)/u)$ is not a power of $E(u)$ for any $n \geq 0$ and
that $v_F(a_1)>r$.
Then the natural functor $\Rep \rightarrow \t{Rep}_F (G_{\ul{\pi}})$ given by
$V \rightsquigarrow V|_{G_{\ul{\pi}}}$ is fully faithfull.
\end{theorem}

\begin{remark}
	We suspect that the theorem remains valid if we drop the assumption that
	$\varphi ^n (f(u)/u)$ is not a power of $E(u)$ for any $n \geq 0$. When $r =1$, we
	will show that this is indeed the case in the next section.
\end{remark}

In order to prove Theorem \ref{thm-main1}, we prepare several preliminary results.
In what follows, we keep our running notation with $f(u)=u^p + a_{p-1}u^{p-1}+ \cdots + a_1 u$,
and we assume throughout that $\varpi^{r+1}| a_1$ in $\O_F$.

Let $\M \in \t{Mod}_{\gs_F}^{\varphi, r}$ and set
$M : = \varphi ^* \M / u \varphi ^*\M$. 
\begin{lemma}\label{lem-phisection 2}
	There exists a unique  $\varphi$-equivariant isomorphism
\begin{equation*}
	\xymatrix{
		{\xi_\alpha:  \gO_\alpha \otimes_{\O_{F_0}}M}\ar[r]^-{\simeq}  & {\gO_\alpha \otimes_{\gs_F} \varphi ^* \M}
		}
\end{equation*}
whose reduction modulo $u$ is the identity on $M$.
\end{lemma}

\begin{proof} The proof is similar to that of Proposition 2.4.1 in \cite{liu6}, and is motivated by
the proof of Lemma 1.2.6 in \cite{kisin2}.
Choose an $\gs_F$-basis $\fe_1 , \dots , \fe_d$ of $\M$ and let $A \in \Md (\gs_F)$
be the resulting matrix of $\varphi$; {\em i.e.}~ $(\varphi (\fe_1), \dots , \varphi (\fe_d)) = (\fe_1, \dots, \fe_d) A $. Then $e_i := 1 \otimes \fe_i$ forms a basis of $\varphi ^*\M$ and we have $ (\varphi(e_1), \dots , \varphi(e_d) ) = (e_1, \dots, e_d ) \varphi (A)$. Put $A_0 : = A \bmod u$ and $\bar e_i : = e_i \bmod u$.
Then we have $ (\varphi(\bar e_1) , \dots , \varphi(\bar e_d)) = (\bar e_1, \dots, \bar e_d) \varphi (A_0)$.
If the map $\xi_\alpha$ of the Lemma exists,  then writing $f_i :=  {\xi}_\alpha (\bar e_i) \in \gO_\alpha \otimes _{\gs_F} \varphi ^* \M $ and denoting by $Y \in \GL_d (\gO_\alpha)$ the matrix with
$(f_1 , \dots, f_d ) = (e_1, \dots , e_d)Y$, we necessarily have $Y \equiv I _d \bmod u$ and
\begin{equation}\label{eqn-xi_I exists}
 Y \varphi (A_0) = \varphi (A) \varphi(Y).
\end{equation}
Conversely, if (\ref{eqn-xi_I exists}) has a solution $Y \in \GL_d (\gO_\alpha)$ satisfying $Y \equiv I_d \mod u$, then we may define ${\xi}_\alpha$ by ${\xi}_\alpha(\bar e_1 , \dots , \bar e_d ) = (e_1, \dots ,e _d) Y$.
Thus, it remains to solve  Equation \eqref{eqn-xi_I exists}. Put
\begin{equation}\label{eqn-Yn}
 Y_n : = \varphi (A) \cdots \varphi ^n (A) \varphi ^n (A^{-1}_0) \cdots \varphi (A_0^{-1}).
 \end{equation}
We claim that the sequence $\{Y_n\}_n$ converges to a matrix $Y \in \Md (\gO_\alpha)$.
To see this, note that there exists $B_0\in \GL_d (\gO_\alpha)$ with $A_0 B_0 = \varpi ^r  I_d$ since $\M $ has height $r$. It follows that $ A A_0^{-1} = I_d + \frac{u}{\varpi^r } Z$ for $Z \in \Md (\gs_F).$ Thus,
$$ Y_n = Y_{n-1} + \varphi (A) \cdots \varphi^{n-1} (A) \frac{\varphi ^n (u)}{\varpi ^{rn}}\varphi ^n (Z) \varphi^{n- 1 }( B_0) \dots  \varphi (B_0),$$
so to prove our claim it suffices to show that $\varphi^n(u)/ \varpi ^{rn} $ converges to $0$ in $\gO_\alpha$, which
is the content of Lemma \ref{lem-f(n)}.

To prove that $Y$ is invertible, we compute the determinant of $Y$.
Writing $d:=\t{rank}_{\gs_F}\M$, we observe that
since $\wedge ^d  \M$ has finite $E(u)$-height, we have $\det (A) = \gamma E(u)^m$ for some
$\gamma \in \gs _F^\times$. It follows that $\det (\varphi (A) \varphi (A_0^{-1}))=\gamma ' ( \frac{\varphi (E(u))}{\varpi}) ^m$ for some $\gamma ' \in \gs_F ^\times$. One then checks that  $\varphi (E(u))/ \varpi$ is a unit in $\gO_\alpha$, and hence that $\det (Y)$ is a unit in $\gO_\alpha$ so $Y$ is invertible as desired.

Finally, we prove that the solution $Y$ to (\ref{eqn-xi_I exists}) that we have constructed is unique.
Suppose that equation \eqref{eqn-xi_I exists} admits two solutions $Y, Y'$ satisfying $Y, Y' \equiv I_d \bmod u$.
Then their difference is also a solution $Y-Y' = u Z $ for $Z \in \Md (\gO_\alpha)$.
Equation \eqref{eqn-xi_I exists} then implies that
 \begin{eqnarray*}
 Y-Y' = \varphi (A) \varphi (Y) \varphi (A_0^{-1}) &= &  \varphi (A) \cdots \varphi ^n (A)  \varphi ^n (Y) \varphi ^n (A_0^{-1})  \cdots \varphi (A^{-1}_0)\\  &=&  \varphi (A) \cdots \varphi ^n (A) \frac{ \varphi ^n (u)}{\varpi^{rn} } \varphi^n  (Z) \varphi ^n (B_0)  \cdots \varphi (B_0)
\end{eqnarray*}
for all $n$. As $\varphi ^n (u)/ \varpi ^{rn} $ converges to $0$ in $\gO_\alpha$, we conslude that $Y=Y'$
as desired.
\end{proof}

For $\M \in \t{Mod}_{\gs_F}^{\varphi, r}$, let us write $D(\M) = \xi_\alpha (M[1/p]) \subset \gO_\alpha \otimes_{\gs_F}\varphi^* \M$ for the image of $M[1/p]$ under the map of Lemma \ref{lem-phisection 2}.
If $\M$ is a Kisin module associated to some $F$-crystalline $G$-representation $V$
with $D:= D_{\cris, F} (V)$
({\em i.e.}~$\gO\otimes_{\gs_F} \M \simeq \CM (D)$), then
by the very construction of $\CM(D)$ there is a
natural $\varphi$-compatible inclusion $D\simeq \varphi^*D\hookrightarrow \varphi^*\CM(D)$
which, thanks to Lemma \ref{lem-phisection},
becomes an isomorphism after tensoring over $\gO$ with $\gO_\alpha$.
Recalling that $\gO_\alpha \otimes_{\gs_F} \M \simeq \CM(D)$, we therefore
have a $\varphi$-equivariant inclusion
\begin{equation}
		\xymatrix@C=30pt{
			{D}\ar@{^{(}->}[r]^-{d\mapsto 1\otimes d} &
			{\gO_{\alpha}\otimes_{F_0} D} \ar[r]_-{\simeq}^-{\ref{lem-phisection}} &
			{\gO_{\alpha} \otimes_{\gO} \varphi^*\CM(D) \simeq \gO_{\alpha} \otimes_{\gs_F} \varphi^*\M}
		}\label{HowDisSub}
	\end{equation}
via which we view $D$ as a $\varphi$-stable $F_0$-subspace of $\gO_{\alpha} \otimes_{\gs_F} \varphi^*\M$.

\begin{co}\label{co-D(M)=D-cris}
	For $V \in \t{Rep}_F^{F\t{-\cris}, r}(G)$, with $D:= D_{\cris, F} (V)$ and $\M \in \sfi$
	a Kisin module attached to $D$, we have $D(\M)=D$ inside $\gO_\alpha \otimes_{\gs} \varphi^*\M$.
\end{co}
\begin{proof}
	The reduction of (\ref{HowDisSub}) modulo $u$ is the $\varphi$-compatible isomorphism
	\begin{equation*}
			{D \simeq (\gO_{\alpha} \otimes_{\gs_F} \varphi^*\M)\bmod (u) \simeq	
			M[1/p]}. 
	\end{equation*}
	Since the map $\xi_{\alpha}$ of Lemma \ref{lem-phisection 2} reduces to the identity modulo $u$,
	we conclude that {\em both} $D$ and $D(\M)$ inside $\gO_\alpha \otimes _{\gs_F} \varphi ^*\M$
	are $\varphi$-equivariant liftings of $M[1/p]$; the uniqueness aspect of Lemma
\ref{lem-phisection 2} then forces $D = D(\M)$ as claimed.
\end{proof}

It follows from Corollary \ref{co-D(M)=D-cris} that the map
$\xi_\alpha$ of Lemma \ref{lem-phisection 2} coincides with that of Lemma \ref{lem-phisection}, which justifies our notation.

Recall that $V_{\gs}(\M) = T_{\gs} (\M)[1/p]$ for $\M \in \sfi$.

\begin{lemma} \label{lem-inside}
	Let  $\gf : \M \to \M'$ be any morphism of height-$r$ Kisin modules, and let $\gf_{\alpha}$
	be the scalar extension
	$\gf_\alpha : \gO _\alpha \otimes_{\gs_F} \varphi ^* \M \to \gO _\alpha \otimes_{\gs_F} \varphi ^*\M'$
	of $\varphi^*\gf$ along $\gs_F\rightarrow \gO_{\alpha}$. Then
	$\gf_\alpha (D(\M)) \subset D(\M')$.
\end{lemma}
\begin{proof}
	Put $V = V_{\gs}(\M)$ and  $ V' = V_{\gs}(\M')$ and denote by $\psi = V_{\gs} (\gf) : V' \to V$ the induced map.
	By Proposition \ref{prop-rep-height r}, we can modify $\M$ and $\M'$ (inside $\M[1/p]$ and $\M' [1/p]$,
	respectively) so that $\gf$ may be decomposed by two exact sequences inside $\sfi$:
	\begin{equation*}
		\xymatrix@1{
			 0 \ar[r]& {\L} \ar[r] &{\M}\ar[r]^-{\gf}& {\N} \ar[r] & 0
			 }\quad\text{and}\quad
		\xymatrix@1{
			 0 \ar[r]& {\N} \ar[r]& {\M'}\ar[r]& {\N'} \ar[r]& 0
			 },
		\end{equation*}
where $\N = \gf(\M)$.
From the explicit construction of  $\xi_\alpha$ in Lemma \ref{lem-phisection 2}
(in particular, from the explicit construction of $Y$ in \eqref{eqn-Yn}),
we obtain the following exact sequences
\begin{equation*}
	\xymatrix@C=15pt{
	0 \ar[r]& {D(\L)} \ar[r]& {D(\M)} \ar[r]^-{\gf_\alpha } & {D(\N)}\ar[r]& 0
	}\quad\text{and}\quad
	\xymatrix@C=15pt{
		 0 \ar[r]& {D(\N)} \ar[r]& {D(\M')} \ar[r]& {D(\N')} \ar[r]&  0
		 },
\end{equation*}
which shows that  $\gf_\alpha (D(\M) ) \subset D(\M ')$ as claimed.
\end{proof}


\begin{proof}[Proof of Theorem $\ref{thm-main1}$]
	Let $V'$, $V$ be two objects of $\Rep$, and set  $D= D_{\cris, F} (V)$ and $D' = D_{\cris,F} (V')$.
	Suppose that there exists $h : V _{\cris, F} (D' ) \to V_{\cris, F} (D)$ that is $G_{\ul{\pi}}$-equivariant.
	 By Corollary \ref{co-Ginftylattice}, there exist  $G_{\ul{\pi}}$-stable $\O_F$-lattices  $T $ and  $T '$
	 inside $V_{\cris, F} (D)$ and $V_{\cris, F }(D')$, respectively, with $h(T')\subseteq T$,
	 and objects $\M$ and $\M '$ of $\sfi$ such that $T_{\gs}(\M) \simeq T $ and $T _{\gs} (\M ') \simeq T'$
	 via the map $\iota$ of Proposition \ref{prop-Ginftycompatible}. By Proposition \ref{prop-fullfaith}, there exists 	
	  a map $\gf : \M \to \M ' $ in $\sfi$ with $V _{\gs _F} (\gf) \simeq h$.
	  We may therefore realize $h$ as the composite
$$\xymatrix{V_{\cris ,F } (D' ) \ar[r]^-{\iota^{-1}}_-{\sim} &  V_{\gs} (\M')[1/p] \ar[r] ^{V_{\gs}(\gf)}  & V_{\gs} (\M)[\frac 1 p ] \ar[r] ^-{\iota}_-\sim & V_{\cris, F} (D)}, $$
where $\iota$ is constructed using the isomorphism $\xi'_\alpha : \gO_\alpha \otimes_{\gs_F} \varphi^* \M \simeq \gO_\alpha \otimes D$ of Lemma \ref{lem-fil}.  By Lemma \ref{lem-inside} and Corollary \ref{co-D(M)=D-cris}, the map $\gf_\alpha: \gO_\alpha \otimes_{\gs_F} \varphi^* \M  \to \gO_\alpha \otimes_{\gs_F} \varphi^* \M'$ induced by $\gf$
carries $D$ to $D'$, so for any  $a \in V_{\cris, F}(D') \subset \Hom_{F_0, \varphi} (D, B^+_{\cris,F})$
we have $h (a)= a\circ \bar \gf  \in V_{\cris, F } (D') \subset \Hom_{F_0, \varphi} (D' , B^+_{\cris , F})$ where
we write $\bar \gf :D \to D'$ for the map $\gf_\alpha |_{D}$.
It follows at once that $h$ is compatible with the action of $G=G_K$, as desired.
\end{proof}

We note that Theorem  \ref{thm-main1} is false if we replace ``$\t{Rep}^{F\t{-cris}}_{F} (G)$" with
``$\t{Rep}^{\Q_p\t{-cris}}_{F}(G)$," as the following example shows:

\begin{example}\label{ex-counter}
	Consider again the situation in Example \ref{ex-cyclotomic}, where
	$K = \Q_p(\zeta_p)$ with $\pi = \zeta_p -1$ and $\varphi (u) = f(u) = (1+ u)^p -1 $ and
	$K_{\ul{\pi}} = \bigcup_{n \geq 1} \Q_p (\zeta_{p ^n})$.
	Let $F= \Q_p$. Then the assumption of Theorem \ref{thm-main1} is not satisfied as $a_1 = p$, and
	the restriction functor $\t{Rep}^{\Q_p\t{-cris}}_{F}(G)\rightsquigarrow \rep_F (G_\infty)$ induced by $V \mapsto V|_{G_{\ul{\pi}}}$ is visibly {\em not} fully faithful: letting $\chi$ denote the $p$-adic cyclotomic character,
	 we have $\chi |_{G_{\ul{\pi}}} = 1 |_{G_{\ul{\pi}}}$, but $\chi\not\simeq 1$ as $G$-representations.
	 On the other hand, if $F=K$ then we easily check that the assumptions of Theorem \ref{thm-main1} are satisfied.
	 Of course, there is no contradiction here as $\chi$ is \emph{not} an $F$-crystalline representation
	 because $\t{HT}_\tau(\chi) =1$ for \emph{all} $\tau$.
\end{example}

\section{$F$-Barsotti-Tate groups}\label{Sec-BT}

Recall that by an {\em $F$-Barsotti--Tate group} over $\O_K$, we mean
a Barsotti--Tate group over $\O_K$ whose $p$-adic Tate module is an $F$-crystalline representation
of $G:=G_K$.
In this section, we prove that the category of $F$-Barsotti--Tate groups over $\O_K$
is (anti)equivalent to the category of height-1 Kisin modules:

\begin{theorem} \label{thm-F-BT} Assume $v_F(a_1)>1$.  Then there is an $($anti$)$equivalence of
categories between the category of Kisin modules of height $1$ and the category of $F$-Barsotti-Tate groups.
\end{theorem}

Using well-known results of Breuil, Kisin, Raynaud, and Tate, one shows as in
\cite[Theorem 2.2.1]{liu-BT} that the $p$-adic Tate module gives an equivalence
between thet category of $F$-Barsotti--Tate groups over $\O_K$
and the category $\t{Rep}_{\O_F} ^{F\t{-cris},1 } (G)$ of $G$-stable $\O_F$-lattices inside $F$-crystalline representations with Hodge-Tate weights in $\{0, 1\}$.
Thus, to prove Theorem \ref{thm-F-BT}
we must construct an (anti)equivalence between $\t{Mod}_{\gs_F}^{\varphi,1}$ and  $\t{Rep}_{\O_F} ^{ F\t{-cris},1 } (G)$. In what follows, we show that for each $\M \in \t{Mod}_{\gs_F}^{\varphi,1}$
the natural $G_{\ul{\pi}}$-action on $T_{\gs} (\M)$ can be functorially extended to
to a $G$-action such that $T_{\gs}(\M) \in \t{Rep}_{\O_F} ^{ F\t{-cris},1 } (G)$.
This construction will provide a contravariant functor $ T_{\gs} : \t{Mod}_{\gs_F}^{\varphi,1}\rightarrow  \t{Rep}_{\O_F} ^{F\t{-cris}, 1 } (G) $ that we will then prove is an (anti)equivalence.

\subsection{A natural $G$-action on $T_{\gs}(\M)$.}\label{subsec-6.1}

Fix a Kisin module $\M$ of height 1.  In this subsection,
we will construct a natural $G$-action on $T_{\gs}(\M)$ which extends the given action
of $G_{\ul{\pi}}$.  The key input to this construction is:

\begin{lemma}\label{lem-extend Ginfty action}
	There exists a unique $W(R)_F$-semilinear
	$G$-action on $W(R)_F \otimes_{\gs_F} \M$ that commutes with $\varphi$ and satisfies
\begin{enumerate}
	\item If $g \in G_{\ul{\pi}}$ and $m \in \M$ then $g(1\otimes m)= 1\otimes m$;
	\item If $ m \in \M$ then  $1\otimes (g(m)-m) \in \gt_FI ^+ (W(R)_F \otimes_{\gs_F}\M )$.
\end{enumerate}
\end{lemma}

Here, we remind the reader that $\gt_F \in W(R)_F$, constructed in Example \ref{ex-dim1},
satisfies $\varphi(\gt_F)=E(u)\gt_F$ and $\gt_F\not\equiv 0\bmod \varpi$.

\begin{proof}
	Fix an $\gs_F$-basis $\fe_1 , \dots , \fe_d$ of $\M $ and let $A$ be the resulting matrix of Frobenius,
	so $ (\varphi(\fe_1), \dots , \varphi(\fe_d)) = (\fe_1, \dots, \fe_d) A$.
	Supposing that the required $G$-action exists, for any $g \in G$
	we have a matrix $X_g \in \t{M}_d (W(R)_F)$ with
	$(g\fe_1, \dots, g\fe_d)= (\fe_1, \dots, \fe_d) X_g$, and the
	requirement that $g$ and $\varphi$ commute
	is equivalent the matrix equation
\begin{equation}\label{eqn-existence of X_g}
  X_g g(A) = A \varphi (X_g).
\end{equation}
	We claim that for each $g\in G$, equation \eqref{eqn-existence of X_g}
	has a \emph{unique} solution $X_g$ satisfying the condition  $X _g- I_d \in \Md(\gt_F I ^+).$
	Granting this for a moment, it is easy to see that the Lemma follows once we check
	that $g\mapsto X_g$ really does define an action of $G$, which is equivlant to
	the cocycle condition $X_\sigma \sigma(X_\tau)=X_{\sigma\tau}$
	for all $\sigma, \tau\in G$.  But $X_\sigma \sigma(X_\tau)$
	and $X_{\sigma\tau}$ are visibly two solutions to $X\sigma\tau(A)=A\varphi(X)$,
	and the condition $X-I_d\in \Md(\gt_F I ^+)$ holds for $X=X_{\sigma\tau}$
	by our claim and for $X=X_{\sigma}\sigma X_{\tau}$ thanks to Lemma \ref{lem-stable}.
	Thus, the uniqueness aspect of our claim gives $X_\sigma \sigma(X_\tau)=X_{\sigma\tau}$,
	as desired.
	
	It remains to prove our claim.	Let us first dispense with the uniqueness aspect.
	Suppose that for some $g\in G$, equation (\ref{eqn-existence of X_g})
	has two solutions $X_1, X_2$ satisfying $X_i- I_d \in \Md(\gt_F I ^+)$ for $i=1,2$.
	Then their difference is a solution as well, and has the form $X_1-X_2=\gt_F Z$ for some $Z \in \Md (I^+)$.
	Equation (\ref{eqn-existence of X_g}) then takes the shape
	\begin{equation}
		\gt_F Z g(A) = A \varphi (\gt_F Z),\label{Zeqn}
	\end{equation}
	and we will show that this forces $Z=0$, giving uniqueness.
	Writing $\bar Z := Z \bmod \varpi \in \Md(R)$, note that
	it suffices to prove that $\bar Z = 0$: indeed, if $Z=\varpi Z_1$ for some $Z_1\in \Md(W(R)_F)$,
	the necessarily $Z_1\in \Md(I^+)$ is another solution to (\ref{Zeqn}),
	so boot-strapping the argument one gets $Z\in \cap_{n\ge 1} \varpi^n W(R)_F=\{0\}$.
	Now since $\M$ has height 1, there exists a matrix $B\in \Md(\gs_F)$ with $AB=E(u) I_d$.
	On the other hand, we have $\varphi (\gt _F) = E(u) \gt_F$ as noted above,
	so it follows from (\ref{Zeqn}) that there exists
	a matrix $C \in \Md (W(R)_F)$ with $ Z = A \varphi (Z) C$.
	Reducing modulo $\varpi$ gives a matrix equation
	$\bar Z = \bar A \varphi (\bar Z) \bar C$ in $\Md(R)$.
	If $\bar Z \not = 0 $, then there exists an  entry $z$, say, of  $\bar Z$ which has minimal valuation.
	On the other hand, as $Z \in \Md (I ^+)$, we must have $v_R(z) > 0$.
	But the minimal possible valuation of entries in $\bar A \varphi (\bar Z) \bar C$ is $p v_R(z) > v_R(z)$,
	which is a contradiction.  Thus $\bar Z=0$, settling uniqueness.

	Finally, let us prove the existence of $X_g$ solving (\ref{eqn-existence of X_g}) for each $g\in G$.
	Defining
	\begin{equation}
		X_n :=A  \varphi (A) \cdots\varphi ^n (A) \varphi ^n (g(A^{-1})) \cdots \varphi (g(A^{-1}))  g(A^{-1}),
		\label{Xndef}
	\end{equation}
	it suffices to prove the following:
\begin{enumerate}
	\item $X_n \in \Md (W(R)_F)$ for all $n$;
	\item $X_n- I_d \in \Md (\gt_F  I ^+)$ for all $n$;
	\item $X_n$ converges as $n\rightarrow\infty$.
\end{enumerate}
	For (1) and (2), we argue by induction on $n$.
	When $n = 0$, by definition we have $X_0 = A g (A^{-1})= g ( g ^{-1} (A) A^{-1})$.
	On the other hand, thanks to Lemma \ref{lem-z in I plus.}, we may write $g ^{-1} A = A + \varphi (\gt_F ) C$ for some $C \in \Md(I ^+)$,
	which gives
	$$g ^{-1}(A) A^{-1} =  I_d + \varphi (\gt_F) C A^{-1} = I_d +  \gt_F C E(u) A^{-1}= I_d + \gt_F CB$$
	thus proving (1) and (2) in the base case $n=0$.

 Now suppose we have proved $X_n = I_d + \gt _F C_n $ with $C_n \in \Md (I^{+})$, and
	let us show that $X_{n +1}$ satisfies the same equation for some $C_{n+1}\in \Md (I^{+})$.
	Writing $A g(A^{-1}) = I_d + \gt_F  C_0 $ with $C_0 \in \Md (I ^+)$, we have
$$
X_{n +1} = X_n + A  \varphi (A) \cdots\varphi ^n (A)\varphi^{n+1}(\gt_F) \varphi ^{n+1} (C_0)   \varphi ^n (g(A^{-1})) \cdots \varphi (g(A)^{-1})(g(A)^{-1}).
$$
Now $E(u) g(A^{-1}) \in \Md (W(R)_F)$ as $g(E(u)) = \mu_g E(u)$ for some unit $\mu_g \in W(R)_F$,
and we have
$\varphi^{n+1}(\gt_F) = \varphi ^n (E(u)) \cdots \varphi (E(u)) E(u) \gt_F$.
We conclude that the matrix
$$\varphi^{n+1}(\gt_F) \varphi ^n (g(A^{-1})) \cdots \varphi (g(A)^{-1}) g(A)^{-1}$$ lies
in $ \Md (\gt_F W(R)_F)$, which gives $X_{n +1} \in \Md (W(R)_F )$ and $X_{n+1} - I_d \in \Md (\gt_F  I ^+)$
as desired.  By construction, we then have $X_{n+1}-X_n =D_n \cdot \varphi^{n+1}(C_0)\cdot D'_n$,
for some matrices $D_n, D'_n  \in \Md(W(R)_F)$ and since $\varphi ^{n+1} (C_0)$ converges to $0$ in $W(R)_F$,
we conclude that $X_n$ converges, which gives (3) and completes the proof.
\end{proof}


\begin{co}\label{co-G-action} The natural $G_{\ul{\pi}}$-action on $T_{\gs} (\M)$ can be functorially
extended to an action of  $G$. In particular, $T_{\gs}$ extends to a contravariant functor from
$\t{Mod}_{\gs_F}^{\varphi,1}$ to  $\t{Rep}_{\O_F}(G) $.
\end{co}

\begin{proof} By Lemma \ref{lem-Tgsbasic} (2), we have an isomorphism of $\O_F[G_{\ul{\pi}}]$-modules
\begin{equation}
	 T_{\gs} (\M) \simeq \Hom_{\gs_F, \varphi} (\M , W(R)_F) \simeq \Hom_{W(R)_F, \varphi} (W(R)_F \otimes_{\gs _F} \M , W(R)_F).  \label{Tisom}
\end{equation}
Thanks to Lemma \ref{lem-extend Ginfty action}, we have an action of $G$ on $W(R)_F \otimes_{\gs _F} \M $
that extends the given action of $G_{\ul{\pi}}$, so the final term in (\ref{Tisom})
has an action of $G$ given by
$$(g \circ h) (x) = g (h(g ^{-1} (x))), \ \forall g \in G, \ \forall h \in \Hom_{W(R)_F, \varphi} (W(R)_F \otimes_{\gs _F} \M , W(R)_F)$$
and one checks easily that this action extends the given action of $G_{\ul{\pi}}$ on $T_{\gs}(\M)$.

It remains to prove that $T_{\gs}$ is a  functor.
So suppose that $h : \M \to \M '$ is a map in $\Mod ^{\varphi, 1}_{\gs_F}$ and let us check that
the induced map $T_{\gs} (h) : T_{\gs} (\M ' ) \to T_{\gs} (\M)$ is indeed a map of $\O_F[G]$-modules.
To do this, using (\ref{Tisom}), it suffices to show that the map
$$1\otimes h : W(R)_F \otimes _{\gs_F} \M \to W(R)_F \otimes_{\gs _F} \M' $$ is $G$-equivariant, that is,
that $(1\otimes h) \circ g = g \circ (1\otimes h) $ for all $g \in G$.
Choose  $\gs_F$-bases $\fe_1, \dots, \fe_d$ and $\fe'_1, \dots, \fe'_{d'}$ of $\M$ and $\M'$, respectively,
and let $A$ and $A'$ be the corresponding matrices of Frobenius,
so $(\varphi(\fe_1), \dots, \varphi(\fe_d)) = (\fe_1, \dots, \fe_d) A$ and
$(\varphi(\fe'_1), \dots, \varphi(\fe'_{d'})) = (\fe'_1, \dots, \fe'_{d'}) A'$.
Letting $Z$ be the $d'\times d$-matrix with entries in $\gs_F$ determined by the relation
$h(\fe_1, \dots , \fe_d) = (\fe'_1, \dots, \fe'_{d'}) Z$,
we seek to prove that
$g \circ (1\otimes h)(\fe_1, \dots, \fe_d) = (1\otimes h) \circ g (\fe_1, \dots, \fe_d)$,
which is equivalent to the matrix equation
$$ X'_g g (Z) =Z X_g ,  $$
where $X_g$ (resp. $X'_g$) is the matrix constructed in the proof of Lemma \ref{lem-extend Ginfty action}
giving the action of $g$ on $\M$ (resp. $\M'$).
By construction, $X_g = \lim \limits_{n \to \infty} X_n$, and similarly for $X_g'$, so it suffices to check that
$X'_n  g (Z) =Z X_n$ for all $n$.  From the definition (\ref{Xndef}) of $X_n$ and $X_n'$, this amounts to the equation
\begin{equation}
A' \cdots \varphi ^n (A') \varphi ^n (g ({A'}^{-1})) \cdots g ({A'} ^{-1})g(Z) = Z A \cdots \varphi ^n (A) \varphi ^{n} (g(A^{-1})) \cdots g(A^{-1}).\label{finalcheck}
\end{equation}
But as $\varphi \circ h = h \circ \varphi$, we have $A' \varphi (Z) =  ZA$, or equivalently,
${A'} ^{-1} Z = \varphi (Z) A^{-1}$, and equation (\ref{finalcheck}) follows easily.
\end{proof}

\subsection{An equivalence of categories}



In this subsection, we prove Theorem \ref{thm-F-BT}.  Let us first recall the setup and some notation.
For $\M \in {\rm Mod} ^{\varphi, 1}_{\gs_F}$, put  $\CM : = \gO \otimes_{\gs _F} \M$ and define a decreasing filtration on $\varphi ^*\CM$ as in \eqref{eqn-definefil}.
Since $\M$ has height 1, we have $\Fil ^i \varphi^* \CM = E(u)^{i -1} \Fil ^1\varphi^* \CM$ for $i \geq 2$.
Recall that $M : = \varphi ^*\M/ u \varphi ^*\M$ and
let us put $D:=D(\M):= \xi_\alpha (M[1/p])$, which is naturally a $\varphi$-stable $F_0$-subspace
of $\gO_\alpha \otimes_{\gs_F} \varphi ^* \M$ via the unique isomorphism $\xi_\alpha$ constructed in Lemma \ref{lem-phisection 2}.
Since ${\xi}_\alpha : \gO_\alpha \otimes _{F_0} D \to \gO_\alpha  \otimes \varphi ^*\CM  $ is an isomorphism, we may identify $\varphi ^* \CM/ E(u) \varphi^*\CM$ with $D_{F_0, K} = K \otimes_{F_0} D$, and we write
$\psi _\pi : \varphi^*\CM \to D_{F_0, K}$ for the natural projection.
We define $\Fil ^i D_{F_0,  K} : = \psi_\pi (\Fil ^i \varphi ^*\CM)$, and note that since
$\Fil ^2 \varphi ^*\CM \subset E(u) \varphi ^* \CM$, we  have $\Fil ^2 D_{F_0, K}= 0$.
In this way we obtain from $\M$ a filtered $\varphi$-module $D= D(\M)$ in $\t{MF}^{\varphi}_{F_0, K}$.

Suppose that $\M'= \gs_F \cdot \fe $  is  a rank-1 Kisin module with $\gs_F$-basis $\fe$.
Then we have $\varphi (\fe)= \gamma E(u) ^m \fe$ with $\gamma \in \gs ^\times_F$  a unit
thanks to Example \ref{ex-dim1}, and we call $m$ the \emph{minimal height} of $\M'$.
\begin{lemma}\label{lem-fil-revisit}
With notation as above,
\begin{enumerate}
	\item  The natural injection
	\begin{equation*}
		\xymatrix{
			{\gO_\alpha  \otimes _\gO \varphi^* \CM} \ar[r]^-{\xi'_\alpha} &
			{\gO _\alpha \otimes _{F_0} D} \ar@{^{(}->}[r] & {\wgs_ 0 \otimes D_{F_0, K}}
			}
	\end{equation*}
	is compatible with filtrations, where $\xi'_\alpha = (\xi_\alpha)^{-1}$.
	
\item Suppose $\M$ has rank $d$. Then the minimal height of $\wedge ^d \M$ is $\dim_{K_0}\Fil ^1 D_{F_0, K}$.

\end{enumerate}
\end{lemma}
\begin{proof} Since $\Fil ^i \varphi ^*\CM = E(u)^{i -1}\Fil ^1 \varphi ^*\CM$ for $i \geq 2$,
to prove (1) it suffices to check the given injection is compatible with $\Fil^1$.
As $E(u)$ is a generator of $\Fil ^1 \wgs_0$,
such compatibility is equivalent to the condition
that $ x \in \Fil ^1 \varphi ^* \CM $ if and only if $\psi_{\pi} (x) \in \Fil ^1 D_{F_0,  K}$.
But this is clear as $E(u) \varphi ^* \CM \subset \Fil ^1 \varphi^* \CM$.

We now prove (2).
Fix an $\gs_F$-basis $\fe _1, \dots , \fe _d$ of $\M$ and let $A\in \Md (\gs _F)$
be the corresponding matrix of Frobenius.
Since $\M$ has height $1$, there exists a matrix $B\in \Md(\gs_F)$ with $AB = E(u)I_d$.
Defining $e_i =1 \otimes \fe _i \in \varphi ^* \CM$, we easily check that $\{e_i\} $
is an $\gO$-basis of $\varphi ^*\CM$ with
$(\alpha _1, \dots, \alpha _d):= (e_1, \dots , e_d)B $ an $\gO$-basis of $\Fil ^1 \varphi ^*\CM$.

 Now the inclusion $\varphi ^*\CM / \Fil ^1 \varphi ^* \CM \subset \varphi ^*\CM / E(u) \varphi ^* \CM = D_{F_0, K}$ realizes $\varphi ^*\CM / \Fil ^1 \varphi ^* \CM$ as a $K$-subspace of $D_{F_0,K}$, so there
 exists a basis $f_1, \dots , f_d$ of $\varphi ^* \CM$ with the property that
 $f_1 , \dots f_{s}, E(u) f_{s+1}, \dots , E(u) f_d$ generates $\Fil ^1 \varphi^* \CM$.
 Since $\Fil ^1 D_{F_0, K} = \psi_\pi (\Fil ^1 \varphi ^* \CM  )$ we have $\dim _{K}\Fil ^1 D_{F_0, K}= s$.
 On the other hand, since $\alpha _1, \dots , \alpha_d$ also generates $\Fil ^1 \varphi ^*\CM$,  there exist invertible matrices $X, Y \in \GL_d (\gO)$ with
 $$B = X \Lambda Y\quad\text{for}\quad\Lambda= \diag(1, \dots ,1 , E(u),\dots ,  E(u ))$$
 the diagonal matrix with $s$ many 1's and $d-s$ many $E(u)$'s along the diagonal.
 Thus, $\det B = E(u)^{d-s} \gamma $ for $\gamma \in \gO ^\times$ a unit and $AB = E(u) I _d $
 implies that $\det (A) = E(u) ^s \gamma ^{-1} $.
 It follows that the minimal height of $\wedge ^d \M$ is $s=\dim_K \Fil ^1 D_{F_0, K}$,
 as desired.
\end{proof}

Recall that  we have defined $V_{\gs} (\M):  = F \otimes_{\O_{F} } T_{\gs} (\M )$.

\begin{prop}\label{prop-crystalline}
	With notation as above, we have $V_{\gs} (\M) \simeq V_{\cris, F} (D(\M))$ as $F[G]$-modules.
	In particular, $V_{\gs} (\M)$ is crystalline with Hodge-Tate weights in $\{0, 1\}$.
\end{prop}

\begin{proof}
The proof of Proposition \ref{prop-Ginftycompatible} carries over {\em mutatis mutandis}
to show that there exists a natural injection of $\O_F[G_{\ul{\pi}}]$-modules
\begin{eqnarray*}
\iota :  T _{\gs}(\M) \inj \Hom_{\gO, \varphi, \Fil} (\varphi ^*\CM, B^+_{\cris, F} ) \inj \Hom_{\gO_\alpha, \varphi, \Fil} (\gO_\alpha \otimes _{\gO} \varphi ^*\CM , B^+_{\cris, F}) \\
\simeq \Hom _{\gO_\alpha, \varphi, \Fil} (\gO_\alpha \otimes_{F_0} D, B^+_{\cris, F} ) \simeq V_{\cris, F} (D),
\end{eqnarray*}
where instead of using Lemma \ref{lem-fil}, we must appeal to
Lemma \ref{lem-phisection 2} and Lemma \ref{lem-fil-revisit} (note that
{\em a priori} we know neither that $\CM (D) \simeq \CM $ nor that $D$ is admissible).
Since $ \dim_{F_0} (D) = \t{rank}_{\gs_F} \M $ and $\iota$ is injective, we conclude that $D$ is admissible. In particular, $V_{\cris, F }(D)$ is crystalline with Hodge-Tate weights in $\{0, 1\}$.

It remains to show that $\iota$ is compatible with the given actions
of $G=G_K$.
By construction, the $G$-action on $T_{\gs}(\M)$ is induced from the identification
$$T_{\gs}(\M) \simeq \Hom_{W(R)_F, \varphi} (W(R)_F  \otimes_{\gs_F}\M, W(R)_F) $$
of (\ref{Tisom}) with $G$-action on the right side that of Lemma \ref{lem-extend Ginfty action}.
We clearly have an isomorphism of $\O_{F}[G]$-modules
$$\Hom_{W(R)_F, \varphi} (W(R)_F  \otimes_{\gs_F}\M, W(R)_F) \simeq \Hom_{W(R)_F, \varphi, \Fil} ( W(R)_F \otimes_ {\gs_F}\varphi ^* \M, W(R)_F ),$$
with the right side an $\O_{F}$-lattice in $\Hom_{B^+_{\cris, F}, \varphi, \Fil} (B^+_{\cris, F} \otimes_{F_0} D, B^+_{\cris, F} )$. Thus, to prove that $\iota$ is $G$-equivariant, we must show that the
$G$-action on $B^+_{\cris, F} \otimes_{\gs_F} \varphi^* \M$ deduced from Lemma \ref{lem-extend Ginfty action}
agrees with the $G$-action on $B^+_{\cris, F} \otimes_{F_0} D$ via the map
$$
\xymatrix{
	{B^+_{\cris, F} \otimes_{F_0} D} \ar[r]^-{\simeq} & {B^+_{\cris, F} \otimes_{\gs_F} \varphi^* \M}
	}
$$
deduced from (\ref{HowDisSub}) (which is an isomorphism thanks
to Lemma \ref{lem-phisection 2}); here, $G$ acts trivially on $D$.
Equivalently, we must show that the $G$-action on $B^+_{\cris, F} \otimes_{\gs_F} \varphi^* \M$
provided by Lemma \ref{lem-extend Ginfty action} restricts to the \emph{trivial} action on $D(\M)$,
viewed as a subspace of this tensor product again via (\ref{HowDisSub}).

As in the proofs of Lemma \ref{lem-phisection 2} and Lemma \ref{lem-extend Ginfty action},
let $\fe_1, \dots, \fe_d$ be an $\gs_F$-basis of $\M$ and put $\{e_i := 1 \otimes \fe_i\}$, which
is then an $\gs_F$-basis of $\varphi ^*\M$.
The proof of Lemma \ref{lem-phisection 2} shows that $(f_1, \dots , f_d) := (e_1, \dots ,  e_d) Y$ is a basis of
$D(\M)$ for $$Y= \lim_{n \to \infty }  \varphi (A) \cdots \varphi ^n (A) \varphi ^n (A^{-1}_0) \cdots \varphi (A_0^{-1}). $$
For any $g \in G$, by the proof of Lemma \ref{lem-extend Ginfty action} we have
$g (e_ 1 , \dots, e_d) = (e_1, \dots ,e _d)\varphi (X_g)$ with
$$\varphi (X_g)  = \lim_{n \to \infty }  \varphi (A) \cdots \varphi ^n (A) \varphi ^n (g(A^{-1}) ) \cdots \varphi (g(A^{-1})).  $$
Thus,
$
g(f_1, \dots, f_d) =  (e_1, \dots, e_d) \varphi (X_g)g(Y)
=  (e_1, \dots , e_d) \lim\limits_{n\to  \infty }  \varphi (X_n) g(Y_n)
$
where
\begin{eqnarray*}
\varphi (X_n) g(Y_n ) & =  & \left ( \varphi (A) \cdots \varphi ^n (A) \varphi ^n (g(A^{-1}) ) \cdots \varphi (g(A^{-1})) \right ) \\ & & \left ( \varphi (g(A)) \cdots \varphi ^n (g(A)) \varphi ^n (g(A^{-1}_0) ) \cdots \varphi (g(A_0^{-1})) \right ) \\ &= & Y_n
\end{eqnarray*}
In other words, $g(f_1, \dots, f_d) = (f_1, \dots , f_d)$, which completes the proof.
\end{proof}

\begin{proof}[Proof of Theorem $\ref{thm-F-BT}$]

Thanks to Proposition \ref{prop-crystalline} and Corollary \ref{co-G-action}, we have
a contravariant functor $T_{\gs} : \t{Mod}^{\varphi,1}_{\gs_F} \rightarrow \t{Rep}^{F\t{-cris},1}_{\O_F} (G)$,
which it remains to prove is fully faithful and essentially surjective.

For full-faithfulness, let us suppose given a map
$h:T_{\gs}(\M) \to T_{\gs} (\M')$ of $\O_F[G]$-modules. Restricting to $G_{\ul{\pi}}$ gives
a map $h\big|_{G_{\ul{\pi}}} :T_{\gs}(\M)|_{G_{\ul{\pi}}} \to T_{\gs _F} (\M')|_{G_{\ul{\pi}}}$,
and by Corollary \ref{co-TM} we obtain a morphism
$\gf : \O_\E \otimes_{\gs_F} \M ' \to \O_E \otimes_{\gs_F} \M $ with $T_{\gs} (\gf) = h|_{G_{\ul{\pi}}}$.
It then suffices to show that $\gf(\M' )\subset \M$.
Arguing as in the proof of Proposition \ref{prop-fullfaith}, it suffices to check that if
$\M \subset \M ' \subset \O_\E \otimes_{\gs_F} \M$ then $\M = \M'$. Applying
$\wedge^d$, we then easily reduce to proving that
$\wedge ^d \M $ and $\wedge ^d \M '$ have the same minimal height.
By our reductions we now have $T_{\gs} (\M) \simeq T_{\gs} (\M ')$ as $\O_F[G]$-modules thanks to Corollary \ref{co-G-action},
so by Proposition \ref{prop-crystalline} we have $D(\M) \simeq D(\M')$ as filtered $\varphi$-modules.
In particular, $\Fil ^1 D(\M)_{F_0, K} \simeq \Fil ^1 D(\M')_{F_0, K}$ and the minimal heights of
$\wedge^d \M$ and $\wedge^d\M'$
are the same by Lemma \ref{lem-fil-revisit} (2).  Thus, $T_{\gs}$ is fully faithful.

We now show that $T_{\gs}$ is essentially surjective. Fix $T \in \t{Rep}_{\O_F} ^{F\t{-\cris}, 1 } (G)$,
put  $V := F \otimes _{\O_F} T$ and let $D:= D_{\cris, F } (V)$ be the corresponding filtered $\varphi$-module.  By Corollary \ref{co-Ginftylattice},  there exists $\M \in \t{Mod}_{\gs_F} ^{\varphi, 1}$ with $\CM(D) \simeq \gO \otimes_{\gs_F}\M$ and
$\iota: T_{\gs} (\M) \xrightarrow{\sim} T|_{G_{\ul{\pi}}}$.
It suffices to show that $\iota$ is compatible with the actions of $G$ on source and target, with
the $G$-action on the source provided by Corollary \ref{co-G-action}.
Using Proposition \ref{prop-crystalline}, we obtain an isomorphism
of $F[G]$-modules $\iota' : V_{\gs} (\M) \simeq V_{\cris, F} (D(\M))$,
which one verifies is compatible with the identification $\iota$.
It therefore remains to check that $ D(\M) \simeq D$ as filtered $\varphi$-modules.
Thanks to Lemma \ref{lem-phisection} and Lemma \ref{lem-phisection 2},
we can identify each of $D$ and $D(\M)$ as the image of the unique $\varphi$-equivariant section to projection
$\varphi ^* \CM(D) \twoheadrightarrow \varphi ^* \CM(D)/ u \varphi^* \CM(D)$, which gives $D\simeq D(\M)$
as $\varphi$-modules.
Thus, it remains to prove that $\Fil ^i D_K = \Fil ^i  D(\M)_K$ for all $i>0$,
or equivalently that $ \Fil ^1 D_{F_0, K} = \Fil ^1 D(\M)_{F_0, K} $.
Thanks to Corollary \ref{co-comp fil},
the projection $ \psi_\pi : \varphi^*\CM (D) \twoheadrightarrow \varphi^* \CM(D)/ E(u) \simeq D_{F_0, K} $ is compatible with filtrations, and one checks using the very definition of $\Fil^1 \varphi^*\CM(D)$
that $x \in \Fil ^1 \varphi ^* \CM (D) $ if and only if $ f_\pi (x) \in \Fil ^1 D_{F_0, K}$.
Thus, $\Fil ^1 D_{F_0, K} = f_\pi (\varphi ^* \CM(D)) = \Fil ^1 D(\M)_{F_0, K}$, as desired.
\end{proof}

\begin{remark}\label{rem-final}
	In the classical situation, let $S$ be the $p$-adic completion of the divided-power envelope
	of the surjection $W(k)[\![u]\!]\twoheadrightarrow \O_K$ sending $u$ to $\pi$.
	If $\M$ is the Kisin module attached to a Barsotti--Tate group $H$ over $\O_K$,
	then one can show (\cite[\S2.2.3]{kisin2}) that there is a functorial isomorphism
	$\varphi^*\M\simeq \mathbb D (H)_S$ where $\mathbb D (H)$ is the Dieudonn\'e crystal attached to $H$,
	which gives a {\em geometric interpretation} of $\M$ in terms of the crystalline cohomology of $H$.
	It is natural to ask for such an interpretation in the general case, for arbitrary $F$ and $f(u)$
	as in the introduction of this paper.  If $F/\Q_p$ is unramified, then this interpretation is provided
	by \cite{CaisLau}.  However, for $F$ ramified over $\Q_p$, things are more subtle as
	it is necessary to use the $\O$-divided powers of Faltings \cite{faO}; this case is work in progress
	of A. Henniges. In any case, we conjecture that one has a natural isomorphism
	$A_{\cris,F} \otimes_{\gs _F} \varphi ^*\M \simeq \mathbb D (H)_{A_\cris}$,
	and expect to be able to prove this conjecture using the ideas of \S \ref{subsec-6.3}.
\end{remark}

\section{Further Questions}

As Theorem \ref{thm-intro-1} and Theorem \ref{thm-intro-2} provide the foundations of the theory of Kisin modules and its variants ({\em e.g.}~the theory of $(\varphi, \hat G)$-modules as in \cite{liu4}),
it is natural to ask to what extent we can extend these theories to accommodate general $F$ and $f(u)$.
In this section, we list some questions that are natural next steps to consider in furthering
the general theory we have laid out in this paper.

\subsection{The case $q= p^s$}
Recall the setup of the introduction: $F/\Q_p$ is an arbitrary finite extension
with uniformizer $\varpi$ and residue field $k_F$ of cardinality $q=p^s$, and $f(u)\in \O_F[\![u]\!]$
is any power series $f(u)=a_1u+\cdots$ with $f(u)\equiv u^q\bmod \varpi$.
We allow $K$ to be any finite extension of $F$ with uniformizer $\pi=\pi_0$ amd residue field $k\supseteq k_F$,
and consider the Frobenius-iterate extension $K_{\ul{\pi}}$ formed by adjoing to $K$
an $f$-compatible system $\{\pi_n\}_n$ in $\ol{K}$ with, $f(\pi_{n})=\pi_{n-1}$.
Such extensions and their associated norm fields are considered in \cite{CaisDavis} and \cite{CaisDavisLubin}.
In this paper, we have restricted ourselves to $q=p$, or what is the same, that $F/\Q_p$ is totally ramified.
Certainly this restriction is unnecessary, and we are confident that the results of this paper can be adapted
to the general case of arbitrary $F$ with minor modifications. In particular,
in this general case, for any $W(k)$-algebra $A$ we set $A_F := A \otimes_{W(k_F)} \O_F$,
and we equip $\gs_F$ with the ``$q$-power Frobenius'' $\varphi_q$ which acts on $F$-trivially, acts on $W(k)$ via $\varphi_{W(k)}^s$ and sends $u$ to $f(u)$. We write $F_0:=K_0F$ and again denote by $E(u)\in \O_{F_0}[u]$
the minimal polynomial of $\pi$ over $F_0$.
Then our theory should be able to be adapted to functorially associate Kisin modules of finite $E$-height
to $\O_F$-lattices in $F$-crystalline $G$-representations.
We note that such a theory is already known in the ``Lubin--Tate'' case that $v(a_1)=v(\varpi)$
and $K\subseteq F_{\ul{\varpi}}$
thanks to the work of Kisin and Ren \cite{Kisin-Ren}, but that there are many details in our general setup that
still need to be checked.

\subsection{Semi-stable representations and Breuil theory}
In the classical situation, Theorem \ref{thm-maintech} includes semi-stable representations.  This fact is one of the key inputs for Breuil's classification of lattices in semistable representations via strongly divisible lattices over $S$ (see \cite{liu3}).   It is therefore natural to ask if   Theorem \ref{thm-maintech} remains valid for semi-stable representations and general $f(u)$. This appears to be a rather nontrivial question, as
the case of semi-stable representations requires a monodromy operator. But for general $F$ and $f$,
we do not even know how to define a reasonable monodromy operator over $\gs_F$
({\em i.e.}, one satisfying $N \varphi = p \varphi N $ as in the classical situation).
New ideas are needed for this direction.

\subsection{Comparison between different choices of $f(u)$}\label{subsec-6.3}
For a fixed $F$-crystalline representation $V$ of $G$ and a fixed uniformizer $\pi \in K$, we may select different $f(u)$. It is then natural to ask for the relationship between the associated Kisin modules attached to $V$
and $f(u)$, as $f$ varies.
Motivated  by \cite{liu-comparison},  we conjecture that all such Kisin modules become isomorphic after base change to $W(R)_F$. Note that if true, this result provides a proof of the conjecture mentioned in Remark \ref{rem-final}, because we know that $A_{\cris,F} \otimes_{\gs _F} \varphi ^*\M \simeq \mathbb D (H) (A_\cris)$ in the classical situation. To prove such comparison results, the key point is to generalize \cite[Theorem 3.2.2]{liu2} to allow general $f(u)$. This is likely relatively straightforward, as we have recovered many results of  \cite{liu2} in \S\ref{sec-3} already.

\subsection{Torsion theory} A major advantage of the theory of Kisin modules is that
it provides a powerful set of tools for dealing with torsion representations.
It is therefore natural to try and rebuild the torsion theory in our general situation, and we hope that such a theory will have some striking applications, for example, to the computation of the reduction of potentially crystalline representations as discussed in the introduction.
One obvious initial goal is to establish the equivalence between torsion Kisin modules of height 1 and finite flat group schemes over $\O_K$; this would be achievable quickly once we know the truth of the conjecture formulated in Remark \ref{rem-final}.

\bibliographystyle{amsalpha}
\bibliography{BIBLIO1}


\end{document}